\documentclass[a4paper,11pt]{amsart}
\usepackage{amsmath,amsthm,amssymb,amsfonts,enumerate,color,esint}
\usepackage[pdftex]{graphicx}
\usepackage{float}

\usepackage[colorlinks=true, allcolors=blue]{hyperref}
\usepackage{amsrefs}
\newcommand{\la}{\langle}
\newcommand{\ra}{\rangle}
\newcommand{\N}{\mathbb{N}}

\newcommand{\R}{\mathbb{R}}

\newcommand{\mQ}{\mathcal{Q}}

\newcommand{\norm}[1]{\left\Vert#1\right\Vert}
\newcommand{\abs}[1]{\left\vert#1\right\vert}
\newcommand{\w}{\omega}
\def\({\left(}
\def\){\right)}

\newcommand{\loc}{\operatorname{loc}}

\newcommand{\trace}{\operatorname{tr}}

\newcommand{\Dom}{\operatorname{Dom}}

\newtheorem{thm}{Theorem}[section]
\newtheorem{prop}[thm]{Proposition}
\newtheorem{cor}[thm]{Corollary}
\newtheorem{lem}[thm]{Lemma}

\theoremstyle{definition}
\newtheorem{defn}[thm]{Definition}
\newtheorem{rem}[thm]{Remark}

\newtheorem{notation}[thm]{Notation}

\numberwithin{equation}{section}

\allowdisplaybreaks

\author[P. R. Stinga]{\href{http://pabloraulstinga.github.io/}{Pablo Ra\'ul Stinga} }

\address{Department of Mathematics\\
Iowa State University\\
396 Carver Hall, Ames\\
IA 50011, United States of America}
\email{stinga@iastate.edu}

\author[M. Vaughan]{\href{https://maryvaughan.github.io/}{Mary Vaughan}}

\address{Department of Mathematics\\
The University of Texas at Austin\\
2515 Speedway, Austin\\
TX 78712, United States of America}
\email{maryv@iastate.edu}

\thanks{Research supported by Simons Foundation grant 580911}



\begin{document}
%%%%%%%%%%%%%%%%%%%%%%%%%%%%%%%%%%%%%%%%%%%%%%%%%%%%%%

%%%%%%%%%%%%%%%%%%%%%%%%%%%%%%%%%%%%%%%%%%%%%%%%%%%%%%
\title[Fractional elliptic equations]{Fractional elliptic equations in nondivergence form: definition,
applications and Harnack inequality}
%%%%%%%%%%%%%%%%%%%%%%%%%%%%%%%%%%%%%%%%%%%%%%%%%%%%%%

%%%%%%%%%%%%%%%%%%%%%%%%%%%%%%%%%%%%%%%%%%%%%%%%%%%%%%
\begin{abstract}
We define the fractional powers $L^s=(-a^{ij}(x)\partial_{ij})^s$, $0 < s < 1$, of nondivergence form elliptic operators
$L=-a^{ij}(x)\partial_{ij}$  in bounded domains $\Omega\subset\R^n$, under minimal regularity assumptions on the coefficients
$a^{ij}(x)$ and on the boundary $\partial\Omega$. We show
that these fractional operators appear in several applications such as
fractional Monge--Amp\`ere equations, elasticity, and finance. The solution $u$ to the nonlocal Poisson problem
$$\begin{cases}
(-a^{ij}(x) \partial_{ij})^su = f&\hbox{in}~\Omega\\
u=0&\hbox{on}~\partial\Omega
\end{cases}$$
is characterized by a local degenerate/singular extension problem.
We develop the method of sliding paraboloids in the Monge--Amp\`ere geometry and
prove the interior Harnack inequality and H\"older estimates for solutions to 
the extension problem when the coefficients $a^{ij}(x)$ are bounded, measurable functions.
This in turn implies the interior Harnack inequality and H\"older estimates for solutions $u$ to the fractional problem. 
\end{abstract}
%%%%%%%%%%%%%%%%%%%%%%%%%%%%%%%%%%%%%%%%%%%%%%%%%%%%%%

\maketitle

%%%%%%%%%%%%%%%%%%%%%%%%
\section{Introduction}
%%%%%%%%%%%%%%%%%%%%%%%%

In this paper, we inaugurate the analysis of equations driven by fractional powers of nondivergence
form uniformly elliptic operators
\begin{equation}\label{eq:Ls}
L^s=(-a^{ij}(x)\partial_{ij})^s\quad\hbox{in}~\Omega\qquad\hbox{for}~0<s<1,
\end{equation}
under minimal regularity assumptions on the coefficients $a^{ij}(x)$ and the boundary of the domain $\Omega\subset\R^n$, $n\geq1$.
We show in Section \ref{sec:applications} that
fractional power operators as in \eqref{eq:Ls} in such minimal regularity regime arise in applications to 
fractional Monge--Amp\`ere equations, elasticity, and finance, among others.

The very first difficulty we need to overcome when considering \eqref{eq:Ls}
is that of giving a meaningful definition of the fractional power operator $L^s$ when
\begin{equation}\label{eq:L}
L=-a^{ij}(x)\partial_{ij}\equiv-\sum_{i,j=1}^na^{ij}(x)\partial_{ij}\qquad x\in\Omega
\end{equation}
is an elliptic operator in \emph{nondivergence} form with \emph{nonsmooth} coefficients.
As in other well-known cases, this is not immediately obvious. 
For example, the fractional Laplacian $(-\Delta)^s$ in $\R^n$ can be defined using the Fourier transform as
$\widehat{(-\Delta)^su}=|\xi|^{2s}\widehat{u}$. However, the nondivergence form operator \eqref{eq:L}
has nonsmooth coefficients in a bounded domain $\Omega$, so the Fourier transform is not 
the most convenient tool.  Indeed, \eqref{eq:L} is not translation invariant and not defined in the whole space. 
If $-\Delta_D$ denotes the Laplacian in a bounded domain $\Omega$ subject to
homogeneous Dirichlet boundary conditions on $\partial\Omega$,
then $(-\Delta_D)^s$ is naturally defined in a spectral way using the basis of eigenfunctions
and the corresponding eigenvalues of $-\Delta_D$ in the Sobolev space $H^1_0(\Omega)$.
In contrast, there is no natural Hilbert space structure for nondivergence form operators as in \eqref{eq:L}.
The spectral method is also used to 
define fractional powers of \emph{divergence form} elliptic operators $(-\partial_i(a^{ij}(x)\partial_j))^s$, see \cite{Caffarelli-Stinga}.
Nevertheless, our operator \eqref{eq:L} has nonsmooth coefficients so
it cannot be written in divergence form. We further remark that these definitions, though adequate from 
the operator theory point of view, do not immediately give explicit pointwise, nonlocal formulas.

Our idea to define \eqref{eq:Ls} is to apply the \emph{method of semigroups}.
The main ingredient in this approach is the semigroup $\{e^{-tL}\}_{t\geq0}$ generated by $L$. With this, we define
\begin{equation}\label{eq:semigroupformula}
L^su(x)=\frac{1}{\Gamma(-s)}\int_0^\infty\big(e^{-tL}u(x)-u(x)\big)\,\frac{dt}{t^{1+s}}
\end{equation}
for $0<s<1$, $x\in\Omega$, where $\Gamma$ is the Gamma function.
Using the semigroup, we can also give formulas for the solution $u$ to the Poisson problem $L^su=f$ as $u=L^{-s}f$ and
for the solution $U$ to local extension problems.
Moreover, if $L$ has a heat kernel, then one can derive explicit pointwise expressions for $L^su(x)$, $L^{-s}f(x)$ and $U(x,z)$.
These results are presented in Section \ref{sec:fractionalpowers}.
For details about the semigroup method applied to the fractional Laplacian in the whole space
and to other different contexts, see \cite{Stinga} and the references therein.

We then consider the following fractional elliptic equation in nondivergence form:
\begin{equation}\label{eq:Poisson}
 \begin{cases}
(-a^{ij}(x) \partial_{ij})^s u = f & \hbox{in}~\Omega\\
u = 0 & \hbox{on}~\partial \Omega.
\end{cases}
\end{equation}
Here $\Omega \subset \R^n$, $n \geq 1$, is a bounded domain satisfying
the uniform exterior cone condition. The coefficients $a^{ij}(x):\Omega \to \R$ are symmetric $a^{ij}(x)=a^{ji}(x)$, $i,j=1,\ldots,n$,
$a^{ij}(x) \in C(\Omega)\cap L^{\infty}(\Omega)$, and uniformly elliptic, meaning that there exist constants $0 < \lambda \leq \Lambda $ such that
\begin{equation}\label{eq:ellipticity}
\lambda \abs{\xi}^2 \leq a^{ij}(x)\xi_i \xi_j \leq \Lambda \abs{\xi}^2 \quad \hbox{for all}~\xi \in \R^n~\hbox{and a.e.}~x\in \Omega.
\end{equation}
Under these conditions, the operator $L = -a^{ij}(x) \partial_{ij}$ generates a uniformly bounded $C_0$-semigroup
with exponential decay. Therefore, $L^su$ is well-defined by means of \eqref{eq:semigroupformula}.
See Section \ref{sec:fractionalpowers} for these notions and the necessary notation.

The main regularity result for \eqref{eq:Poisson} in this paper is the following interior Harnack inequality and H\"older regularity estimate.

\begin{thm}\label{thm:harnack Ls}
Assume that $\Omega\subset \R^n$ is a bounded domain that satisfies the uniform exterior cone condition,
$a^{ij}(x) \in C^{\alpha}(\Omega) \cap L^{\infty}(\Omega)$, for some $0 < \alpha <1$, are symmetric, satisfy \eqref{eq:ellipticity},
and $f \in C_0(\Omega)$.
There exist positive constants $C_H = C_H(n,\lambda,\Lambda, s)>1$, $\kappa = \kappa(n,s)<1$, and $\hat{K} = \hat{K}(n,s)>1$ such
that for every ball $B_{\hat{K}R} = B_{\hat{K}R}(x_0)$ satisfying 
$B_{\hat{K}R} \subset\subset \Omega$ 
and every $u \in \Dom(L)$, $u \geq 0$ in $\Omega$, solution to
\begin{equation}\label{eq:eqLsinB}
(-a^{ij}(x) \partial_{ij})^su =f  \quad \hbox{in}~B_{\hat{K}R},
\end{equation}
we have that
\begin{equation}\label{eq:Ls Harnack3}
\sup_{B_{\kappa R}} u \leq C_H \( \inf_{B_{\kappa R}} u + \norm{f}_{L^{\infty}(B_{\hat{K}R})}R^{2s}\).
\end{equation}

Furthermore,
there exist positive constants $\alpha_0 = \alpha_0(n,\lambda,\Lambda,s)<1$ and $\hat{C}= \hat{C}(n,\lambda,\Lambda,s)>0$ such that for any $u \in \Dom(L)$ satisfying \eqref{eq:eqLsinB}, we have that, for every $x \in B_{\hat{K}R}(x_0)$,
\begin{equation}\label{eq:Ls Holder}
\abs{u(x_0) - u(x)} \leq \frac{\hat{C}}{(\hat{K}R)^{\alpha_0}}\abs{x_0-x}^{\alpha_0} \(\sup_{\Omega} \abs{u}+\norm{f}_{L^{\infty}(B_{\hat{K}R})}R^{2s}\).
\end{equation}
\end{thm}

We mention that Grubb \cite{Grubb, Grubb2} and Seeley \cite{Seeley} studied properties of fractional powers of
nondivergence form elliptic operators with smooth coefficients and smooth domains from the operator theory and
pseudo-differential operators points of view. 
Gradient estimates for fractional powers of constant coefficients, nondivergence form operators in $\R^n$ were studied in \cite[Remark 1.10]{Valdinoci}. 
In particular, none of their results 
include the Harnack inequality
and H\"older estimate in Theorem \ref{thm:harnack Ls}.

Our proof of Theorem \ref{thm:harnack Ls} is based on the extension problem characterization of fractional power operators 
in Banach spaces given by the method of semigroups in \cite{Gale} (see \cite{Stinga-Torrea} for the case of Hilbert spaces).
In our particular case, the extension result of \cite{Gale} allows us to rewrite the nonlocal equation \eqref{eq:Poisson}
in an equivalent way as a local PDE problem.

\begin{thm}[{Particular case of  \cite{Gale}}]\label{thm:Calpha extension}
Assume that the bounded domain $\Omega \subset \R^n$ satisfies the uniform exterior cone condition and that 
$a^{ij}(x) \in C({\Omega}) \cap L^{\infty}(\Omega)$ are symmetric and satisfy \eqref{eq:ellipticity}. 
If $u \in \Dom(L)$, then a solution $U\in C^{\infty}((0,\infty);\Dom(L)) \cap C([0,\infty);C_0(\Omega))$
to the extension problem
\begin{equation} \label{eq:extension intro}
\begin{cases}
a^{ij}(x) \partial_{ij}U + z^{2-\frac{1}{s}}\partial_{zz}U= 0 & \hbox{in}~\Omega \times \{z>0\}\\
U(x,0) = u(x) & \hbox{on}~\Omega \times \{z=0\}\\
U= 0 & \hbox{on}~\partial\Omega \times \{z\geq0\}
\end{cases}
\end{equation}
is given by
\begin{equation}\label{eq:U-defn}
U(x,z) = \frac{s^{2s}z}{\Gamma(s)}\int_0^{\infty} e^{-s^2z^{1/s}/t} e^{-tL}u(x) \,\frac{dt}{t^{1+s}}
\end{equation}
and satisfies 
\[
\norm{U(\cdot,z)}_{L^{\infty}(\Omega)} \leq M \norm{u}_{L^{\infty}(\Omega)} \quad \hbox{for some}~M>0.
\]
Furthermore, $U_{z+} \in C([0,\infty);C_0(\Omega))$ and
\[
-\partial_{z+} U(x,0) = d_sL^su(x) \in C_0(\Omega)
\]
where $d_s=\frac{s^{2s}\Gamma(1-s)}{\Gamma(1+s)}>0$ and
\[
\partial_{z+}U(x,0) = \lim_{z \to 0^+} \frac{U(x,z) - U(x,0)}{z}\quad\hbox{for all}~x\in\Omega.
\]
If, in addition, $a^{ij}(x)\in C^{\alpha}(\Omega)$ for some $0 < \alpha<1$, then the solution $U$ in \eqref{eq:U-defn} is the unique classical solution
$U \in C^2(\Omega \times (0,\infty))\cap C(\overline{\Omega} \times [0,\infty))$ such that $\lim_{z\to\infty}\|U(\cdot,z)\|_{L^\infty(\Omega)}=0$.
\end{thm}

Hence, to prove Theorem \ref{thm:harnack Ls},
we will show interior Harnack inequalities and H\"older estimates for solutions $U$ to the local,
degenerate/singular elliptic equation in \eqref{eq:extension intro}
subject to the Neumann boundary condition $-\partial_{z+} U(x,0) =f(x)$
up to $\{z=0\}$, and then take the trace at $\{z=0\}$. 
Towards this end, we define the even reflection of $U$ in the variable $z$ by $\tilde{U}(x,z) = U(x,|z|)$,
for $x\in\Omega$, $z\in\R$. For convenience, we continue to use the notation $U$ instead of $\tilde{U}$ and notice that $U$,
being symmetric across $\{z=0\}$, satisfies the equation
\begin{equation}\label{eq:Utilde}
a^{ij}(x) \partial_{ij}U + |z|^{2-\frac{1}{s}}\partial_{zz}U = 0\qquad\hbox{in}~\Omega \times\{z\neq0\}.
\end{equation}
Furthermore, if $f(x)\neq0$ then $\partial_{z}U$ is discontinuous across $(x,0)$.
Since $0<s<1$, the equation \eqref{eq:Utilde} either degenerates or blows up at $z=0$, unless $s=1/2$.

It turns out that we can recast \eqref{eq:Utilde} as an equation comparable to a linearized Monge--Amp\`ere equation.
Recall that the Monge--Amp\`ere equation for a convex function $\psi$ is given by $\det D^2\psi = G$.
By taking the directional derivative $\partial_e$ in a unit direction $e$ to the equation
and defining $v = \partial_e\psi$ and $g = \partial_eG$, we find that $v$
satisfies the linearized Monge--Amp\`ere equation
\begin{equation}\label{eq:linMA}
\trace(A_{\psi}(x) D^2 v) = g.
\end{equation}
Here, $A_\psi(x) = \det(D^2\psi(x)) (D^2\psi(x))^{-1}$ is the matrix of cofactors of $D^2\psi(x)$. 
Notice that \eqref{eq:linMA} is a linear equation in nondivergence form that is elliptic as soon as $D^2\psi >0$ and $G>0$. 
However, it is not uniformly elliptic in general since the eigenvalues of $A_{\psi}(x)$ are not \textit{a priori} controlled. 

For our degenerate equation \eqref{eq:Utilde}, we consider the strictly convex function $\Phi = \Phi(x,z): \R^{n+1} \to \R$ given by
\[
\Phi(x,z) = \frac{1}{2} \abs{x}^2 + \frac{s^2}{1-s}\abs{z}^{\frac{1}{s}},\qquad 0<s<1.
\]
Then $\Phi$ is in $C^1(\R^{n+1})$ but, when $s>1/2$, is not in $C^2(\R^{n+1})$.
Since the Hessian of $\Phi$ is 
\[
D^2\Phi(x,z) = \begin{pmatrix} I & 0 \\ 0 &  \abs{z}^{\frac{1}{s}-2} \end{pmatrix},
\]
where $I$ denotes the identity matrix of size $n\times n$,
the linearized Monge--Amp\`ere equation associated with $\Phi$ is 
\begin{equation}\label{eq:recast}
\trace((D^2\Phi)^{-1} D^2 U) =  \Delta_xU + \abs{z}^{2-\frac{1}{s}}\partial_{zz} U=0\qquad\hbox{for}~z\neq0.
\end{equation}
As the coefficients $a^{ij}(x)$ are uniformly elliptic, see \eqref{eq:ellipticity}, we see that the coefficients in
\eqref{eq:Utilde} are comparable to the coefficients in \eqref{eq:recast}.

An important feature of the linearized Monge--Amp\`ere equation is its
intrinsic geometry that was first discovered by Caffarelli--Guit\'errez \cite{Caffarelli-Guti}.
They proved Harnack inequality for classical nonnegative solutions $v$ to \eqref{eq:linMA} when $\psi\in C^2$
and $g\equiv 0$, where the Euclidean balls and distance are replaced by Monge--Amp\`ere sections
and the Monge--Amp\`ere quasi-distance, respectively.
The Monge--Amp\`ere sections associated to a convex, $C^1$ function $\psi$ are the sublevel sets of
$\psi - \ell$ where $\ell$ is any linear function, while the corresponding quasi-distance is given by
$\delta_\psi(x_0,x)=\psi(x)-\psi(x_0)-\langle\nabla\psi(x_0),x-x_0\rangle$.

We show that the geometry for our degenerate/singular equation \eqref{eq:Utilde} with Neumann boundary condition at $\{z=0\}$
is given by the Monge--Amp\`ere sections $S_R$ associated to the strictly convex function $\Phi$,
that is, the sublevel sets of $\Phi - \ell$, and the Monge--Amp\`ere quasi-distance $\delta_\Phi$ in $\R^{n+1}$.
See Section \ref{sec:MA} for more details.
We prove the following Harnack inequality and H\"older regularity estimate for the extension equation in such sections.

\begin{thm}\label{thm:harnack extension}
Let $\Omega$ be a bounded domain,
$a^{ij}(x):\Omega \to \R$ be bounded, measurable functions that satisfy \eqref{eq:ellipticity} and let $f \in L^{\infty}(\Omega)$.
There exist positive constants 
$C_H = C_H(n,\lambda,\Lambda, s)>1$, 
$\kappa_0 = \kappa_0(n,s)<1$, 
and $\hat{K}_0= \hat{K}_0(n,s)>1$
such that for every section $S_{\hat{K}_0R} = S_{\hat{K}_0R}(x_0,z_0) \subset \subset \Omega \times \R$ 
and every nonnegative solution
$U \in C^2(S_{\hat{K}_0 R} \setminus \{z=0\}) \cap C({S}_{\hat{K}_0 R})$ 
such that $U$ is symmetric across $\{z=0\}$  and  $U_{z+} \in C(S_{\hat{K}_0 R} \cap \{z\geq0\})$ to
\begin{equation}\label{eq:extensionS}
\begin{cases}
a^{ij}(x) \partial_{ij}U + \abs{z}^{2-\frac{1}{s}}\partial_{zz}U = 0 & \hbox{in}~S_{\hat{K}_0 R} \cap  \{z\not=0\} \\
-\partial_{z+}U(x,0) = f & \hbox{on}~S_{\hat{K}_0 R} \cap \{z=0\},
\end{cases}
\end{equation}
we have that
\begin{equation}\label{eq:extension Harnack3}
\sup_{S_{\kappa_0 R}} U 
\leq C_H \( \inf_{S_{\kappa_0 R}} U + \norm{f}_{L^{\infty}(S_{\hat{K}_0 R}\cap \{z=0\})}R^s \).
\end{equation}
Consequently, there exist constants $0<\alpha_1 = \alpha_1(n,\lambda,\Lambda,s)<1/2$ and $\hat{C}_1= \hat{C}_1(n,\lambda,\Lambda,s)>1$ 
such that, for every solution
$U \in C^2(S_{\hat{K}_0 R} \setminus \{z=0\}) \cap C({S}_{\hat{K}_0 R})$ such that $U$ is symmetric across $\{z=0\}$ and  $U_{z+} \in C(S_{\hat{K}_0 R} \cap \{z\geq0\})$ 
to \eqref{eq:extensionS},
\begin{equation}\label{eq:extension Holder}
\begin{aligned}
|U(x_0,z_0)& - U(x,z)|\\
& \leq\frac{\hat{C}_1}{(\hat{K}_0R)^{\alpha_1}}[\delta_{\Phi}((x_0,z_0),(x,z))]^{\alpha_1}\(\sup_{S_{\hat{K}_0R}(x_0,z_0)} \abs{U}+  \norm{f}_{L^{\infty}(S_{\hat{K}_0 R}\cap \{z=0\})}R^s \)
\end{aligned}
\end{equation}
for every $(x,z) \in S_{\hat{K}_0R}(x_0,z_0)$.
\end{thm}

Note that if $S_{\hat{K}_0R}\cap \{z=0\} = \varnothing$, then 
$\norm{f}_{L^{\infty}(S_{\hat{K}_0 R}\cap \{z=0\})}R^{s}$ does not appear in the right hand side of \eqref{eq:extension Harnack3} and \eqref{eq:extension Holder}.

Regularity estimates, such as Harnack inequalities, for the linearized Monge--Amp\`ere equation \eqref{eq:linMA} 
have been studied by Caffarelli--Guti\'errez \cite{Caffarelli-Guti}, Forzani--Maldonado \cite{Forzani},
Le \cite{Le}, Maldonado \cite{Maldonado2,Maldonado-CVPDE}, Savin \cite{Savin-Liouville}, among others.
In each case, they either assume that $\det(D^2\psi)$ is bounded
away from zero and infinity (that is, the Monge--Amp\`ere measure $\mu_\psi(E)=|\nabla\psi(E)|$, $E\subset\R^n$, is
comparable to the Lebesgue measure), or that $\psi$ is sufficiently regular, e.g.~$\psi \in C^2$. 
For our function $\Phi$, we have that $D^2\Phi$ either degenerates or blows up near $\{z=0\}$ when $s\neq1/2$,
and, moreover, $\Phi\notin C^2$ when $s>1/2$. Therefore, \eqref{eq:Utilde}
is not covered by such previous results.
On the other hand, Maldonado proved Harnack inequality for degenerate 
elliptic equations associated with convex functions of the form $\psi(x) = \abs{x}^{p}$, $p\geq 2$, see \cite{Maldonado3} and  also \cite{MaldonadoPDEs3}.
However, not only are his techniques different than the ones presented here but also
his work does not include the singular case in which $s>1/2$.
Moreover, when we write \eqref{eq:extensionS} in $\Omega\times\R$ as a single equation,
we see that $U$ satisfies \eqref{eq:Utilde} in $\Omega\times\R$ with a right hand side
that is a singular measure with density $f(x)$ supported on $\{z=0\}$.

We develop a method of sliding paraboloids inspired by the work of Savin for
fully nonlinear uniformly elliptic equations \cite{Savin}.
For our setting, we work with a Neumann problem in a Monge--Amp\`ere geometry that brings additional challenges
because $\Phi$ is only $C^1$ and $D^2\Phi$ is degenerate/singular. For this,
we define \emph{paraboloids} $P:\R^{n+1} \to \R$ of opening $a>0$ with vertex $(x_v,z_v)$ by
\[
P(x,z) = -a \delta_{\Phi}((x_v,z_v),(x,z)) + c
\]
where $c$ is a constant.
We lift these paraboloids from below until they touch the graph of $U$ in a section $S_R$ for the first time. 
We estimate the Monge--Amp\`ere measure of the resulting set of contact points by the Monge--Amp\`ere measure of the set of vertices.
Observe that, since our equation \eqref{eq:Utilde} is degenerate/singular and $-\partial_{z+}U(x,0)=f(x)$, we need to
be able to control the contact
points $(x,z)$ for which $z=0$ in terms of the size of $f$.
Next, we show that, by increasing the opening of these paraboloids, they almost cover the section $S_R$ in measure.
This relies on explicit barriers whose construction is very delicate because of Neumann boundary condition
and the degeneracy/singularity of \eqref{eq:Utilde}.
Then, we build a refined geometric argument to obtain a localization estimate.
Thus, using a covering argument, we can conclude the proof of Theorem \ref{thm:harnack extension}
and deduce Theorem \ref{thm:harnack Ls}.

Our function $\Phi$ was also considered in \cite{Maldonado} to study the
fractional nonlocal linearized Monge--Amp\`ere equation. 
They established Harnack inequality and H\"older estimates for solutions to \eqref{eq:Poisson} 
when the coefficients $a^{ij}(x)$ are given by the matrix of cofactors of $D^2\psi$, where $\psi$ is a $C^3$ strictly convex function
and $\Omega$ is a section of $\psi$. Observe that in \cite{Maldonado} the weak Harnack inequality
is proved using the divergence form structure of the equation.
Whereas, in \eqref{eq:Poisson}, we not only consider general elliptic coefficients $a^{ij}(x)$, but also
the equation cannot be written in divergence form. Nevertheless, since 
the proof of the local boundedness of the solution to the extension problem in \cite{Maldonado} 
uses purely nondivergence form techniques, one can easily check that solutions to our extension problem
\eqref{eq:extension intro} satisfy the same local boundedness estimate as that of \cite[Theorem~11.3]{Maldonado}.

We additionally mention that Le in \cite{Le} proved Harnack inequality for the linearized Monge--Amp\`ere equation \eqref{eq:linMA}
when $\psi \in C^2$ and $0<\lambda\leq \det(D^2\psi(x)) \leq \Lambda$, by using sliding paraboloids
within the roadmap of the proof of Caffarelli--Guti\'errez \cite{Caffarelli-Guti}. Again, our methods (inspired by Savin \cite{Savin})
and results are different and independent of \cite{Le} (in particular, $\Phi$ is not smooth when $s>1/2$, $D^2\Phi$ is degenerate/singular,
and we have the Neumann boundary condition $-\partial_{z^+}U(x,0)=f(x)$).

Theorem \ref{thm:harnack extension} holds for bounded domains $\Omega$ and bounded, measurable
coefficients $a^{ij}(x)$. In Theorem \ref{thm:harnack Ls} we additionally require
that $\Omega$ satisfies the uniform exterior cone condition and that $a^{ij}(x)$ are H\"older continuous.
There are several reasons for these technical assumptions. First, the uniform exterior cone condition and the hypothesis
$a^{ij}(x)\in C(\Omega)\cap L^\infty(\Omega)$ give us the existence of an appropriate $C_0$-semigroup generated by $L$, so 
the fractional power operator $L^s$ can be defined using \eqref{eq:semigroupformula}.
Furthermore, under these conditions, the extension problem characterization in Theorem \ref{thm:Calpha extension} holds.
Second, our proof of Theorem \ref{thm:harnack extension} is for \emph{classical} solutions $U$ to the extension problem
and does not require any continuity assumptions on $a^{ij}(x)$ nor geometric conditions on $\Omega$.
Third, to apply Theorem \ref{thm:harnack extension}, we need to ensure that the solution $U$ given in Theorem \ref{thm:Calpha extension}
is classical, and for this we must require $a^{ij}(x)\in C^\alpha(\Omega)$.
It is an open problem and will be the object of future work to define $(-a^{ij}(x) \partial_{ij})^s$ in bounded domains 
when the coefficients are only bounded, measurable and 
to establish a corresponding extension equation and Harnack inequality for viscosity solutions to \eqref{eq:Poisson}.

The paper is organized as follows.
First, in Section \ref{sec:applications}, we show several applications of fractional powers of nondivergence form operators \eqref{eq:Ls}. Then, 
in Section \ref{sec:fractionalpowers}, we precisely define the fractional operator $(-a^{ij}(x) \partial_{ij})^s$ and prove the extension characterization.
In Section \ref{sec:MA}, we provide the necessary Monge--Amp\`ere background associated to our function $\Phi$. 
We prove a sequence of reductions of Theorem \ref{thm:harnack extension} in Section \ref{sec:reductions}.
Section \ref{sec:paraboloids} contains preliminary results on the Monge--Amp\`ere paraboloids $P$ associated to $\Phi$. 
Next, we establish several key results that will be used to prove the final reduction of Theorem  \ref{thm:harnack extension}.
In Section \ref{sec:lem1}, we estimate the Monge--Amp\`ere measure of the set of contact points for sliding paraboloids of fixed opening by the measure of the set of vertices.
The delicate construction of the barriers is done in Section \ref{sec:barrier}.
These are used in Section \ref{sec:lem2} to prove a localization estimate by means of a refined geometric argument. 
A Calder\'on--Zygmund-type covering lemma is proved in Section \ref{sec:lem3}.
Finally, in Section \ref{sec:main proofs}, we present the proof of the final reduction of Theorem \ref{thm:harnack extension} and the proof of Theorem \ref{thm:harnack Ls}.

%%%%%%%%%%%%%%%%%%%%%%%%
\section{Applications}\label{sec:applications}
%%%%%%%%%%%%%%%%%%%%%%%%

In this section we present some applications where fractional powers of nondivergence form elliptic operators naturally arise. 

%%%%%%%%%%%%%%%%%%%%%%%%
\subsection{Fractional Monge--Amp\`ere equations}
%%%%%%%%%%%%%%%%%%%%%%%%
 
If $u=u(x)$ is a convex, $C^2$ function, then one can check that the Monge--Amp\`ere operator acting on $u$ at a point $x$ can be written as
\begin{align*}
n \det (D^2u(x))^{1/n} 
	&= \inf\big\{ \Delta (u \circ B) (B^{-1}x) : B \in \mathcal{M}\big\}\\
	&= \inf\big\{a^{ij}\partial_{ij}u(x) : (a^{ij})= B^2,~B \in \mathcal{M}\big\},
\end{align*}
where the infimum is taken over the class $\mathcal{M}$ of all positive definite, symmetric
matrices $B$ of size $n\times n$ such that $\det(B) =1$.  
Motivated by these identities, Caffarelli--Charro defined in \cite{Caffarelli-Charro} the fractional Monge--Amp\`ere operator by 
\begin{equation}\label{eq:fracMA}
\mathcal{D}_su(x) =  \inf\big\{ -(-\Delta)^s (u \circ B) (B^{-1}x) : B \in \mathcal{M}\big\}, \quad 1/2 < s < 1.
\end{equation}
On the other hand, it was shown in \cite{Stinga-Jhaveri} that the operator in \eqref{eq:fracMA} can also be written as
\begin{equation}\label{eq:uniformlyDs}
\mathcal{D}_su(x) = \inf\big\{ -(-a^{ij}\partial_{ij})^su(x) : (a^{ij})= B^2,~B \in \mathcal{M}\big\},
\end{equation}
where $(-a^{ij}\partial_{ij})^s$ is the fractional power of the constant coefficients operator $-a^{ij}\partial_{ij}$.

The fractional Monge--Amp\`ere operator \eqref{eq:fracMA} is degenerate elliptic because the eigenvalues of the matrices $B \in \mathcal{M}$ are not
\emph{a priori} controlled from below or above. 
Nevertheless, it is proved in \cite{Caffarelli-Charro} that if $u$ is Lipschitz, semiconcave, 
and $\mathcal{D}_s u \geq \eta_0 >0$ in a bounded domain $\Omega$,
then $\mathcal{D}_s$ becomes uniformly elliptic in $u$, that is, there is a constant $\lambda>0$ such that
$$\mathcal{D}_su(x) = \mathcal{D}_s^\lambda u(x):= \inf\big\{ -(-\Delta)^s (u \circ B) (B^{-1}x) : B \in \mathcal{M},~B\geq\lambda I\big\}$$
for all $x\in\Omega$. Equivalently, in the description of \eqref{eq:uniformlyDs},
\[
\mathcal{D}_s^\lambda u(x) = \inf\big\{ -(-a^{ij}\partial_{ij})^su(x) : (a^{ij})= B^2,~B \in \mathcal{M},~B\geq\lambda I\big\}.
\]
It was observed in \cite{Vaughan} that, for each $x\in\Omega$, the infimum above is attained at some matrix $a^{ij}=a^{ij}(x)$. Therefore,
the fractional Monge--Amp\`ere operator in the uniformly elliptic regime is in fact given by
\[
\mathcal{D}_s^\lambda u(x)=-(-a^{ij}(x) \partial_{ij})^s u(x)\qquad\hbox{for every}~x\in\Omega.
\]
In other words, $\mathcal{D}_s^\lambda u(x)$ is the fractional power of the nondivergence form uniformly elliptic 
operator $L=-a^{ij}(x)\partial_{ij}$, where $a^{ij}(x)$ are bounded, measurable coefficients.

%%%%%%%%%%%%%%%%%%%%%%%%
\subsection{Elasticity}
%%%%%%%%%%%%%%%%%%%%%%%%

Consider an anisotropic elastic membrane represented by the graph of a function $U(x,z)$, for $(x,z)\in\Omega\times[0,\infty)$.
Suppose that we place a thin obstacle $\phi:\Omega \to \R$ on the hyperplane $\{z=0\}$, such that $\phi\leq0$ on $\partial\Omega$,
which pushes $U$ from below at $\{z=0\}$. By fixing $U=0$ on $\partial\Omega\times[0,\infty)$,
this problem is modeled by the following thin obstacle problem:
\begin{equation}\label{eq:thin}
\begin{cases}
a^{ij}(x) \partial_{ij}U + \partial_{zz}U = 0 & \hbox{in}~\Omega \times \{z>0\}\\
U(x,z) = 0 & \hbox{on}~ \partial \Omega \times \{z\geq0\}\\
U(x,0) \geq \phi(x) & \hbox{on}~\Omega \\
-\partial_{z+} U(x,0) \geq 0 & \hbox{on}~\Omega \\
-\partial_{z+} U(x,0) = 0 & \hbox{on}~\{ U(x,0) > \phi(x)\}.
\end{cases}
\end{equation}
The last two conditions are called the \emph{Signorini complementary conditions}. They follow from the fact that $\phi$
is pushing $U$ upwards, while $U$ is actually free in the noncoincidence set $\{U(x,0)>\phi(x)\}$.
The coefficients $a^{ij}(x)$ encode the heterogeneity of the membrane. The thin obstacle problem \eqref{eq:thin}
is equivalent to the problem of semipermeable cell membranes in biology
(see \cite{Duvaut}), where $a^{ij}(x)$ are a model for the cytoplasm inside the cell.

It follows from the extension problem characterization (see Theorem \ref{thm:Calpha extension}) with $s = 1/2$,
that the trace $u(x):= U(x,0)$ satisfies
$$-\partial_{z+} U(x,0) = (-a^{ij}(x) \partial_{ij})^{1/2}u(x)\quad\hbox{in}~\Omega.$$ 
Therefore, $U$ solves the thin obstacle problem \eqref{eq:thin}
if and only if its trace $u$ is the solution to the following fractional obstacle problem
$$\begin{cases}
(-a^{ij}(x) \partial_{ij})^{1/2}u \geq 0 & \hbox{in}~\Omega\\
(-a^{ij}(x) \partial_{ij})^{1/2}u = 0 &  \hbox{in}~\Omega \cap\{u > \phi\}\\
u \geq \phi & \hbox{in}~\Omega \\
u = 0 & \hbox{on}~\partial \Omega.
\end{cases}$$

%%%%%%%%%%%%%%%%%%%%%%%%
\subsection{Finance}
%%%%%%%%%%%%%%%%%%%%%%%%

Consider a particle moving randomly in a heterogeneous domain $\Omega$ that is killed when it reaches the boundary $\partial\Omega$. 
This random behavior can be modeled by a diffusion process $X_t$ whose infinitesimal generator
is a nondivergence form elliptic operator $L = -a^{ij}(x) \partial_{ij}$ in $\Omega$,
subject to homogeneous Dirichlet boundary conditions on $\partial\Omega$.
In this situation, the coefficients $a^{ij}(x)$ serve as a measure of the anisotropy of the medium,
or the preferred directions the particle chooses at every point $x$.
A model for particles randomly jumping inside a heterogeneous medium that are killed as soon as they reach  or
try to cross the boundary can be given by subordinating the process $X_t$ with a
$2s$-stable L\`evy subordinator $T_t$, for $0<s<1$. The resulting subordinated process
$Y_t = X_{T_t}$ is then generated by the fractional power operator $L^s = (-a^{ij}(x)\partial_{ij})^s$, $0 < s < 1$. 
See \cites{Jacob} for the case of smooth coefficients and domains, and
\cite{Song} for the case when $X_t$ is a Wiener process.

Next, let $\tau$ be the optimal stopping time that maximizes
the function 
\[
u(x) = \sup_{\tau} \mathbb{E}[\phi(Y_\tau) ; \tau <+\infty],
\]
where $\phi \in C_0(\Omega)$ (see \eqref{eq:definitionofC0}), $\mathbb{E}$ denotes the expected value, and the process
$Y$ is set to start at $x\in\Omega$. It turns out that $u$ is the solution to the following obstacle problem:
\begin{equation}\label{eq:frac-obstacle}
\begin{cases}
(-a^{ij}(x) \partial_{ij})^su \geq 0 & \hbox{in}~\Omega\\
(-a^{ij}(x) \partial_{ij})^su = 0 &  \hbox{in}~\Omega \cap\{u > \phi\}\\
u \geq \phi & \hbox{in}~\Omega \\
u = 0 & \hbox{on}~\partial \Omega.
\end{cases}
\end{equation}
These free boundary problems appear in financial models (see \cite{Cont}) where $u$ is
the value of a perpetual American option in which the asset prices are modeled by $Y_t$ and $\phi$ is the payoff function.

%%%%%%%%%%%%%%%%%%%%%%%%
\section{Fractional powers of elliptic operators and extension problem}\label{sec:fractionalpowers}
%%%%%%%%%%%%%%%%%%%%%%%%

Here, we give the precise definition of the fractional power operator $L^s=(-a^{ij}(x)\partial_{ij})^s$ in \eqref{eq:Poisson}
and present the extension problem characterization, i.e.~Theorem \ref{thm:Calpha extension}. For this, we apply the method of semigroups
of \cite{Gale,Stinga-Torrea} (see also \cite{Stinga}) which we
describe next.

%%%%%%%%%%%%%%%%%%%%%%%%
\subsection{Method of semigroups for fractional power operators}
%%%%%%%%%%%%%%%%%%%%%%%%

A family $\{T_t\}_{t \geq 0}$ of bounded, linear operators on a Banach space $X$ is a \emph{semigroup} on $X$ if 
\[
T_0 = Id \quad \hbox{and} \quad T_{t_1}\circ T_{t_2} = T_{t_1+t_2} \quad \hbox{for every}~t_1,t_2\geq0,
\]
where $Id$ denotes the identity operator. We say that a semigroup
$\{T_t\}_{t \geq 0}$ is a \emph{$C_0$-semigroup} if $T_tu \to u$ as $t\to 0^+$ for all $u \in X$. 
A semigroup $\{T_t\}_{t \geq 0}$ is \emph{uniformly bounded} if its operator norm is uniformly bounded in $t$,
that is, there is a constant $M \geq 1$ such that $\|T_t\|\leq M$ for all $t \geq 0$.
The \emph{infinitesimal generator} $A$ of a semigroup $\{T_t\}_{t \geq 0}$ is the closed linear operator defined as
\begin{equation}\label{eq:generator}
-Au=\lim_{t\to0^+}\frac{T_tu-u}{t}
\end{equation}
in the domain $\Dom(A)=\{u\in X:\hbox{the limit in }\eqref{eq:generator}\hbox{ exists}\}\subset X$. In this case,
we write $T_t=e^{-tA}$. Hence, if $A$ is the infinitesimal generator of a $C_0$-semigroup
$\{e^{-tA}\}_{t\geq0}$ on $X$, then the function $v=e^{-tA}u$, for $u\in X$, satisfies the heat equation for $A$:
\[
\begin{cases}
\partial_tv = -Av & \hbox{for}~t>0\\
v = u & \hbox{for}~t=0.
\end{cases}
\]
Conversely, a linear operator $(A,\Dom(A))$ on $X$ is said to \emph{generate} a semigroup if there is 
a semigroup $\{T_t\}_{t \geq 0}$ for which $A$ is its infinitesimal generator, that is, $T_t=e^{-tA}$.
Given a uniformly bounded $C_0$-semigroup $\{T_t=e^{-tA}\}_{t\geq0}$ on $X$, the fractional power $A^s$ of its
infinitesimal generator is defined as
$$A^su=\frac{1}{\Gamma(-s)}\int_0^\infty\big(e^{-tA}u-u\big)\,\frac{dt}{t^{1+s}},\quad\hbox{for all}~u\in\Dom(A)\subset X,$$
where $0<s<1$. If the semigroup $\{e^{-tA}\}_{t\geq0}$
has \emph{exponential decay}, that is, $\|e^{-tA}\|\leq Me^{-\varepsilon t}$, for some $\varepsilon>0$, for all $t\geq0$,
then the negative power $A^{-s}$, $s>0$, is given by
$$A^{-s}f=\frac{1}{\Gamma(s)}\int_0^\infty e^{-tA}f\,\frac{dt}{t^{1-s}},\quad\hbox{for all}~f\in X.$$
Thus, under the exponential decay assumption on $\{e^{-tA}\}_{t\geq0}$,
given $f\in X$, the solution $u\in\Dom(A^s)$ to the fractional problem $A^su=f$ is $u=A^{-s}f$. Here $\Dom(A^s)$ is defined
as the range of $A^{-s}$. For all these details, see \cite{Pazy,Yosida}.

Fractional powers $A^{\pm s}$ of infinitesimal generators $A$ of uniformly bounded $C_0$-semigroups
can be characterized by extension problems. For the case when $X$ is a Hilbert space see \cite{Stinga-Torrea},
while for the case when $X$ is a general Banach space see \cite{Gale}.

\begin{thm}[{See \cite[Theorems~1.1~and~2.1,~Remark~2.2]{Gale}}]\label{thm:Gale}
Let $(A,\Dom(A))$ be the infinitesimal generator of a uniformly bounded $C_0$-semigroup $\{e^{-tA}\}_{t\geq0}$ on a Banach space $X$.
Let $0<s<1$. Define, for $y>0$ and any $u\in X$,
\begin{equation}\label{eq:Uy}
U(y)=\frac{y^{2s}}{4^s\Gamma(s)}\int_0^\infty e^{-y^2/(4t)}e^{-tA}u\,\frac{dt}{t^{1+s}}.
\end{equation}
Then 
$U\in C^\infty((0,\infty),\Dom(A))\cap C([0,\infty),X)$ is a solution to the extension problem
$$\begin{cases}
-AU+\frac{1-2s}{y}\partial_yU+\partial_{yy}U=0&\hbox{for}~y>0\\
\lim_{y\to0^+}U(y)=u&\hbox{in}~X.
\end{cases}$$
Moreover, $\|U(y)\|_X\leq M\|u\|_X$, for all $y\geq 0$.
Furthermore, if $u\in\Dom(A)$ then
$$-\lim_{y\to0^+}y^{1-2s}\partial_yU(y)=c_sA^su=-2s\lim_{y\to0^+}\frac{U(y)-u}{y^{2s}}\quad\hbox{in}~X$$
where $c_s=\frac{\Gamma(1-s)}{4^{s-1/2}\Gamma(s)}>0$. If, in addition, $\{e^{-tA}\}_{t\geq0}$ has exponential decay and
$u\in\Dom(A)$ satisfies $A^su=f$, for some $f\in X$,
then the solution $U$ in \eqref{eq:Uy} can also be written as
$$U(y)=\frac{1}{\Gamma(s)}\int_0^\infty e^{-y^2/(4t)}e^{-tA}f\,\frac{dt}{t^{1-s}}$$
and, in particular, $-\lim_{y\to0^+}y^{1-2s}\partial_yU(y)=c_sf$ and $U(0)=u$.
\end{thm}

%%%%%%%%%%%%%%%%%%%%%%%%
\subsection{Fractional powers of nondivergence form elliptic operators}
%%%%%%%%%%%%%%%%%%%%%%%%

To give the definition of \eqref{eq:Ls}, we need conditions on $a^{ij}(x)$ and $\Omega$ so that
$L$ as in \eqref{eq:L} generates a uniformly bounded $C_0$-semigroup with exponential decay in an appropriate Banach space $X$.
For this, we assume that the bounded domain $\Omega\subset\R^n$ satisfies the \emph{uniform exterior cone condition}, namely,
that there is a right circular cone $C$ such that for all $x\in\partial\Omega$ there is a cone $C_x$ with vertex $x$
that is congruent to $C$ such that $\overline{\Omega}\cap C_x=\{x\}$.
We define the Banach space $X = C_0(\Omega)$ by
\begin{equation}\label{eq:definitionofC0}
C_0({\Omega}) = \{ u \in C(\overline{\Omega}) : u \equiv 0~\hbox{on}~\partial\Omega\},
\end{equation}
endowed with the $L^\infty$-norm. Let $L$ be the linear operator on $C_0(\Omega)$ given by
\begin{equation}\label{eq:Ldomain}
L = -a^{ij}(x) \partial_{ij}, \quad \Dom(L) = \{ u \in C_0(\Omega) \cap W^{2,n}_{\loc}(\Omega) : Lu \in C_0(\Omega)\},
\end{equation}
where the coefficients $a^{ij}(x) \in C(\Omega)\cap L^{\infty}(\Omega)$ are symmetric and satisfy \eqref{eq:ellipticity}.
Under these hypotheses, $L$ generates a uniformly bounded $C_0$-semigroup on $C_0(\Omega)$ with exponential decay. 

\begin{prop}[{See \cite[Proposition~4.7]{Arendt}}]\label{thm:arendt}
Assume that $\Omega \subset \R^n$ is a bounded domain that satisfies the uniform exterior cone condition and that 
$a^{ij}(x) \in C({\Omega}) \cap L^{\infty}(\Omega)$ are symmetric and satisfy \eqref{eq:ellipticity}.
The operator $L$ defined by \eqref{eq:Ldomain}
generates a uniformly bounded $C_0$-semigroup, denoted by $\{e^{-tL}\}_{t\geq0}$, on $C_0(\Omega)$,
such that if $u\in C_0(\Omega)$ satisfies $u\geq 0$, then $e^{-tL}u\geq 0$, for all $t\geq0$. Moreover,
there are constants $M\geq1$ and $\varepsilon>0$ such that
\begin{equation}\label{eq:semigroup bound}
\|e^{-tL}u\|_{C_0(\Omega)}\leq Me^{-\varepsilon t}\norm{u}_{C_0(\Omega)}, \quad \hbox{for all}~t\geq 0.
\end{equation}
\end{prop}

In other words, by Proposition \ref{thm:arendt} and the maximum principle for parabolic equations (see \cite{Friedman}), 
for any $u\in C_0(\Omega)$, the function 
 $v(x,t) = e^{-tL}u(x)\in C^1((0,\infty),\Dom(L))\cap C([0,\infty),C_0(\Omega))$ is the unique solution to the heat equation driven by $L$ with initial data $u$:
\begin{equation}\label{eq:heat}
\begin{cases}
\partial_tv(x,t) = a^{ij}(x) \partial_{ij}v(x,t) & \hbox{in}~\Omega \times\{t >0\}\\
v(x,t) = 0 & \hbox{on}~\partial \Omega \times\{t \geq 0\}\\
v(x,0) = u(x) & \hbox{on}~\Omega \times\{t =0\}.
\end{cases}
\end{equation}
Now we can formalize the definition of the fractional power operator \eqref{eq:Ls}.

\begin{defn}
Assume that the bounded domain $\Omega \subset \R^n$ satisfies the uniform exterior cone condition and that 
$a^{ij}(x) \in C({\Omega}) \cap L^{\infty}(\Omega)$ are symmetric and satisfy \eqref{eq:ellipticity}. 
Consider the Banach space $C_0(\Omega)$ and the operator $L=-a^{ij}(x)\partial_{ij}$ given by \eqref{eq:Ldomain}. 
We define the fractional power operator $L^s=(-a^{ij}(x)\partial_{ij})^s: \Dom(L) \to C_0(\Omega)$ by 
\begin{equation}\label{eq:semigroup}
(-a^{ij}(x)\partial_{ij})^s u(x) = \frac{1}{\Gamma(-s)} \int_0^{\infty}\big(e^{-tL}u(x) - u(x)\big) \, \frac{dt}{t^{1+s}}, \quad 0 < s < 1,~x\in\Omega.
\end{equation}
\end{defn}

\begin{rem}[Pointwise formula]\label{rem:pointwise}
It is known that, under certain conditions on $L$, the semigroup $\{e^{-tL}\}_{t\geq0}$ has a heat kernel, namely, there is a function $H_t(x,y)$ such that
$$e^{-tL}u(x)=\int_\Omega H_t(x,y)u(y)\,dy\qquad\hbox{for all}~t>0,~x\in\Omega.$$
For example, the heat kernel exists and satisfies the Gaussian estimate
\begin{equation}\label{eq:Gaussian}
0\leq H_t(x,y)\leq C\frac{e^{-c|x-y|^2/t}}{t^{n/2}}\qquad\hbox{for all}~t>0,~x,y\in\Omega,
\end{equation}
for some constants $C,c>0$, whenever the coefficients $a^{ij}(x)$ are H\"older continuous, see \cite{Friedman}.
In this situation, it follows from \eqref{eq:semigroup} that for any smooth function $u\in\Dom(L)$,
$$(-a^{ij}(x)\partial_{ij})^su(x)=\int_\Omega\big(u(x)-u(y)\big)K_s(x,y)\,dy+B_s(x)u(x)\qquad\hbox{for all}~x\in\Omega$$
where $0\leq K_s(x,y)\leq C_{n,s}|x-y|^{-(n+2s)}$, for $x,y\in\Omega$, $x\neq y$, and $B_s(x)\in L^\infty(\Omega)$.
Therefore, the fractional operator $L^s$ is a nonlocal, integro-differential operator in $\Omega$.
\end{rem}

\begin{rem}[Negative fractional powers]
Let $f \in C_0(\Omega)$ and assume that $u\in\Dom(L^s)$ is a solution to \eqref{eq:Poisson}, that is, $(-a^{ij}(x)\partial_{ij})^su=f$ in $\Omega$.
By Proposition \ref{thm:arendt}, the semigroup $\{e^{-tL}\}_{t\geq0}$ has exponential decay. Then $u$ can be written as
\begin{equation}\label{eq:inverse}
u(x) = (-a^{ij}(x)\partial_{ij})^{-s} f(x) = \frac{1}{\Gamma(s)} \int_0^{\infty} e^{-tL}f(x) \, \frac{dt}{t^{1-s}}\qquad\hbox{for all}~x\in\Omega.
\end{equation}
If the coefficients $a^{ij}(x)$ are H\"older continuous then we can use the heat kernel from Remark \ref{rem:pointwise}
into \eqref{eq:inverse} to write
$$u(x)= (-a^{ij}(x)\partial_{ij})^{-s} f(x)=\int_\Omega K_{-s}(x,y)f(y)\,dy\qquad\hbox{for all}~x\in\Omega$$
where, by the estimate \eqref{eq:Gaussian}, $0\leq K_{-s}(x,y)\leq C_{n,-s}|x-y|^{-(n-2s)}$, for all $x,y\in\Omega$, $x\neq y$. 
\end{rem}

\begin{proof}[Proof of Theorem \ref{thm:Calpha extension}]
We choose $X=C_0(\Omega)$ and $A=L$ as in \eqref{eq:Ldomain} so that, by Proposition \ref{thm:arendt},
$L$ generates a uniformly bounded $C_0$-semigroup on $C_0(\Omega)$ with exponential decay.
Then, the solution $U(y)$ in \eqref{eq:Uy} satisfies the properties stated in Theorem \ref{thm:Gale}.
With the change of variables $z = (y/(2s))^{2s}$, we obtain that $U(z)\equiv U(x,z)$ verifies
the formulas and properties of Theorem \ref{thm:Calpha extension}.
If the coefficients $a^{ij}(x)$ are also H\"older continuous then, by interior Schauder estimates (see \cite[Theorem 9.19]{GilbargTrudinger}),
the solution $U$ is classical. Moreover, by the weak maximum principle (see \cite[Theorem~3.1]{GilbargTrudinger}),
there is at most one classical solution to \eqref{eq:extension intro}
such that $\lim_{z \to \infty} \norm{U(\cdot,z)}_{L^{\infty}(\Omega)}= 0.$
Using \eqref{eq:semigroup bound} it is easy to check that the solution $U$ given by \eqref{eq:U-defn} has
such decay at infinity and hence is the unique solution. 
\end{proof}

%%%%%%%%%%%%%%%%%%%%%%%%
\section{Monge--Amp\`ere setting}\label{sec:MA}
%%%%%%%%%%%%%%%%%%%%%%%%

We present the necessary background for the Monge--Amp\`ere geometry associated to equation \eqref{eq:Utilde} as well as the notation that will be used throughout the rest of this work. We reference the reader to  \cites{Forzani,Gutierrez} for details
about the Monge--Amp\`ere geometry associated to general convex functions. 

Given $0 < s < 1$, we define the functions $\varphi : \R^n \to \R$ and $h: \R \to \R$ by
\begin{equation}\label{eq:phiandh}
\varphi(x) = \frac{1}{2} \abs{x}^2 \quad \hbox{and} \quad h(z) = \frac{s^2}{1-s}\abs{z}^{\frac{1}{s}}.
\end{equation}
Notice that $\varphi \in C^{\infty}(\R)$ and $h \in C^1(\R) \cap C^2(\R \setminus \{0\})$ are strictly convex.
Set
\begin{equation}\label{eq:capitalphi}
\Phi(x,z) = \varphi(x) + h(z) \quad \hbox{for all}~(x,z) \in \R^{n+1}.
\end{equation}
We note that
\[
h'(z) = \frac{s}{1-s} \abs{z}^{\frac{1}{s}-2}z, 
\quad h''(z) = \abs{z}^{\frac{1}{s}-2},
\quad D^2\Phi(x,z) = \begin{pmatrix} I & 0 \\ 0 &  \abs{z}^{\frac{1}{s}-2} \end{pmatrix}.
\]
It is clear that $h'(-z) = -h'(z)$ and $h'(0) = 0$.

The \emph{Monge-Amp\`ere measure} associated to a strictly convex function $\psi \in C^1(\R^n)$ is the Borel measure given by
\[
\mu_{\psi}(E) = \abs{\nabla\psi(E)}\quad \hbox{for every Borel set}~E \subset \R^n,
\]
where $|A|$ denotes the Lebesgue measure of a measurable set $A\subset\R^n$.
Since $\nabla\varphi(x) = x$, it clear that $\mu_{\varphi}(E) = \abs{E}$. 

\begin{lem}\label{lem:h-integral}
For a Borel set $I \subset \R$,
\[
\mu_h(I) = \int_{I} h''(z) \, dz.
\]
Consequently, for a Borel set $E \subset \R^{n+1}$,
\[
\mu_\Phi(E) = \int_{E} h''(z) \, dz \, dx.
\]
\end{lem}

\begin{proof}
Consider an interval $(a,b) \subset \R$ such that $0 \in (a,b)$. Note that $h'$ 
is monotone increasing, injective, and $h'(z) =0$ if and only if $z=0$. 
Since $h$ is $C^2$ and strictly convex in $\R\setminus \{z=0\}$, we have that
\begin{align*}
\mu_h((a,b))
	= \abs{h'((a,b))}
	&= \abs{h'((a,0))} + \abs{h'(0)} + \abs{h'((0,b))}\\
	&= \int_{a}^0 h''(z)\, dz + 0 + \int_0^b h''(z) \, dz= \int_{a}^b h''(z) \, dz.
\end{align*}
The result follows for any interval and hence for any Borel set $I \subset \R$.
\end{proof}

The \emph{Monge-Amp\`ere (quasi)-distance} associated to a strictly convex function $\psi \in C^1(\R^n)$ is given by
\[
\delta_{\psi}(x_0,x) = \psi(x) - \psi(x_0) - \la \nabla\psi(x_0), x-x_0 \ra.
\]
By convexity, $\delta_\psi\geq0$, and $\delta_\psi(x_0,x)=0$ if and only if $x_0=x$.
We use the terminology quasi-distance when there exists a constant $K \geq 1$ such that
\[
\delta_{\psi}(x_1,x_2) 
	\leq K \(\min\{\delta_{\psi}(x_1,x_3),\delta_{\psi}(x_3,x_1)\} 
	+ \min\{\delta_{\psi}(x_2,x_3), \delta_{\psi}(x_3,x_2)\} \) 
\]
for all $x_1,x_2,x_3 \in \R^n$.
For our functions $\varphi$, $h$, and $\Phi$ in \eqref{eq:phiandh} and \eqref{eq:capitalphi}, we have
\begin{align*}
\delta_{\varphi}(x_0,x)
	&= \frac{1}{2} \abs{x}^2 - \frac{1}{2} \abs{x_0}^2 - \la x_0, x-x_0 \ra 
	= \frac{1}{2} \abs{x-x_0}^2\\
\delta_{h}(z_0,z)
	&= h(z) - h(z_0) - h'(z_0)(z-z_0)\\
\delta_{\Phi}((x_0,z_0),(x,z)) 
	&= \delta_{\varphi}(x_0,x) + \delta_h(z_0,z).
\end{align*}
We will later show that $\delta_h$, $\delta_{\varphi}$, and $\delta_{\Phi}$ are indeed quasi-distances (see Corollary \ref{lem:doubling}).
 
A \emph{Monge--Amp\`ere section} of radius $R>0$, centered at $x_0\in\R^n$ associated to a strictly convex function $\psi \in C^1(\R^n)$
is defined as
\[
S_\psi(x_0,R) = \{ x \in\R^n: \delta_\psi(x_0,x)  < R\}.
\]
The supporting hyperplane to $\psi$ at $x_0$ is given by
$\ell(x) = \psi(x_0) + \la \nabla\psi(x_0),x-x_0 \ra$.
Then, $S_{\psi}(x_0,R) = \{x : \psi(x)-\ell(x) < R\}$, and
we see that the Monge--Amp\`ere sections for $\psi$ are the sublevel sets of $\psi-\ell$.
In the case of $\varphi$ in \eqref{eq:phiandh}, the sections correspond to Euclidean balls with the same center
\begin{equation}\label{eq:sections-balls}
S_\varphi(x_0,R)
	= \Big\{x : \frac{1}{2} \abs{x-x_0}^2 < R \Big\}
	= B_{\sqrt{2R}}(x_0).
\end{equation}
The sections for $h$ in \eqref{eq:phiandh} with radius $R>0$ correspond to intervals in $\R$.

We say that the Monge--Amp\`ere measure $\mu_{\psi}$
is \emph{doubling with respect to the center of mass}
on the sections of $\psi$, written $\mu_{\psi} \in (DC)_{\psi}$, if there is a constant $C_d>0$ such that
\begin{equation}\label{eq:doubling-COM}
\mu_\psi(S_\psi(x,R)) \leq C_d \mu_{\psi}\bigg(\frac{1}{2} S_\psi(x,R)\bigg) \quad \hbox{for all sections}~S_\psi(x,R).
\end{equation}
Here, we use the notation $\alpha S_\psi(x,R) = \{ \alpha(y-x^*) + x^* : y \in S_\psi(x,R)\}$, for $\alpha>0$,
where $x^*$ is the center of mass of $S_\psi(x,R)$. 
On the other hand, we say that $\mu_{\psi}$ is \emph{doubling with respect to the parameter}
on the sections of $\psi$ if there is a constant $C_d'>0$ such that
\begin{equation}\label{eq:doubling-radius}
\mu_\psi(S_\psi(x,R)) \leq C_d' \mu_{\psi}\bigg( S_\psi\bigg(x,\frac{R}{2}\bigg)\bigg) \quad \hbox{for all sections}~S_\psi(x,R).
\end{equation}
It can be seen that \eqref{eq:doubling-COM} implies \eqref{eq:doubling-radius}, but the converse is not true in general, see \cite{Gutierrez}.

Finally, we say that $\psi$ satisfies the \emph{engulfing property} if there is a constant $\theta \geq 1$ such that, for every section 
$S_\psi(x,R)$, if $x_1 \in S_\psi(x,R)$, then $S_\psi(x,R) \subset S_\psi(x_1,\theta R)$.

For the next result, see Theorem 5 in \cite{Forzani} and the comments following it and also Lemma 2.1 in \cite{Maldonado2017}.

\begin{thm}\label{thm:DC-engulf}
Let $\psi\in C^1(\R^n)$ be a strictly convex function. The following are equivalent.
\begin{enumerate}
	\item $\mu_{\psi} \in (DC)_\psi$;
	\item $\psi$ satisfies the engulfing property;
	\item there are constants $c,C>0$ such that
	\[
	cR^n \leq \abs{S_\psi(x,R)} \mu_{\psi}(S_\psi(x,R)) \leq C R^n
	\]
	for all sections $S_\psi(x,R)$;
	\item $\delta_\psi$ is a quasi-distance. 
\end{enumerate}
All the statements are equivalent in the sense that the constants in each property only depend on each other. 
If $\mu_{\psi} \in (DC)_{\psi}$, then there exists a constant $K_d >1$, depending only on the doubling constant $C_d$ in \eqref{eq:doubling-COM} and on dimension, such that
\begin{equation}\label{eq:reversedoubling}
\mu_{\psi}(S_{\psi}(x,r_2)) \leq K_d \(\frac{r_2}{r_1}\)^n  \mu_{\psi}(S_{\psi}(x,r_1)) \quad \hbox{for all}~x \in \R^n,~0 < r_1<r_2.
\end{equation}
\end{thm}

In order to show that our convex function $\Phi$ in \eqref{eq:capitalphi} satisfies Theorem \ref{thm:DC-engulf}, we need to introduce
the notion of Monge--Amp\`ere cubes associated with $\Phi$. 
Many of our proofs will rely on the fact that $\Phi(x,z) = \varphi(x) + h(z)$ has separated variables. 

\begin{defn}
A \emph{Monge--Amp\`ere cube} of radius $R>0$, centered at $(x,z) \in \R^{n+1}$, associated to $\Phi$ is given by
\[
Q_R(x,z) = 
 S_{\varphi_1}(x_1,R) \times \dots \times  S_{\varphi_n}(x_n,R) \times  S_{h}(z,R)
\] 
where $x = (x_1,\dots, x_n)$ and $\varphi_i:\R \to \R$ is defined by $\varphi_i(x_i) = \frac{1}{2} \abs{x_i}^2$  for $i= 1,\dots, n$.
\end{defn}

\begin{notation}
We will always use the following notation.
\begin{itemize}
	\item $x = (x_1,\dots,x_n) \in \R^n$, $z \in \R$.
	\item $S_R(x) \subset \R^n$ is a section of radius $R>0$ associated with $\varphi$ centered at $x$.
	\item $S_R(z) \subset \R$ is a section of radius $R>0$ associated with $h$, centered at $z$.
	\item $S_R(x,z)\subset \R^{n+1}$ is a section of radius $R>0$ associated with $\Phi$, centered at $(x,z)$.
	\item $Q_R(x) \subset \R^n$ is a cube of radius $R>0$ associated with $\varphi$ centered at $x$.
	\item $Q_R(z) \subset \R$ is a cube of radius $R>0$ associated with $h$ centered at $z$.
	\item $Q_R(x,z) \subset \R^{n+1}$ is a cube of radius $R>0$ associated with $\Phi$ centered at $(x,z)$.
\end{itemize}
\end{notation}

The relation between Monge--Amp\`ere cubes and sections is given by the following result.

\begin{lem}[Lemma 6 in \cite{Forzani}]\label{lem:tensor}
Fix $m \in \N$. For each $j = 1,\dots, m$, let $\psi_j: \R^{n_j} \to \R$ be strictly convex, differentiable functions.
Set $n = \sum_{j=1}^m n_j$ and define
\[
\psi(x) =\sum_{j=1}^m \psi_j(x_j), \quad x = (x_1,\dots, x_m) \in \R^n, x_j\in \R^{n_j}. 
\]
Then
\[
S_\psi(x,R) \subset \prod_{j=1}^m S_{\psi_j}(x_j,R) \subset S_\psi(x,mR)
\]
for all $x = (x_1,\dots, x_m) \in \R^n$ and $R>0$. 

In particular,
if $\psi_j$ satisfy the engulfing property with corresponding constants $\theta_j$ for all $j = 1,\dots m$,  
then $\phi$ satisfies the engulfing property with $\theta =m \max_j\{\theta_j\}$.
Conversely, if $\psi$ satisfies the engulfing property for some $\theta >1$, then $\psi_j$ satisfies the engulfing property with constant $\theta$ for all $j = 1,\dots, m$. 
\end{lem}

As a consequence of Lemma \ref{lem:tensor},
\[
S_{R}(x,z) \subset S_R(x) \times S_R(z) \subset S_{2R}(x,z)
\]
and
\[
S_R(x,z) \subset Q_R(x,z) \subset S_{(n+1)R}(x,z) 
\]
for all $(x,z) \in \R^{n+1}$ and $R>0$.

As discussed in \cite[Section 7.1]{Maldonado}, 
$h''(z) = \abs{z}^{1/s-2}$ is a Muckenhoupt $A_{\infty}(\R)$ weight.
In particular, the following lemma holds. 
See \cite[Section 9.3]{Grafakos} for definitions and properties of the class $A_{\infty}(\R)$. 

\begin{lem}\label{lem:A-infty}
Given $0 < \varepsilon<1$, there is $0 < \varepsilon_0<1$, depending only on $\varepsilon$ and $0 < s < 1$, such that for any section $S_R=S_R(z_0)$ and any measurable set $E \subset S_R$,
\[
\frac{\abs{E}}{\abs{S_R}} < \varepsilon_0 \quad \hbox{implies} \quad \frac{\mu_h(E)}{\mu_h(S_R)} < \varepsilon.
\]
\end{lem}

We can now establish Theorem \ref{thm:DC-engulf} for $\psi=\Phi$.

\begin{cor}\label{lem:doubling}
We have $\mu_{\varphi} \in (DC)_{\varphi}$ and $\mu_{h} \in (DC)_h$, so that (1)--(4) in Theorem \ref{thm:DC-engulf} hold and are equivalent for $\varphi$ and $h$. Moreover, the following statements hold
and are equivalent.
\begin{enumerate}
\item  $\mu_{\Phi} \in (DC)_\Phi$ with corresponding doubling constant $C_d = C_d(n,s)>0$;
 
\item $\Phi$ satisfies the engulfing property with corresponding constant $\theta = \theta(n,s)>0$;

\item there are positive constants $C= C(n,s),~c = c(n,s)$ such that
	\[
	c R^{n+1} \leq \abs{S_R(x,z)} \mu_{\Phi}(S_R(x,z)) \leq C R^{n+1}
	\]
	for all sections $S_R(x,z)$;
\item $\delta_\Phi$ is a quasi-distance with constant $K = K(n,s)\geq1$.
\end{enumerate}
Consequently, $\mu_{\varphi}$, $\mu_h$ and $\mu_{\Phi}$ satisfying the doubling estimate \eqref{eq:reversedoubling}.
\end{cor}

\begin{proof}
By \eqref{eq:sections-balls}, we can write 
\[
\mu_{\varphi}(S_R(x_0))
	= |B_{\sqrt{2R}}(x_0)|
	\leq 2^n \big|\frac{1}{2}B_{\sqrt{2R}}(x_0) \big|
	=2^n\mu_{\varphi}\bigg(\frac{1}{2}S_R(x_0)\bigg)
\]
Hence $\varphi \in (DC)_{\varphi}$ with doubling constant $C_d^{\varphi} = C_d^{\varphi}(n)$.
Since $h''(z)$ is a Muckenhoupt $A_{\infty}(\R)$ weight for all $0 < s < 1$,
we have that 
$\mu_h \in (DC)_h$ with doubling constant $C_d^h = C_d^h(s)$.
It follows from Theorem \ref{thm:DC-engulf} that $\mu_\varphi$ and $\mu_h$ satisfy the engulfing
property and, by Lemma \ref{lem:tensor}, so does $\mu_\Phi$. Hence, 
the conclusion follows from Theorem \ref{thm:DC-engulf}.
\end{proof}

\begin{rem}\label{rem:r^s new}
There is a constant $q_s$, depending only on $s$, so that $S_{R}(0) = B_{q_sR^s}(0)$ for any $R>0$.
Indeed, $z \in S_{R}(0)$ if and only if 
\[ 
\delta_h(0,z) = h(z) = \frac{s^2}{1-s}\abs{z}^{1/s}<R
\]
which is equivalent to
\begin{equation}\label{eq:qs}
\abs{z} < \(\frac{1-s}{s^2}\)^s R^s =: q_sR^s.
\end{equation}
\end{rem}

\begin{notation}\label{note:2}
We will always use the following notation.
\begin{itemize}
	\item $\theta$ is the engulfing constant associated with $\Phi$.
	\item $K$ is the quasi-triangle constant associated with $\Phi$.
	\item $K_d$ is the the constant in \eqref{eq:reversedoubling}.
	\item $q_s$ is the constant in \eqref{eq:qs}.
\end{itemize}
\end{notation}

We end this section by presenting two lemmas that will be used later in the proofs.

\begin{lem}\label{lem:ordering}
Let $\psi \in C^1(\R)$ be a strictly convex function. 
If $x_0 < x_1 < x_2$, then
\[
\delta_\psi(x_0,x_1) < \delta_\psi(x_0,x_2) \quad \hbox{and} \quad
\delta_\psi(x_1,x_2) < \delta_\psi(x_0,x_2).
\]
Consequently, for any $x_0,x_1 \in \R$ and $R>0$, if $r>0$ is such that $S_\psi(x_1,r) \subset S_\psi(x_0,R)$, then $r \leq R$.
In particular, for any $(x_0,z_0), (x_1,z_1) \in \R^{n+1}$ and $R>0$, if $r>0$ is such that
$Q_r(x_1,z_1) \subset Q_R(x_0,z_0)$,  then $r \leq R$.
\end{lem}

\begin{proof}
By the convexity of $\psi$, 
\begin{align*}
\psi'(x_0) < \frac{\psi(x_1) - \psi(x_2)}{x_1-x_2},
\end{align*}
so that
\begin{align*}
\delta_\psi(x_0,x_1)
	&= \psi(x_1) - \psi(x_0) - \psi'(x_0)(x_1-x_0) \\
	&< \psi(x_2) - \psi(x_0) - \psi'(x_0)(x_2-x_0) = \delta_\psi(x_0,x_2).
\end{align*}

Next, define a function $\Psi: \R \to \R$ by 
\[
\Psi(x) := [\psi(x_1) + \psi'(x_1)(x-x_1)] -[\psi(x_0) + \psi'(x_0)(x-x_0)]. 
\]
By the convexity of $\psi$, $\Psi'(x) = \psi'(x_1) - \psi'(x_0) >0$,
so $\Psi$ is increasing. Since
\[
\Psi(x_1) = \psi(x_1)  -\psi(x_0) + \psi'(x_0)(x_1-x_0) = \delta_\psi (x_0,x_1)>0,
\]
we know that
\begin{align*}
0 	
	< \Psi(x_2)
	&=  [\psi(x_1) + \psi'(x_1)(x_2-x_1)] -[\psi(x_0) + \psi'(x_0)(x_2-x_0)]\\
	&=[-\psi(x_2) + \psi(x_1) + \psi'(x_1)(x_2-x_1)] -[- \psi(x_2) + \psi(x_0) + \psi'(x_0)(x_2-x_0)].
\end{align*}
Hence,
\begin{align*}
\delta_{\psi}(x_1,x_2)
	&= \psi(x_2) - \psi(x_1) - \psi'(x_1)(x_2-x_1) \\
	&<  \psi(x_2) - \psi(x_0) - \psi'(x_0)(x_2-x_0) 
	= \delta_{\psi}(x_0,x_2).
\end{align*}

Lastly, fix $x_0,x_1 \in \R$ and $R>0$. Suppose that $r>0$ is such that $S_{\psi}(x_1,r) \subset S_{\psi}(x_0,R)$.
Write
\begin{align*}
S_\psi(x_0,R) = (x_0^L, x_0^R), &\quad x_0^L  < x_0 < x_0^R \\
S_\psi(x_1,r) = (x_1^L, x_1^R), &\quad x_1^L  < x_1 < x_1^R.
\end{align*}
Without loss of generality, assume that $x_0 \leq x_1$ so that $x_0 \leq x_1 < x_1^R \leq x_0^R$.
Then
\begin{align*}
r 	= \delta_\psi(x_1,x_1^R) 
	\leq \delta_\psi(x_1,x_0^R) 
	\leq \delta_\psi(x_0,x_0^R)  = R.
\end{align*}
\end{proof}

\begin{lem}[Theorem 3.3.10 in \cite{Gutierrez}]\label{lem:Guti}
There exist constants $C_0 >0$, $p \geq 1$, depending only on $n$ and $s$, such that for $0 < r_1< r_2 \leq 1$, $t>0$ and $(x_1,z_1) \in S_{r_1t}(x_0,z_0)$, we have that
\[
S_{C_0(r_2-r_1)^{p}t}(x_1,z_1) \subset S_{r_2t}(x_0,z_0).
\]
\end{lem}

%%%%%%%%%%%%%%%%%%%%%%%%%%%%%%%%
\section{Reductions of Theorem \ref{thm:harnack extension}}\label{sec:reductions}
%%%%%%%%%%%%%%%%%%%%%%%%%%%%%%%%

In this section we show that, after a series of reductions,
Theorem \ref{thm:harnack extension} will follow from Theorem \ref{thm:reduction3}.
First, we show in Theorem \ref{thm:reduction1} that it is enough to consider Monge--Amp\`ere cubes,  
instead of Monge--Amp\`ere sections, and to take a nonnegative right hand side $f$.
The second reduction, Theorem \ref{thm:reduction2},  demonstrates how we only need to show that
 the supremum of $U$ in a small cube is controlled by the value of $U$ at the center of the cube.
Finally, Theorem \ref{thm:reduction3} is a normalized statement which says that if 
$U$ is controlled at the center of the cube and $\norm{f}_{L^{\infty}}$ is controlled with respect to the size of the section,  
then $U$ is uniformly bounded in a smaller cube. 
Hence, the rest of the paper will be devoted to proving Theorem \ref{thm:reduction3}
and Theorem \ref{thm:harnack Ls}. 

%%%%%%%%%%%%%%%%%%%%%%%%%%%%%%%%
\subsection{First reduction}
%%%%%%%%%%%%%%%%%%%%%%%%%%%%%%%%

We first show that it is enough to prove Theorem \ref{thm:harnack extension} in Monge--Amp\`ere cubes and with $f \geq 0$.

\begin{thm}\label{thm:reduction1}
Let $\Omega$ be a bounded domain,
$a^{ij}(x):\Omega \to \R$ be bounded, measurable functions that satisfy \eqref{eq:ellipticity} and let $f \in L^{\infty}(\Omega)$ be nonnegative.
There exist positive constants 
$C_H = C_H(n,\lambda,\Lambda, s)>1$, 
$\kappa_1 = \kappa_1(n,s)<1$, 
and $\hat{K}_1 = \hat{K}_1(n,s)>1$
such that for every cube $Q_{\hat{K}_1R} = Q_{\hat{K}_1R}(x_0,z_0) \subset \Omega \times\R$ 
and every nonnegative solution
$U \in C^2(Q_{\hat{K}_1 R} \setminus \{z=0\}) \cap C({Q}_{\hat{K}_1 R})$ such that $U$ is symmetric across $\{z=0\}$ and $U_{z+} \in C(Q_{\hat{K}_1 R} \cap \{z\geq0\})$ to
\[
\begin{cases}
a^{ij}(x) \partial_{ij}U + \abs{z}^{2-\frac{1}{s}}\partial_{zz}U = 0 & \hbox{in}~Q_{\hat{K}_1 R} \cap \{z\not=0\} \\
-\partial_{z+}U(x,0) = f & \hbox{on}~Q_{\hat{K}_1 R} \cap \{z=0\},
\end{cases}
\]
we have that
\begin{equation}\label{eq:harnack-reduction1}
\sup_{Q_{\kappa_1 R}} U 
\leq C_H \( \inf_{Q_{\kappa_1 R}} U + \norm{f}_{L^{\infty}(Q_{\hat{K}_1 R}\cap \{z=0\})}R^s\).
\end{equation}
\end{thm}

\begin{proof}[Proof of Theorem \ref{thm:harnack extension} from Theorem \ref{thm:reduction1}]
Observe that
\begin{align*}
Q_{\hat{K}_1R}(x_0,z_0)
	&\subset Q_{\theta^2\hat{K}_1R}(x_0,z_0)
	\subset S_{(n+1)\theta^2\hat{K}_1R}(x_0,z_0)
\end{align*}
and
\[
S_{\kappa_1R}(x_0,z_0) \subset Q_{\kappa_1R}(x_0,z_0)
\]
Let $\hat{K}_0 = (n+1)\theta^2\hat{K}_1$ and $\kappa_0 = \kappa_1$.

\medskip

\noindent{\bf Case 1}: $Q_{\hat{K}_1R}(x_0,z_0) \cap \{z=0\} = \varnothing$. By Theorem \ref{thm:reduction1} and the inclusion above, we get
\begin{align*}
\sup_{S_{\kappa_0 R}} U
	\leq \sup_{Q_{\kappa_1 R}} U 
	\leq C_H  \inf_{Q_{\kappa_1 R}} U
	\leq C_H  \inf_{S_{\kappa_0 R}} U.
\end{align*}

\medskip

\noindent{\bf Case 2}: $Q_{\hat{K}_1R}(x_0,z_0) \cap \{z=0\} \not= \varnothing$ and $z_0 = 0$.
Define
\[
V = {U} - \norm{f}_{L^{\infty}(Q_{\hat{K}_1R} \cap \{z=0\})}\abs{z}  + \norm{f}_{L^{\infty}(Q_{\hat{K}_1R} \cap \{z=0\})}|S_{\hat{K}_1R}(0)|.
\]

We claim that $V$ is nonnegative in $Q_{\hat{K}_1R}$.
Indeed, let $(x,z) \in Q_{\hat{K}_1R}$, so that $z \in S_{\hat{K}_1R}(0) \subset \R$.
By Remark \ref{rem:r^s new}, $S_{\hat{K}_1R}(0) =B_{q_s\hat{K}_1^sR^s}(0)$.
Thus, $\abs{z} \leq | B_{q_s\hat{K}_1^sR^s}(0)|= |S_{\hat{K}_1R}(0)|$. 
Consequently, for any $(x,z) \in Q_{\hat{K}_1R}$, we have that 
\[
- \norm{f}_{L^{\infty}(Q_{\hat{K}_1R}\cap\{z=0\})}\abs{z} + \norm{f}_{L^{\infty}(Q_{\hat{K}_1R}\cap\{z=0\})}|S_{\hat{K}_1R}(0)|\geq0,
\]
so that $V \geq U \geq 0$  in $Q_{\hat{K}_1R}$.

Next, notice that $V$ is symmetric across $\{z=0\}$ and that
$V \in C^2(Q_{\hat{K}_1 R} \setminus \{z=0\}) \cap C({Q}_{\hat{K}_1 R})$, 
$V_{z+} \in C(Q_{\hat{K}_1 R} \cap \{z\geq0\})$. 
Moreover, for $(x,z) \in Q_{\hat{K}_1R} \cap \{z \ne 0\}$, it is clear that
 \[
a^{ij}(x) \partial_{ij}{V} + \abs{z}^{2-\frac{1}{s}}\partial_{zz} {V} = a^{ij}(x) \partial_{ij}{U} + \abs{z}^{2-\frac{1}{s}}\partial_{zz} {U} = 0
 \]
and for $(x,0) \in Q_{\hat{K}_1R} \cap \{z = 0\}$,
\begin{align*}
-\partial_{z+} V(x,0) = f(x) + \norm{f}_{L^{\infty}(Q_{\hat{K}_1R}\cap\{z=0\})} := g(x)  \geq 0.
 \end{align*}
Therefore, $V$ is a nonnegative solution to 
\[
\begin{cases}
a^{ij}(x) \partial_{ij}{V} + \abs{z}^{2-\frac{1}{s}}\partial_{zz} {V} = 0 & \hbox{in}~Q_{\hat{K}_1R} \cap \{z \ne 0\}\\
-\partial_{z+}{V}=g & \hbox{on}~Q_{\hat{K}_1R} \cap \{z = 0\}.
\end{cases}
\]

Therefore, by Theorem \ref{thm:reduction1} applied to $V$ and  using that $|S_{\hat{K}_1R}(0)| = | B_{q_s\hat{K}_1^sR^s}(0)| = C_{n,s}R^s$,
\begin{align*}
\sup_{S_{\kappa_0 R}}{U}
&\leq \sup_{Q_{\kappa_1 R}}{U}\\
&\leq \sup_{Q_{\kappa_1 R}}V\\
&\leq C_H \( \inf_{Q_{\kappa_1 R}} V + \norm{g}_{L^{\infty}(Q_{\hat{K}_1 R}\cap \{z=0\})}R^s\)\\
&= C_H \bigg(  \inf_{Q_{\kappa_1 R}} \({U} - \norm{f}_{L^{\infty}(Q_{\hat{K}_1R \cap \{z=0\}})}\abs{z}  + \norm{f}_{L^{\infty}(Q_{\hat{K}_1R}\cap \{z=0\})}|S_{\hat{K}_1 R}(0)|\) \\
	&\qquad \qquad+ \norm{f +  \norm{f}_{L^{\infty}(Q_{\hat{K}_1R} \cap \{z=0\})}}_{L^{\infty}(Q_{\hat{K}_1 R}\cap \{z=0\})}R^s\bigg)\\
	&\leq C_H \bigg(  \inf_{Q_{\kappa_1 R}} {U}  +(C_{n,s} +2)\norm{f}_{L^{\infty}(Q_{\hat{K}_1R}\cap\{z=0\})}R^s \bigg)	\\
	&\leq C_H' \bigg(  \inf_{S_{\kappa_0 R}} {U}  +\norm{f}_{L^{\infty}(S_{\hat{K}_0R}\cap\{z=0\})}R^s \bigg).
\end{align*}

\medskip

\noindent{\bf Case 3}: $Q_{\hat{K}_1R}(x_0,z_0) \cap \{z=0\} \not= \varnothing$ and $z_0 \not= 0$.
In this case, $0 \in S_{\hat{K}_1R}(z_0)$. 
Then, by the engulfing property,
\begin{align*}
Q_{\hat{K}_1R}(x_0,z_0)
	&= Q_{\hat{K}_1R}({x_0}) \times S_{\hat{K}_1R}({z}_0)\\
	&\subset Q_{\theta\hat{K}_1R}({x}_0) \times S_{\theta\hat{K}_1R}(0)
	= Q_{\theta\hat{K}_1R}({x}_0,0).
\end{align*}
Again, applying the engulfing property,
\begin{align*}
Q_{\theta\hat{K}_1R}({x_0},0)
	& =Q_{\theta\hat{K}_1R}({x_0}) \times S_{\theta\hat{K}_1R}(0)\\
	 &\subset Q_{\theta^2\hat{K}_1R}({x_0}) \times S_{\theta^2\hat{K}_1R}({z_0})
	= Q_{\theta^2\hat{K}_1R}(x_0,z_0).
\end{align*}

Define
\[
V = {U} - \norm{f}_{L^{\infty}(Q_{\theta\hat{K}_1R}(x_0,0)\cap\{z=0\})}\abs{z}  + \norm{f}_{L^{\infty}(Q_{\theta\hat{K}_1R}({x}_0,0)\cap\{z=0\})} |S_{\theta\hat{K}_1 R}(0)|.
\] 

We claim that $V$ is nonnegative in $Q_{\theta\hat{K}_1R}(x_0,0)$.
Let $(x,z) \in Q_{\theta\hat{K}_1R}(x_0,0)$, so that, by Remark \ref{rem:r^s new}, $z \in S_{\theta \hat{K}_1R}(0) = B_{q_s\theta^s \hat{K}_1^sR^s}(0) \subset \R$. In particular, 
 $\abs{z} \leq | B_{q_s\theta^s \hat{K}_1^sR^s}(0)|= |S_{\hat{K}_1\theta R}(0)|$.
Consequently, for any $(x,z) \in Q_{\hat{K}_1R}(x_0,z_0)$, we have that 
\[
- \norm{f}_{L^{\infty}(Q_{\theta\hat{K}_1R}(x_0,0)\cap\{z=0\})}\abs{z}  + \norm{f}_{L^{\infty}(Q_{\theta\hat{K}_1R}({x}_0,0)\cap\{z=0\})}|S_{\theta\hat{K}_1 R}(0)|
\geq0,
\]
so that $V \geq U \geq 0$ in $Q_{\theta\hat{K}_1R}(x_0,0)$.

Next, notice that $V$ is symmetric across $\{z=0\}$ and that
$V \in C^2(Q_{\theta\hat{K}_1R}(x_0,0) \setminus \{z=0\}) \cap C(Q_{\theta\hat{K}_1R}(x_0,0) )$, 
$V_{z+} \in C(Q_{\theta\hat{K}_1R}(x_0,0)  \cap \{z\geq0\})$. 
Moreover, for $(x,z) \in Q_{\theta\hat{K}_1R}(x_0,0)  \cap \{z \ne 0\}$, it is clear that
 \[
a^{ij}(x) \partial_{ij}{V} + \abs{z}^{2-\frac{1}{s}}\partial_{zz} {V} =  a^{ij}(x) \partial_{ij}{U} + \abs{z}^{2-\frac{1}{s}}\partial_{zz} {U}=0
 \]
and for $(x,0) \in Q_{\theta\hat{K}_1R}(x_0,0)  \cap \{z = 0\}$,
\begin{align*}
-\partial_{z+} V(x,0)
 = f(x) + \norm{f}_{L^{\infty}(Q_{\theta\hat{K}_1R}(x_0,0) \cap\{z=0\})} := g(x)  \geq 0.
 \end{align*}
Therefore, $V$ is a nonnegative solution to 
\[
\begin{cases}
a^{ij}(x) \partial_{ij}{V} + \abs{z}^{2-\frac{1}{s}}\partial_{zz} {V} = 0 & \hbox{in}~Q_{\hat{K}_1R}(x_0,z_0)  \cap \{z \ne 0\}\\
-\partial_{z+}{V}=g & \hbox{on}~Q_{\hat{K}_1R}(x_0,z_0)  \cap \{z = 0\}.
\end{cases}
\]

Applying Theorem \ref{thm:reduction1} to $V$
and using that $|S_{\theta\hat{K}_1R}(0)| = | B_{q_s\theta^s\hat{K}_1^sR^s}(0)| = C_{n,s}R^s$, 
 we get
\begin{align*}
\sup_{S_{\kappa_0 R}(x_0,z_0)}{U}
&\leq \sup_{Q_{\kappa_1 R}(x_0,z_0)}{U}\\
&\leq \sup_{Q_{\kappa_1 R}(x_0,z_0)}V \\
&\leq C_H \( \inf_{Q_{\kappa_1 R}(x_0,z_0)} V + \norm{g}_{L^{\infty}(Q_{\hat{K}_1 R}(x_0,z_0)\cap \{z=0\})}R^s\)\\
&= C_H \bigg( \inf_{Q_{\kappa_1 R}(x_0,z_0)} \({U} - \norm{f}_{L^{\infty}(Q_{\theta\hat{K}_1R}(x_0,0))}\abs{z}  +\norm{f}_{L^{\infty}(Q_{\theta\hat{K}_1R}(x_0,0))}|S_{\theta\hat{K}_1 R}(0)|\)\\
&\qquad\qquad + \norm{f + \norm{f}_{L^{\infty}(Q_{\theta\hat{K}_1R}(x_0,0) \cap \{z=0\})}}_{L^{\infty}(Q_{\hat{K}_1 R}(x_0,z_0)\cap \{z=0\})}R^s\bigg)\\
&\leq C_H \bigg( \inf_{Q_{\kappa_1 R}(x_0,z_0)} U+ (C_{n,s}+2)\norm{f}_{L^{\infty}(Q_{\theta\hat{K}_1 R}(x_0,0)\cap \{z=0\})}R^s\bigg)\\
&\leq C_H' \bigg( \inf_{S_{\kappa_0 R}(x_0,z_0)} U+ \norm{f}_{L^{\infty}(S_{\hat{K}_0 R}(x_0,z_0)\cap \{z=0\})}R^s\bigg).
\end{align*}

Therefore, \eqref{eq:extension Harnack3} holds in all cases.

It remains to prove the H\"older estimate \eqref{eq:extension Holder}. 
The proof follows by a standard argument (see, for example, \cite[Sections 8.9 and 9.8]{GilbargTrudinger}). We provide the details for completeness. 
Let $0 < r \leq R$ and define
\[
M(r) = \sup_{S_r(x_0,z_0)} U \quad \hbox{and} \quad m(r) = \inf_{S_r(x_0,z_0)}U.
\]
Apply \eqref{eq:extension Harnack3} to $M(\hat{K}_0r) - U \geq 0$ in $S_{\hat{K}_0r}(x_0,z_0)$ to obtain
\[
\sup_{S_{\kappa_0r}(x_0,z_0)}(M(\hat{K}_0r) -U) 
	\leq C_H \(\inf_{S_{\kappa_0 r}(x_0,z_0)}(M(\hat{K}_0r)-U) 
	+\norm{f}_{L^{\infty}(S_{\hat{K}_0 r}(x_0,z_0)\cap \{z=0\})}r^s\).
\]
Therefore,
\begin{equation}\label{Mr1}
M(\hat{K}_0r) - m(\kappa_0 r) \leq C_H \(M(\hat{K}_0r) - M(\kappa_0 r)+\norm{f}_{L^{\infty}(S_{\hat{K}_0 R}(x_0,z_0)\cap \{z=0\})}r^s\).
\end{equation}
Similarly, apply \eqref{eq:extension Harnack3} to $U- m(\hat{K}_0r) \geq 0$ in $S_{\hat{K}_0r}(x_0,z_0)$ to obtain 
\[
\sup_{S_{\kappa_0 r}(x_0,z_0)}(U -m(\hat{K}_0r)) \leq C_H\( \inf_{S_{\kappa_0 r}(x_0,z_0)}(U -m( \hat{K}_0r)+\norm{f}_{L^{\infty}(S_{\hat{K}_0 r}(x_0,z_0)\cap \{z=0\})}r^s\),
\]
so that
\begin{equation}\label{Mr2}
M(\kappa_0 r) -m(\hat{K}_0r) \leq C_H \(m(\kappa_0r) -m(\hat{K}_0r)+\norm{f}_{L^{\infty}(S_{\hat{K}_0 R}(x_0,z_0)\cap \{z=0\})}r^s\).
\end{equation}
Let $\w(r)= M(r) - m(r)$. Adding \eqref{Mr1} and \eqref{Mr2} together, we get
\begin{align*}
\w(\hat{K}_0r)+\w(\kappa_0  r)
	\leq C_H \(\w(\hat{K}_0r) - \w(\kappa_0r)+2\norm{f}_{L^{\infty}(S_{\hat{K}_0 R}(x_0,z_0)\cap \{z=0\})}r^s\).
\end{align*}
After rearranging,
\begin{align*}
\w(\kappa_0 r)
	&\leq \gamma \w(\hat{K}_0r)+\sigma(r), \quad \gamma = \frac{C_H-1}{C_H+1}<1
\end{align*}
where 
\[
\sigma(r) =\begin{cases}
\frac{2C_H}{C_H+1}\norm{f}_{L^{\infty}(S_{\hat{K}_0 R}(x_0,z_0)\cap \{z=0\})}r^s & \hbox{if}~S_{\hat{K}_0 R}(x_0,z_0)\cap \{z=0\} \not= \varnothing \\
0 & \hbox{if}~S_{\hat{K}_0 R}(x_0,z_0)\cap \{z=0\} = \varnothing
\end{cases} 
\]
is a non-decreasing function of $r$.
Note that $\gamma = \gamma(n,\lambda,\Lambda,s)$.
By \cite[Lemma 8.23]{GilbargTrudinger}, for any $\mu \in (0,1)$, 
there are constants $C = C(n,\lambda,\Lambda,s)>0$ and  $\alpha_1=(1-\mu)\log \gamma/ \log (\kappa_0/\hat{K}_0)$ such that
\begin{align*}
\w(\hat{K}_0r)  \leq C\(\(\frac{r}{R}\)^{\alpha_1}\w(\hat{K}_0R) + \sigma(r^{\mu}R^{1-\mu}) \).
\end{align*}
Choose $\mu = \mu(n,\lambda, \Lambda, s)$ so that $2\alpha_1 < \mu s$. Then,
\begin{align*}
\w( \hat{K}_0r)  
	&\leq C\(\frac{r}{ R}\)^{\alpha_1}\(\w(\hat{K}_0R) + \sigma(r^{\mu}R^{1-\mu})\( \frac{R}{r}\)^{\alpha_1} \)\\
	&\leq C\(\frac{r}{ R}\)^{\alpha_1}\( 2\sup_{S_{\hat{K}_0R}(x_0,z_0)} \abs{U}
		+ \frac{2C_H}{C_H+1}\norm{f}_{L^{\infty}(S_{\hat{K}_0 R}(x_0,z_0)\cap \{z=0\})}{r^{\mu s}R^{(1-\mu)s}R^{\alpha_1} r^{-\alpha_1}}\)\\
	&\leq \hat{C}_1\(\frac{\hat{K}_0r}{\hat{K}_0R}\)^{\alpha_1}\( \sup_{S_{\hat{K}_0R}(x_0,z_0)} \abs{U}
		+ \norm{f}_{L^{\infty}(S_{\hat{K}_0 R}(x_0,z_0)\cap \{z=0\})}R^s\).
\end{align*}
By taking $\hat{K}_0r = \delta_{\Phi}((x_0,z_0),(x,z))$, the estimate in \eqref{eq:extension Holder} follows.
\end{proof}

%%%%%%%%%%%%%%%%%%%%%%%%%%%%%%%%
\subsection{Second reduction}
%%%%%%%%%%%%%%%%%%%%%%%%%%%%%%%%

Next, we show that Theorem \ref{thm:reduction1} follows from the following result 
which says that the supremum of $U$ in a small cube can be controlled by its value at the center. 

\begin{thm}\label{thm:reduction2}
Let $\Omega$ be a bounded domain,
$a^{ij}(x):\Omega \to \R$ be bounded, measurable functions that satisfy \eqref{eq:ellipticity} and let $f \in L^{\infty}(\Omega)$ be nonnegative.
There exist positive constants 
$C_H = C_H(n,\lambda,\Lambda, s)>1$,
$\kappa_2 = \kappa_2(n,s)<1$, and $\hat{K}_3=\hat{K}_3(n,s)>1$ such
that for every cube $Q_{\hat{K}_3R} = Q_{\hat{K}_3R}(\tilde{x},\tilde{z}) \subset \subset \Omega \times \R$
and every nonnegative solution 
$U \in C^2(Q_{\hat{K}_3 R} \setminus \{z=0\}) \cap C({Q}_{\hat{K}_3 R})$ such that $U$ is symmetric across $\{z=0\}$ and  $U_{z+} \in C(Q_{\hat{K}_3 R} \cap \{z\geq0\})$ to
\[
\begin{cases}
a^{ij}(x) \partial_{ij}U + \abs{z}^{2-\frac{1}{s}}\partial_{zz}U = 0 & \hbox{in}~Q_{\hat{K}_3R} \cap \{z\not=0\} \\
-\partial_{z+}U(x,0) = f & \hbox{on}~Q_{\hat{K}_3R} \cap \{z=0\},
\end{cases}
\]
we have that
\[
\sup_{Q_{\kappa_2 R}} U \leq C_H \(U(\tilde{x},\tilde{z}) + \norm{f}_{L^{\infty}(Q_{\hat{K}_3R}\cap \{z=0\})}R^s\).
\]
\end{thm}

\begin{proof}[Proof of Theorem \ref{thm:reduction1} from Theorem \ref{thm:reduction2}]
Let $\hat{K}_1 = \hat{K}_1(n,s)$ and $\kappa_1 = \kappa_1(n,s)$ be such that
\[
1 < \theta \hat{K}_3  \leq \hat{K}_1
\quad\hbox{and} \quad
\theta \kappa_1 \leq \kappa_2 < 1. 
\]
Let $(\tilde{x},\tilde{z}) \in Q_{\kappa_1 R}(x_0,z_0)$.
By the engulfing property,
\[
Q_{\kappa_1 R}(x_0,z_0)
	\subset Q_{\theta \kappa_1 R}(\tilde{x},\tilde{z})
	\subset Q_{\kappa_2 R}(\tilde{x},\tilde{z}).
\]
Again applying the engulfing property, we have
\begin{align*}
Q_{\hat{K}_3R}(\tilde{x},\tilde{z})
	&\subset Q_{\theta \hat{K}_3 R}(x_0,z_0) 
	\subset Q_{\hat{K}_1R}(x_0,z_0) 
	\subset \subset \Omega \times \R.
\end{align*}
By Theorem \ref{thm:reduction2}, we get
\begin{align*}
\sup_{Q_{\kappa_1 R}(x_0,z_0)}U
	&\leq \sup_{Q_{\kappa_2 R}(\tilde{x},\tilde{z})} U\\
	&\leq C_H\(U(\tilde{x},\tilde{z}) + \norm{f}_{L^{\infty}(Q_{\hat{K}_3R}(\tilde{x},\tilde{z})\cap \{z=0\})}R^s\)\\
	&\leq C_H\(U(\tilde{x},\tilde{z}) + \norm{f}_{L^{\infty}(Q_{\hat{K}_1R}(x_0,z_0)\cap \{z=0\})}R^s\).
\end{align*}
Taking the infimum over all $(\tilde{x},\tilde{z}) \in Q_{\kappa_1 R}(x_0,z_0)$, the Harnack inequality \eqref{eq:harnack-reduction1}  holds.
\end{proof}

%%%%%%%%%%%%%%%%%%%%%%%%%%%%%%%%
\subsection{Third reduction}
%%%%%%%%%%%%%%%%%%%%%%%%%%%%%%%%

Here we will see that Theorem \ref{thm:reduction2} follows from the next, and final, reduction
which is a normalized statement.

\begin{thm}\label{thm:reduction3}
Fix $a>0$. 
Let $\Omega$ be a bounded domain,
$a^{ij}(x):\Omega \to \R$ be bounded, measurable functions that satisfy \eqref{eq:ellipticity} and let $f \in L^{\infty}(\Omega)$ be nonnegative.
Let $\hat{K}_3$ be as in Theorem \ref{thm:reduction2}.
There exist positive constants 
$C_H= C_H(n,\lambda,\Lambda, s)>1$, 
$\kappa_2 = \kappa_2(n,s)<1$, and $K_0=K_0(n,s)$
such
that for any cube $Q_{\hat{K}_3R} = Q_{\hat{K}_3R}(\tilde{x},\tilde{z}) \subset \subset \Omega \times \R$
and every nonnegative solution
$U \in C^2(Q_{\hat{K}_3 R} \setminus \{z=0\}) \cap C({Q}_{\hat{K}_3 R})$ such that  
$U$ is symmetric across $\{z=0\}$ and $U_{z+} \in C(Q_{\hat{K}_3 R} \cap \{z\geq0\})$ to
\[
\begin{cases}
a^{ij}(x) \partial_{ij}U + \abs{z}^{2-\frac{1}{s}}\partial_{zz}U = 0 & \hbox{in}~Q_{\hat{K}_3R} \cap \{z\not=0\} \\
-\partial_{z+}U(x,0) = f & \hbox{on}~Q_{\hat{K}_3R} \cap \{z=0\},
\end{cases}
\]
if
\[
U(\tilde{x},\tilde{z}) \leq \frac{aR}{2K_0}
\]
and 
\[
\norm{f}_{L^{\infty}(Q_{\hat{K}_3R} \cap\{z=0\})} \leq a\mu_h(S_{\hat{K}_3R}(\tilde{z})),
\]
then
\begin{equation}\label{eq:harnack-finalreduction}
U \leq C_H aR \quad \hbox{in}~Q_{\kappa_2 R}.
\end{equation}
\end{thm}

\begin{proof}[Proof of Theorem \ref{thm:reduction2} from Theorem \ref{thm:reduction3} ]
Let $\varepsilon >0$. Define the nonnegative function $W_{\varepsilon}$ in $Q_{\hat{K}_3R}$ by
\[
W_{\varepsilon}(x,z) = \frac{aR}{2K_0U(\tilde{x},\tilde{z}) + \norm{f}_{L^{\infty}(Q_{\hat{K}_3R}\cap\{z=0\})} R/\mu_h(S_{\hat{K}_3R}(\tilde{z})) +\varepsilon } U(x,z).
\]
If $0 \notin S_{\hat{K}_3R}(\tilde{z})$, then $ \norm{f}_{L^{\infty}(Q_{\hat{K}_3R}\cap\{z=0\})} R/\mu_h(S_{\hat{K}_3R}(\tilde{z})) =0$. 
Notice that $W_{\varepsilon}$ is symmetric across $\{z=0\}$ and 
$W_{\varepsilon} \in C^2(Q_{\hat{K}_3 R} \setminus \{z=0\}) \cap C({Q}_{\hat{K}_3 R})$, $(W_{\varepsilon})_{z+} \in C(Q_{\hat{K}_3 R} \cap \{z\geq0\})$.
Moreover, 
in $Q_{\hat{K}_3R} \cap \{z\not=0\}$, we have
$$a^{ij}(x) \partial_{ij}{W}_{\varepsilon} + \abs{z}^{2-\frac{1}{s}}\partial_{zz} {W}_{\varepsilon}= 0$$
and, in $Q_{\hat{K}_3R} \cap \{z=0\}$, 
\[
-\partial_{z+} W_{\varepsilon}(x,0)
	= \frac{aR}{2K_0U(\tilde{x},\tilde{z}) + \norm{f}_{L^{\infty}(Q_{\hat{K}_3R}\cap\{z=0\})}R/\mu_h(S_{\hat{K}_3R}(\tilde{z})) +\varepsilon }  f(x)=: g(x)\geq 0.
\]
Therefore, $W_{\varepsilon}$ is a nonnegative solution to 
\[
\begin{cases}
a^{ij}(x) \partial_{ij}W_{\varepsilon} + \abs{z}^{2-\frac{1}{s}}\partial_{zz}W_{\varepsilon} = 0 & \hbox{in}~Q_{\hat{K}_3R} \cap \{z\not=0\} \\
-\partial_{z+}W_{\varepsilon}(x,0) = g & \hbox{on}~Q_{\hat{K}_3R} \cap \{z=0\}.
\end{cases}
\]
Clearly,
\begin{align*}
\norm{g}_{L^{\infty}(Q_{\hat{K}_3R}\cap\{z=0\})}\leq a \mu_h(S_{\hat{K}_3R}(\tilde{z}))
\end{align*}
and
\begin{align*}
W_{\varepsilon}(\tilde{x},\tilde{z})\leq \frac{aR}{2K_0}.
\end{align*}
By Theorem \ref{thm:reduction3} applied to $W_{\varepsilon}$, we get 
\[
W_{\varepsilon} \leq C_HaR \quad \hbox{in}~Q_{\kappa_2 R}
\]
which implies
\begin{align*}
\sup_{Q_{\kappa_2 R}}U\leq C_H'\(U(\tilde{x},\tilde{z}) + \norm{f}_{L^{\infty}(Q_{\hat{K}_3R}\cap\{z=0\})} \frac{R}{\mu_h(S_{\hat{K}_3R}(\tilde{z}))} +\varepsilon \).
\end{align*}
If $0 \notin S_{\hat{K}_3R}(\tilde{z})$, then 
\begin{align*}
\sup_{Q_{\kappa_2 R}}U\leq C_H'\(U(\tilde{x},\tilde{z}) +\varepsilon \).
\end{align*}
If $0 \in S_{\hat{K}_3R}(\tilde{z})$, then by the engulfing property and Remark \ref{rem:r^s new}, 
\[
S_{\hat{K}_3R}(\tilde{z}) \subset S_{\theta \hat{K}_3R}(0) = B_{q_s \theta^s \hat{K}_3^sR^s}(0).
\]
Hence, $| S_{\hat{K}_3R}(\tilde{z})| \leq C R^s$ for some $C = C(n,s)>0$. 
With this and Corollary \ref{lem:doubling} part (3), 
\begin{align*}
\sup_{Q_{\kappa_2 R}}U
	&\leq C_H'\(U(\tilde{x},\tilde{z}) + \norm{f}_{L^{\infty}(Q_{\hat{K}_3R}\cap\{z=0\})} \frac{1}{\hat{K}_3}\frac{\hat{K}_3R}{\mu_h(S_{\hat{K}_3R}(\tilde{z}))} +\varepsilon \)\\
	&\leq C_H'\(U(\tilde{x},\tilde{z}) + \norm{f}_{L^{\infty}(Q_{\hat{K}_3R}\cap\{z=0\})} \frac{1}{c\hat{K}_3} |S_{\hat{K}_3R}(\tilde{z})| +\varepsilon \)\\
	&\leq C_H''\(U(\tilde{x},\tilde{z}) + \norm{f}_{L^{\infty}(Q_{\hat{K}_3R}\cap\{z=0\})}R^s+\varepsilon\).
\end{align*}
As $\varepsilon>0$ was arbitrary, the conclusion follows.
\end{proof}

%%%%%%%%%%%%%%%%%%%%%%%%%%%%%%%%
\section{Paraboloids associated to $\Phi$}\label{sec:paraboloids}
%%%%%%%%%%%%%%%%%%%%%%%%%%%%%%%%

In this section, we define the Monge--Amp\`ere paraboloids associated with $\Phi$ in \eqref{eq:capitalphi} and study their
basic properties and relations with respect to solutions to the extension problem.

\begin{defn}
Let $a>0$. A \emph{paraboloid} $P$ of opening $a$ in $\R^{n+1}$ is defined as
\begin{equation}\label{eq:P defn}
P(x,z) = -a \Phi(x,z) + \la (y,w), (x,z) \ra + b\qquad(x,z)\in\R^{n+1}
\end{equation}
for some $(y,w) \in \R^{n+1}$ and $b \in \R$. 
\end{defn}

Since $\Phi \in C^1(\R^{n+1})$ is strictly convex, the point where the maximum of $P$ occurs, which we call the \emph{vertex}
$(x_v,z_v)$ of $P$, is the unique solution to $\nabla P(x_v,z_v) = 0$. 

We say that $P$ \emph{touches} a continuous function $U: \R^{n+1} \to \R$ \emph{from below} at $(x_0,z_0)$ in a convex set $S \subset \R^{n+1}$ if 
\[
P(x_0,z_0) = V(x_0,z_0) \quad \hbox{and} \quad
P(x,z) \leq V(x,z) \quad \hbox{for all}~(x,z) \in S.
\]

\begin{lem}\label{lem:PintersectsU}
A paraboloid $P$ of opening $a>0$ with vertex $(x_v,z_v)$ given by \eqref{eq:P defn} can be written as
\begin{equation}\label{eq:P defn2}
P(x,z) = -a\delta_{\Phi}((x_v,z_v),(x,z)) + c
\end{equation}
for some constant $c \in \R$. 
Moreover, 
\begin{align*}
\nabla P(x,z) 
	&= -a(\nabla\Phi(x,z) - \nabla\Phi(x_v,z_v))
	= -a(x-x_v,h'(z) - h'(z_v))
\end{align*}
and
\[
\partial_zP(x,0) = ah'(z_v) 
\]
for all $(x,z) \in \R^{n+1}$.
If $P$ coincides with a continuous function  $U: \R^{n+1} \to \R$ at a point $(x_0,z_0)$, i.e. $P(x_0,z_0) = U(x_0,z_0)$, then 
\[
P(x,z) = -a\delta_{\Phi}((x_v,z_v),(x,z)) + a\delta_{\Phi}((x_v,z_v),(x_0,z_0)) + U(x_0,z_0).
\]
\end{lem}

\begin{proof}
Since $0 = \nabla P(x_v,z_v) = -a\nabla\Phi(x_v,z_v) + (y,w)$, we can write
\begin{align*}
P(x,z) 
	&= -a \Phi(x,z) +a\la \nabla\Phi(x_v,z_v) ,(x,z) \ra + b.
\end{align*}
Consequently, 
\begin{align*}
\nabla P(x,z)= -a(x-x_v,h'(z) - h'(z_v))
\end{align*}
and
\begin{align*}
\partial_z P(x,0)
	&=-a\(h'(z) - h'(z_v)\) \big|_{(x,0)}
	= ah'(z_v).
\end{align*}
Moreover, we have
\begin{align*}
P(x,z) 
&= -a \Phi(x,z)  +a\la \nabla\Phi(x_v,z_v) ,(x,z) \ra  +b \\
	&\qquad+ a\Phi(x_v,z_v) - a\Phi(x_v,z_v) - a\la \nabla\Phi(x_v,z_v),(x_v,z_v) \ra + a\la  \nabla\Phi(x_v,z_v),(x_v,z_v) \ra\\
&= -a\(\Phi(x,z) - \Phi(x_v,z_v) - \la \nabla\Phi(x_v,z_v), (x,z)-(x_v,z_v) \ra \)\\
	&\qquad -a\Phi(x_v,z_v) +a\la  \nabla\Phi(x_v,z_v),(x_v,z_v) \ra  + b\\
&=  -a\delta_{\Phi}((x_v,z_v),(x,z)) + c.
\end{align*}

If $P(x_0,z_0) = U(x_0,z_0)$, then
$U(x_0,z_0)= -a\delta_{\Phi}((x_v,z_v),(x_0,z_0)) + c$ and,
after solving for $c$, we conclude that
$P(x,z) = -a\delta_{\Phi}((x_v,z_v),(x,z)) + a\delta_{\Phi}((x_v,z_v),(x_0,z_0)) + U(x_0,z_0)$.
\end{proof}

For the remainder of the paper, we use the terminology \emph{paraboloids} to reference those given by \eqref{eq:P defn}, or equivalently, \eqref{eq:P defn2}.

\begin{lem}\label{lem:narrower opening}
Suppose that $P$ is a paraboloid of opening $a>0$ that touches a continuous function $U:\R^{n+1} \to \R$ from below at $(x_0,z_0)$ in a convex set $S \subset \R^{n+1}$.
For any $\tilde{a} \geq a$, there exists a paraboloid $\tilde{P}$ of opening $\tilde{a}>0$ that touches $U$ from below at $(x_0,z_0)$ in $S$.
\end{lem}

\begin{proof}
Begin by writing
\begin{align*}
P(x,z) &= -a\delta_{\Phi}((x_v,z_v),(x,z)) + a\delta_{\Phi}((x_v,z_v),(x_0,z_0)) + U(x_0,z_0)\\
	&=-a\( \Phi(x,z) -\Phi(x_v,z_v) - \la \nabla\Phi(x_v,z_v), (x,z) - (x_v,z_v) \ra \)\\
		&\quad+ a\( \Phi(x_0,z_0) -\Phi(x_v,z_v) - \la \nabla\Phi(x_v,z_v), (x_0,z_0) - (x_v,z_v) \ra \) + U(x_0,z_0)\\
	&=-a \Phi(x,z)+ a\la \nabla\Phi(x_v,z_v), (x,z) \ra 
		+ a\Phi(x_0,z_0) - a\la \nabla\Phi(x_v,z_v), (x_0,z_0) \ra + U(x_0,z_0)\\
	&= -a \( \Phi(x,z) - \Phi(x_0,z_0) - \la \nabla\Phi(x_0,z_0), (x,z)- (x_0,z_0) \ra \)\\
		&\quad+ a\la \nabla\Phi(x_v,z_v), (x,z) - (x_0,z_0) \ra - a\la \nabla\Phi(x_0,z_0), (x,z)- (x_0,z_0) \ra  + U(x_0,z_0)\\
	&= -a \delta_{\Phi}((x_0,z_0),(x,z))+ a\la \nabla\Phi(x_v,z_v) - \nabla\Phi(x_0,z_0), (x,z) - (x_0,z_0) \ra+ U(x_0,z_0).
\end{align*}
Define $\tilde{P}$ by
\[
\tilde{P}(x,z) = -\tilde{a} \delta_{\Phi}((x_0,z_0),(x,z))+ a\la \nabla\Phi(x_v,z_v) - \nabla\Phi(x_0,z_0), (x,z) - (x_0,z_0) \ra + U(x_0,z_0).
\]
Note that $\tilde{P}$ is a paraboloid of opening $\tilde{a}>0$ since it can be expressed as \eqref{eq:P defn} in the following way
\begin{align*}
\tilde{P}(x,z) 
	&= -\tilde{a}\Phi(x,z) +\tilde{a} \Phi(x_0,z_0) + \tilde{a} \la \nabla\phi(x_0,z_0),(x,z) - (x_0,z_0) \ra\\
	&\quad+ a\la \nabla\Phi(x_v,z_v) - \nabla\Phi(x_0,z_0), (x,z) - (x_0,z_0) \ra + U(x_0,z_0)\\
	&= -\tilde{a}\Phi(x,z) + \la  \tilde{a} \nabla\Phi(x_0,z_0) +a\nabla\Phi(x_v,z_v) - a\nabla\Phi(x_0,z_0),(x,z) \ra\\
	&\quad+\tilde{a} \Phi(x_0,z_0) - \la \tilde{a}\nabla\phi(x_0,z_0)+ a\nabla\Phi(x_v,z_v) - a\nabla\Phi(x_0,z_0), (x_0,z_0) \ra + U(x_0,z_0).
\end{align*}
Since $\tilde{P}(x_0,z_0) = U(x_0,z_0)$ and 
\begin{align*}
\tilde{P}(x,z)
&\leq -{a} \delta_{\Phi}((x_0,z_0),(x,z))+ a\la \nabla\Phi(x_v,z_v) - \nabla\Phi(x_0,z_0), (x,z) - (x_0,z_0) \ra + U(x_0,z_0)\\
&= P(x,z) \leq U(x,z),
\end{align*}
for every $(x,z) \in S$, we conclude that $\tilde{P}$ touches $U$ from below at $(x_0,z_0)$ in $S$.
\end{proof}

The next two lemmas provide some observations regarding how the symmetry of $U$
across $\{z=0\}$ effects the geometry of the paraboloids that touch $U$ from below. 

\begin{lem}\label{lem:pos vertex}
Let $S \subset \R^{n+1}$ be an open, convex set that is symmetric across $\{z=0\}$. 
Consider a continuous function $U : S \to \R$ which is symmetric across $\{z=0\}$. 
Let $P$ be a paraboloid of opening $a>0$ with vertex $(x_v,z_v)$ that touches $U$ from below at $(x_0,z_0)$ in $S$.
If $z_0 > 0$, then $z_v \geq 0$, and 
if $z_0 < 0$, then $z_v \leq 0$.
Moreover, the paraboloid $\tilde{P}(x,z) = P(x,-z)$ of opening $a>0$ with vertex $(x_v,-z_v)$ that touches $U$ from below at $(x_0,-z_0)$ in $S$.
\end{lem}

\begin{proof}
Assume that $z_0>0$. 
Write
\begin{align*}
P(x,z)
	&= -a \delta_\Phi((x_v,z_v),(x,z)) +a \delta_\Phi((x_v,z_v),(x_0,z_0)) +U(x_0,z_0)
\end{align*}
and note that
\begin{align*}
P(x_0, -z_0)
	&= -a \delta_\Phi((x_v,z_v),(x_0,-z_0)) +a \delta_\Phi((x_v,z_v),(x_0,z_0)) +U(x_0,z_0)\\
	&= -a \delta_h(z_v,-z_0) +a \delta_h(z_v,z_0) +U(x_0,-z_0).
\end{align*}
Then
\begin{align*}
0 
	&\leq U(x_0,-z_0) - P(x_0, -z_0)\\
	&= a \delta_h(z_v,-z_0) -a \delta_h(z_v,z_0) \\
	&= a \(h(-z_0)  - h(z_0)\) +2ah'(z_v)z_0= 2ah'(z_v)z_0.
\end{align*}
Since $z_0 >0$, it follows that $h'(z_v)\geq0$. Hence, $z_v\geq 0$, as desired. 
The case for $z_0<0$ follows similarly. 

Next, define $\tilde{P}$ by $\tilde{P}(x,z) = P(x,-z)$. 
Since $h'(-z_1)= -h'(z_1)$ and
\begin{equation}\label{eq:delta_h anti-symmetry}
\begin{aligned}
\delta_h(z_1,-z_2)
	&= h(-z_2) - h(z_1) - h'(z_1)(-z_2-z_1) \\
	&= h(z_2) - h(-z_1) - h'(-z_1)(z_2-(-z_1)) \\
	&=\delta_h(-z_1,z_2)
\end{aligned}
\end{equation}
for all $z_1,z_2 \in \R$, we may write
\begin{align*}
\tilde{P}(x,z)
	&= P(x,-z)\\
	&= -a \delta_\Phi((x_v,z_v),(x,-z)) + a \delta_\Phi((x_v,z_v),(x_0,z_0))  + U(x_0,z_0) \\
	&= -a \delta_\varphi(x_v,x) + a \delta_\varphi(x_v,x_0) 
		-a\delta_h(z_v,-z) + a \delta_h(z_v,z_0)  + U(x_0,-z_0) \\
	&= -a \delta_\varphi(x_v,x) + a \delta_\varphi(x_v,x_0) 
		-a\delta_h(-z_v,z) + a \delta_h(-z_v,-z_0)  + U(x_0,-z_0) \\
	&= -a \delta_\Phi((x_v,-z_v),(x,z)) + a \delta_\Phi((x_v,-z_v),(x_0,-z_0)) + U(x_0,-z_0).
\end{align*}
Hence, $\tilde{P}$ is a paraboloid of opening $a>0$ with vertex $(x_v,-z_v)$.
Since
\[
\tilde{P}(x,z) = P(x,-z) \leq U(x,-z) = U(x,z) \quad \hbox{for all}~(x,z) \in S
\]
and
\[
\tilde{P}(x_0,-z_0) = P(x_0,z_0) = U(x_0,z_0) = U(x_0,-z_0),
\]
 we have that $\tilde{P}$ touches $U$ from below at $(x_0,-z_0)$ in $S$.
\end{proof}

\begin{notation}
Given $f:\Omega\to\R$, we define the functions $f^\pm$ by 
\[
f^-(x) = \min\{0,f(x)\} \leq 0\quad \hbox{and} \quad f^+(x) = \max\{f(x),0\} \geq 0.
\]
\end{notation}

\begin{lem}\label{lem:vertices on z=0}
Let $f \in L^{\infty}(\Omega)$ and let $S \subset \subset \Omega \times \R \subset \R^{n+1}$ be an open, convex set such that $S \cap \{z=0\} \not= \varnothing$.
Suppose that a continuous function $U: \Omega \times \R \to \R$ such that
$U_{z+} \in C([0,\infty);C(\Omega))$ is symmetric across $\{z=0\}$ and satisfies
\[
-\partial_{z+}U(x,0) \geq f(x) \quad\hbox{on}~S \cap \{z=0\}.
\]
If $f(x_0) > 0$, then $U$ cannot be touched from below at $(x_0,0)$ in $S$  by any paraboloid. 
If $f(x_0) \leq 0$ and $P$ is a paraboloid of opening $a>0$ with vertex $(x_v,z_v)$ that touches $U$ from below in $S$ at $(x_0,0)$, then $\abs{h'(z_v)} \leq \abs{f^-(x_0)}/a$. 
Consequently, if $f(x_0) = 0$, then $z_v = 0$. 
\end{lem}

\begin{proof}
Suppose that $P$ is a paraboloid of opening $a>0$ that touches $U$ from below at $(x_0,0)$ in $S$. Write
\begin{align*}
P(x,z) = -a\delta_{\Phi}((x_v,z_v),(x,z))+a\delta_{\Phi}((x_v,z_v),(x_0,0)) + U(x_0,0).
\end{align*}
Let $\varepsilon>0$ be small.
Since $U-P$ attains a local minimum of 0 at $(x_0,0)$, we know that
\[
\frac{\(U(x_0,\varepsilon) - P(x_0,\varepsilon)\)-\(U(x_0,0) - P(x_0,0)\)}{\varepsilon}
=\frac{U(x_0,\varepsilon) - P(x_0,\varepsilon)}{\varepsilon} \geq 0.
\]
Therefore, taking the limit as $\varepsilon \to 0^+$, we obtain
\begin{equation}\label{eq:neum1}
0 \leq \partial_{z+}U(x_0,0) - \partial_zP(x_0,0)
	\leq -f(x_0) - ah'(z_v). 
\end{equation}
We note that, by the symmetry of $U$ across $\{z=0\}$,
we have that
\begin{align*}
\partial_{z-}U(x_0,0) 
	&= \lim_{h \to0^-} \frac{U(x_0,h) - U(x_0,0)}{h}\\
	&= - \lim_{h \to 0^+} \frac{U(x_0,-h) - U(x_0,0)}{-h}\\
	&=- \lim_{h \to 0^+} \frac{U(x_0,h) - U(x_0,0)}{h}\\
	&=-\partial_{z+}U(x_0,0) \geq f(x_0).
\end{align*}
For $\varepsilon >0$ small, we have that
\[
\frac{\(U(x_0,-\varepsilon) - P(x_0,-\varepsilon)\)-\(U(x_0,0) - P(x_0,0)\)}{-\varepsilon} 
	=\frac{-U(x_0,-\varepsilon) + P(x_0,-\varepsilon)}{\varepsilon}  \leq 0.
\]
Taking the limit as $\varepsilon \to 0^+$, we obtain
\begin{equation}\label{eq:neum2}
0 \geq \partial_{z-}U(x_0,0) - \partial_zP(x_0,0)
	\geq f(x_0) - ah'(z_v).
\end{equation}
By combining \eqref{eq:neum1} and \eqref{eq:neum2},
\[
f(x_0) \leq ah'(z_v) \leq -f(x_0).
\]
If $f(x_0) >0$, then the previous set of inequalities provides a contradiction, so $P$ cannot touch $U$ from below in $S$ at $(x_0,0)$. 
If $f(x_0) \leq 0$, then
\[
-\abs{f^-(x_0)} \leq ah'(z_v) \leq  \abs{f^-(x_0)}
\]
as desired. If $f(x_0) = 0$, then $h'(z_v) = 0$ which implies that $z_v=0$.
\end{proof}

%%%%%%%%%%%%%%%%%%%%%%%%%%%%%%%%
\section{Estimate on the Monge--Amp\`ere measure of the set of contact points}\label{sec:lem1}
%%%%%%%%%%%%%%%%%%%%%%%%%%%%%%%%

Our first key result is a measure estimate similar to the  Alexandroff--Bakelman--Pucci estimate for fully nonlinear equations.
We prove that if we lift paraboloids of fixed opening $a>0$ with vertices in a closed, bounded set from below until they touch the graph of $U$
for the first time, 
then, by using the equation and the Neumann boundary condition,
the Monge--Amp\`ere measure of the contact points is a universal proportion of the Monge--Amp\`ere measure of the
set of vertices. 

\begin{thm}\label{lem:ABP-super}
Assume that $\Omega$ is a bounded domain and that
$a^{ij}(x):\Omega \to \R$ are bounded, measurable functions that satisfy \eqref{eq:ellipticity}.
Let $Q_R = Q_R(\tilde{x},\tilde{z}) \subset \subset \Omega \times \R$ 
and $f \in L^{\infty}(Q_R \cap \{z=0\})$.
Suppose $U \in C^2(Q_R \setminus \{z=0\}) \cap C(Q_R)$ such that $U$ is symmetric across $\{z=0\}$ and  $U_{z+} \in C(Q_R \cap \{z\geq0\})$ 
is a supersolution to 
\[
\begin{cases}
a^{ij}(x) \partial_{ij}U + \abs{z}^{2-\frac{1}{s}}\partial_{zz}U \leq 0 & \hbox{in}~Q_R \cap \{z\not=0\}\\
-\partial_{z+}U \geq f & \hbox{on}~Q_R \cap \{z=0\}.
\end{cases}
\]
Let $B \subset \overline{Q}_R$ be a closed set and fix $a>0$. 
For each $(x_v,z_v) \in B$, we slide paraboloids of opening $a>0$ and vertex $(x_v,z_v)$ from below until they touch the graph of $U$
for the first time. 
Let $A$ denote the set of contact points and assume that $A \subset Q_R$. Then $A$ is compact and 
if
\[
\mu_{\Phi}\(B \cap \left\{(x,z): \abs{h'(z)} \leq \frac{\norm{f^-}_{ L^{\infty}(Q_R \cap \{z=0\})}}{a} \right\}\) \leq (1-\varepsilon_0) \mu_{\Phi}(B),
\]
for some $\varepsilon_0>0$, 
then there is a positive constant $c=c(n, \lambda,\Lambda)<1$ such that
\[
\mu_{\Phi}(A) \geq \varepsilon_0 c \mu_{\Phi}(B).
\]
\end{thm}

\begin{proof}
We first show that $A$ is closed. Let $(x_k,z_k)\in A$ be such that $(x_k,z_k) \to (x_0,z_0)$. 
There exist corresponding polynomials $P_k$ with vertices $(x_v^k, z_v^k) \in B$ 
such that $P_k$ touches $U$ from below at $(x_k,z_k)$ in $Q_R$.
Since $B \subset \overline{Q}_R$ is closed, $B$ is compact. 
Thus, up to a subsequence, $(x_v^{{k}}, z_v^{{k}}) \to (x_v^0, z_v^0) \in B$.
By the continuity of $\delta_{\Phi}$ and $U$, as $k\to\infty$,
\begin{align*}
P_k(x,z)
	&= -a\delta_{\Phi}((x_v^{k},z_v^{k}),(x,z))
		+ a\delta_{\Phi}((x_v^{k},z_v^{k}),(x_k,z_k)) + U(x_k,z_k)\\
	&\to -a\delta_{\Phi}((x_v^{0},z_v^{0}),(x,z))
		+a\delta_{\Phi}((x_v^{0},z_v^{0}),(x_0,z_0)) + U(x_0,z_0)=:P(x,z).
\end{align*}
Since $P_k(x,z) \leq U(x,z)$, it must be that $P \leq U$ in $Q_R$. Moreover, $P(x_0,z_0) = U(x_0,z_0)$.
Therefore, $P$ is a paraboloid of opening $a>0$ with vertex $(x_v^0,z_v^0) \in B$ that touches
 $U$ from below at $(x_0,z_0)$. 
 This shows that $(x_0,z_0) \in A$, so that $A$ is closed and, moreover, compact. 

Define the sets
\begin{align*}
B_0 &= B \cap \left\{(x,z): \abs{h'(z)} \leq \frac{\norm{f^-}_{ L^{\infty}(Q_R \cap \{z=0\})}}{a} \right\}\\
B_1 &=  B \setminus \left\{(x,z): \abs{h'(z)} \leq \frac{\norm{f^-}_{ L^{\infty}(Q_R \cap \{z=0\})}}{a} \right\},
\end{align*}
so that $B = B_0 \cup B_1$ and $B_0 \cap B_1 = \varnothing$.
We lift paraboloids of opening $a>0$ from below with vertices in $B_0$ and $B_1$ to form the contact sets
$A_0$ and $A_1$, respectively. Note that $A = A_0 \cup A_1$.

We will first show that $\mu_{\Phi}(B_1) \leq C \mu_{\Phi}(A_1)$ for some positive constant $C = C(n,\lambda,\Lambda)$.

Let $(x_0,z_0) \in A_1$. There exists a paraboloid $P$ of opening $a>0$ and vertex $(x_v,z_v) \in B_1$ that touches $U$ from below at $(x_0,z_0)$. 
If $z_0 = 0$, then, by Lemma \ref{lem:vertices on z=0}, it must be that $f(x_0) \leq 0$ and that
\[
\abs{h'(z_v)} \leq  \frac{\abs{f^-(x_0)}}{a} \leq \frac{\norm{f^-}_{ L^{\infty}(Q_R \cap \{z=0\})}}{a}
\]
which contradicts that $(x_v,z_v) \in B_1$. Hence, $z_0 \not= 0$. 

Since $U-P$ attains a local minimum at $(x_0,z_0)$,
\[
\nabla U(x_0,z_0) =  \nabla P(x_0,z_0) = -a\(x_0-x_v, h'(z_0)-h'(z_v) \)
\]
which implies
\[
(x_v, h'(z_v) ) = (x_0,h'(z_0)) + \frac{1}{a}\nabla U(x_0,z_0).
\]
This is how the vertices $(x_v,z_v) \in B_1$ are uniquely determined by $(x_0,z_0) \in A_1$. Notice that this identity is equivalent to
\[
\nabla\Phi(x_v,z_v) = \nabla\(\Phi +\frac{1}{a}U\)(x_0,z_0) \quad \hbox{for all}~(x_v,z_v) \in B_1.
\]

Consider the map $T:A_1 \to T(A_1) = \nabla\Phi(B_1)$ given by 
\[
T(x_0,z_0) = \nabla\( \Phi + \frac{1}{a}U\)(x_0,z_0). 
\]
For $\varepsilon>0$, let $A_{\varepsilon}\subset A_1$ be given by
\[
A_{\varepsilon} = A_1 \setminus \{(x,z) : \abs{z} <\varepsilon\}.
\]
Then, $T$ is Lipschitz and injective on $A_{\varepsilon}$, so that, by the area formula for Lipschitz maps,
\begin{align*}
\abs{T(A_{\varepsilon})}
	= \int_{T(A_{\varepsilon})} \, dy \, dw
	&=  \int_{A_{\varepsilon}} \abs{\det \(\nabla T(x,z)\)} dz \, dx\\
	&= \int_{A_{\varepsilon}} \abs{\det \(D^2\( \Phi + \frac{1}{a}U\)(x,z)\)} dz \, dx.
\end{align*}

We claim that there is a constant $C = C(n,\lambda,\Lambda)>0$ such that for all $(x_0,z_0) \in A_{1}$
\begin{equation}\label{eq:claimD2U}
-aD^2\Phi(x_0,z_0) \leq D^2U(x_0,z_0) \leq CaD^2\Phi(x_0,z_0).
\end{equation}
The first inequality is straightforward because $P$ touches $U$ from below at $(x_0,z_0)$.
To prove the second inequality in \eqref{eq:claimD2U}, 
suppose, by way of contradiction, that 
\begin{equation}\label{eq:second inequality}
D^2U(x_0,z_0) > CaD^2\Phi(x_0,z_0) \quad \hbox{for all}~C>0. 
\end{equation}
Then
\[
D^2U(x_0,z_0) > Ca\begin{pmatrix} e_k\otimes e_k & 0 \\ 0 & 0\end{pmatrix} > Ca\begin{pmatrix} e_k\otimes e_k & 0 \\ 0 & 0\end{pmatrix} - a\begin{pmatrix} I & 0 \\ 0 & \abs{z_0}^{\frac{1}{s}-2}\end{pmatrix} 
\]
where $e_k$, $k=1,\dots, n$ are the standard basis vectors in $\R^n$.
Since $\tilde{A} = \begin{pmatrix} A(x_0) & 0 \\ 0 & 0 \end{pmatrix} \geq 0$ and 
\[
D^2U(x_0,z_0) - Ca\begin{pmatrix} e_k\otimes e_k & 0 \\ 0 & 0\end{pmatrix} +a\begin{pmatrix} I & 0 \\ 0 & \abs{z_0}^{\frac{1}{s}-2}\end{pmatrix} \geq 0,
\]
we have that
\[
\trace\(\tilde{A}D^2U(x_0,z_0) - Ca\tilde{A}\begin{pmatrix} e_k\otimes e_k & 0 \\ 0 & 0\end{pmatrix} 
+a\tilde{A}\begin{pmatrix} I & 0 \\ 0 & \abs{z_0}^{\frac{1}{s}-2}\end{pmatrix} \)\geq 0.
\]
By ellipticity (see \eqref{eq:ellipticity}),
\begin{equation}\label{lem:x-estimate}
\begin{aligned}
a^{ij}(x_0)\partial_{ij}U(x_0,z_0) 
&\geq (Ca)a^{kk}(x_0) - a\trace(A(x_0)) 
\geq Ca\lambda - an\Lambda.
\end{aligned}
\end{equation}
Similarly, from \eqref{eq:second inequality}, 
\[
D^2U(x_0,z_0) 
>Ca\begin{pmatrix} 0 & 0 \\ 0 & \abs{z_0}^{\frac{1}{s}-2}\end{pmatrix}
> Ca\begin{pmatrix} 0 & 0 \\ 0 & \abs{z_0}^{\frac{1}{s}-2}\end{pmatrix} - a\begin{pmatrix} I & 0 \\ 0 & \abs{z_0}^{\frac{1}{s}-2}\end{pmatrix}.
\]
From the definition of positive definite matrices, 
\[
\partial_{zz}U(x_0,z_0) - Ca \abs{z_0}^{\frac{1}{s}-2} +a\abs{z_0}^{\frac{1}{s}-2} > 0.
\]
Therefore,
\begin{equation}\label{lem:z-estimate}
\abs{z_0}^{2-\frac{1}{s}}\partial_{zz}U(x_0,z_0) >  Ca  -a.
\end{equation}
By \eqref{lem:x-estimate} and \eqref{lem:z-estimate}, it follows that
\begin{align*}
0 &\geq a^{ij}(x_0)\partial_{ij}U(x_0,z_0) +\abs{z_0}^{2-\frac{1}{s}}\partial_{zz}U(x_0,z_0)\\
	&> Ca\lambda - an\Lambda+Ca -a\\
	&= [C(\lambda +1) - (n\Lambda+1)]a,
\end{align*}
which is a contradiction when 
 $C = C(n,\lambda,\Lambda)>0$ is sufficiently large. 
Thus, \eqref{eq:claimD2U} holds. 

From \eqref{eq:claimD2U}, we get
\begin{align*}
0 \leq D^2\(\Phi + \frac{1}{a}U\)(x_0,z_0) 
&\leq D^2\Phi(x_0,z_0) + CD^2\Phi(x_0,z_0)
= (C+1)D^2\Phi(x_0,z_0)
\end{align*}
for all $(x_0,z_0) \in A_1$.
Hence,
\begin{align*}
\abs{T(A_{\varepsilon})}
	&= \int_{A_{\varepsilon}} \det\(D^2\(\Phi + \frac{1}{a}U\)(x,z) \) \, dx \, dz\\
	&\leq  \int_{A_{\varepsilon}} \det\( (C+1)D^2\Phi(x,z) \) \, dx \, dz\\
	&=  (C+1)^{n+1}\mu_{\Phi}(A_{\varepsilon})\\
	&\leq (C+1)^{n+1} \mu_{\Phi}(A_1).
\end{align*}
As this holds for all $\varepsilon>0$, 
\[
\mu_{\Phi}(B_1) 
	= \abs{\nabla\Phi(B_1)}
	=  \abs{T(A_1)} \leq  (C+1)^{n+1} \mu_{\Phi}(A_1).
\]

Thus, 
\begin{align*}
\mu_{\Phi}(B)
	&= \mu_{\Phi}(B_0) + \mu_{\Phi}(B_1) 
	\leq (1-\varepsilon_0)\mu_{\Phi}(B) +(C+1)^{n+1} \mu_{\Phi}(A_1)
\end{align*}
from which it follows that
\[
\mu_{\Phi}(A) \geq \mu_{\Phi}(A_1) \geq \frac{\varepsilon_0}{2(C+1)^{n+1}}\mu_{\Phi}(B) = c\varepsilon_0\mu_{\Phi}(B).
\]
\end{proof}

The following is a parallel result to that of Theorem \ref{lem:ABP-super}
for subsolutions when paraboloids of opening $a<0$ are lowered from above until they touch the graph of $U$ for the first time.
The proof is straightforward. We will apply this lemma in the proof of Theorem \ref{thm:reduction3}.

\begin{thm}\label{lem:ABP-sub}
Assume that $\Omega$ is a bounded domain and that
$a^{ij}(x):\Omega \to \R$ are bounded, measurable functions that satisfy \eqref{eq:ellipticity}.
Let $Q_R = Q_R(\tilde{x},\tilde{z}) \subset \subset \Omega \times \R$ 
and $f \in L^{\infty}(Q_R \cap \{z=0\})$.
Suppose $U \in C^2(Q_R \setminus \{z=0\}) \cap C(Q_R)$ such that $U$ is symmetric across $\{z=0\}$ and  $U_{z+} \in C(Q_R \cap \{z\geq0\})$ 
is a subsolution to 
\[
\begin{cases}
a^{ij}(x) \partial_{ij}U + \abs{z}^{2-\frac{1}{s}}\partial_{zz}U \geq 0 & \hbox{in}~Q_R \cap \{z\not=0\}\\
-\partial_{z+}U \leq f & \hbox{on}~Q_R \cap \{z=0\}.
\end{cases}
\]
Let $B \subset \overline{Q}_R$ be a closed set and fix $a<0$. 
For each $(x_v,z_v) \in B$, we slide paraboloids of opening $a$ and vertex $(x_v,z_v)$ from above until they touch the graph of $U$
for the first time. 
Let $A$ denote the set of contact points and assume that $A \subset Q_R$. Then $A$ is compact and 
if
\[
\mu_{\Phi}\(B \cap \left\{(x,z): \abs{h'(z)} \leq \frac{\norm{f^+}_{ L^{\infty}(Q_R \cap \{z=0\})}}{|a|} \right\}\) \leq (1-\varepsilon_0) \mu_{\Phi}(B),
\]
for some $\varepsilon_0>0$, 
then there is a positive constant $c=c(n, \lambda,\Lambda)<1$ such that
\[
\mu_{\Phi}(A) \geq \varepsilon_0 c \mu_{\Phi}(B).
\]
\end{thm}

\begin{rem}\label{rem:lem1-sections}
By checking the proofs, it is easy to see that Theorems \ref{lem:ABP-super} and \ref{lem:ABP-sub} are still valid when the cube $Q_R$ is replaced by a section $S_R$.
\end{rem}

%%%%%%%%%%%%%%%%%%%%%%%%
\section{Explicit barriers}\label{sec:barrier}
%%%%%%%%%%%%%%%%%%%%%%%%

This section contains the construction of the barriers that will be used in Section \ref{sec:lem2} 
to prove a localization estimate. This is a quite delicate task due to the degeneracy/singularity of the extension equation
and the presence of the Neumann boundary condition. 

The idea for the barrier is to use $\delta_{\Phi}((x_0,z_0),(x,z))^{-\alpha}$, for $\alpha>0$ large, to construct subsolutions in 
a ring $S_{2r}(x_0,z_0) \setminus S_{\gamma r}(x_0,z_0)$.
This depends heavily on whether $s$ is smaller or larger than $1/2$. When $0< s \leq 1/2$, the coefficient $\abs{z}^{2 - 1/s}$ blows up at the origin.
When $1/2 < s <1$, the coefficient $\abs{z}^{2- 1/s}$ degenerates near $z=0$. 
In the latter case, we need to use an auxiliary function that bypasses the points where $\abs{z}^{2- 1/s}$ is small. 
A similar auxiliary function will be used when $z_0 = 0$ to force the Neumann condition to be strictly positive. 
By the symmetry of the equation, it will be enough to consider the nonnegative side of the ring if $z_0 \geq 0$ and the nonpositive side if $z_0 \leq 0$. 

The following is a preliminary result that
will be used in the case when $0 < s \leq 1/2$.

\begin{lem}\label{lem:quotient}
Let $0 < s \leq 1/2$ and $z_0>0$ be fixed. 
Define the function $\mQ:\R\to \R$ by
\[
\mQ(z) = \frac{(h'(z) - h'(z_0))^2}{\delta_h(z_0,z) h''(z)}.
\]
Then $\mQ$ is a continuous function of $z>0$, and $\mQ(z) \geq 1$ for all $z>0$.
\end{lem}

\begin{proof}
By L'H\"opital's rule, $\lim_{z \to z_0} \mQ(z)=2$,
so that $\mQ(z)$ is continuous for $z>0$. 
Also, for $s = 1/2$ and all $z \not= z_0$, we have $\mQ(z)= 2$. Hence, let us assume for the remainder of the proof that $0 < s < 1/2$.

It is easy to see that $\lim_{z \to 0^+} \mQ(z)=\infty$ and that 
$\lim_{z \to \infty} \mQ(z)=\frac{1}{1-s}\geq1$.
Therefore, it is enough to prove that $\mQ(z)$ is decreasing for $z>0$, $z\not= z_0$.
To this end, we will show that $\mQ'(z)<0$ for $z \not= z_0$. First, observe that
\begin{equation}\label{eq:Q}
\mQ'(z) 	= \frac{(h'(z) - h'(z_0))h''(z)}{(\delta_h(z_0,z) h''(z))^2} \cdot I(z)
\end{equation}
where
$$I(z)=2\delta_h(z_0,z) h''(z) - (h'(z) - h'(z_0))^2 - \delta_h(z_0,z) (h'(z) - h'(z_0)) \frac{h'''(z)}{h''(z)}.$$
We can write
$$I(z)=-\frac{s}{1-s}z_0^{\frac{2}{s}-2} 
	 +\frac{s}{1-s} z_0^{\frac{1}{s}} z^{\frac{1}{s}-2} 
	 -\frac{s(1-2s)}{1-s} z_0^{\frac{1}{s}-1}z^{\frac{1}{s}-1} 
	 + \frac{s(1-2s)}{1-s}z_0^{\frac{2}{s}-1}z^{-1}.$$
It follows that $I(z)>0$ for all $z>0$ if and only if
\[
\psi(z) := - z_0^{\frac{2}{s}-1} z
	 + z_0^{\frac{1}{s}+1} z^{\frac{1}{s}-1} 
	 -(1-2s) z_0^{\frac{1}{s}}z^{\frac{1}{s}} 
	 +(1-2s)z_0^{\frac{2}{s}} >0,
\]
for all $z>0$.
Note that $\psi(z_0) = 0$ and  $\psi(0) = (1-2s)z_0^{\frac{2}{s}}>0$.
We claim that $\psi$ is decreasing as function of $z>0$. Indeed, $\psi'(z) <0$ if and only if 
\[
- z_0^{\frac{2}{s}-1} 
	 + \(\frac{1}{s}-1\)z_0^{\frac{1}{s}+1} z^{\frac{1}{s}-2} 
	 -\frac{(1-2s)}{s} z_0^{\frac{1}{s}}z^{\frac{1}{s}-1} <0.
\]
Multiplying both sides by $z_0^{-1/s}s/(1-s)>0$ and rearranging, this is equivalent to 
\[
 z_0 z^{\frac{1}{s}-2} < \(\frac{s}{1-s}\) z_0^{\frac{1}{s}-1}+ \(\frac{1-2s}{1-s}\)z^{\frac{1}{s}-1},
\]
which is true by Young's inequality, and the claim follows. Thus, we conclude that
\[
\psi (z)>0 \quad \hbox{for}~0 < z < z_0 \quad \hbox{and} \quad \psi (z)< 0 \quad \hbox{for}~z>z_0.
\]
This gives that
\[
I(z)>0 \quad \hbox{for}~0 < z < z_0 \quad \hbox{and} \quad I(z)< 0 \quad \hbox{for}~z>z_0.
\]
Since, in addition,
\[
h'(z) - h'(z_0) <0 \quad \hbox{for}~0 < z < z_0 \quad \hbox{and} \quad h'(z) - h'(z_0) > 0 \quad \hbox{for}~z>z_0,
\]
we deduce from \eqref{eq:Q} that $\mQ'(z) <0$ for all $z\not=z_0$. This completes the proof.
\end{proof}

We now construct the barriers $\phi$.
For a set $S\subset \R^{n+1}$, we introduce the notation 
\[
S^+ = S \cap \{z \geq 0\} \quad \hbox{and} \quad 
S^- = S \cap \{z \leq 0\}.
\]
To deal with the singularity at $z=0$, we define $\phi$ in either the positive or negative half spaces. 
In particular, if $z_0 \geq 0$, then we consider the partial ring $[S_r(x_0,z_0) \setminus S_{\gamma r}(x_0,z_0)]^{+}$.
If $z_0 < 0$, then we consider the partial ring $[S_r(x_0,z_0) \setminus S_{\gamma r}(x_0,z_0)]^{-}$.
We will use the condensed notation
\[
[S_r(x_0,z_0) \setminus S_{\gamma r}(x_0,z_0)]^{\pm} 
= \begin{cases}
[S_r(x_0,z_0) \setminus S_{\gamma r}(x_0,z_0)]^{+} & \hbox{if}~z_0 \geq 0 \\
\left[S_r(x_0,z_0) \setminus S_{\gamma r}(x_0,z_0)\right]^{-} & \hbox{if}~z_0 < 0.
\end{cases}
\]

\begin{lem}\label{lem:subsoln}
Fix $0 < \gamma <1$ and consider a section $S_r(x_0,z_0) \subset \R^{n+1}$. 

If $z_0 \geq 0$, then there
exists a classical subsolution $\phi = \phi(x,z)$ to
\begin{equation} \label{eq:subsoln1}
\begin{cases}
a^{ij}(x) \partial_{ij}\phi + \abs{z}^{2-\frac{1}{s}}\partial_{zz} \phi > a(n\Lambda+1) & \hbox{in}~[S_{2r}(x_0,z_0) \setminus {S}_{\gamma r}(x_0,z_0)]^{+}\cap \{z\not=0\}\\
-\partial_{z+}\phi(x,0) < 0 & \hbox{on}~[S_{2r}(x_0,z_0) \setminus {S}_{\gamma r}(x_0,z_0)]^{+}\cap \{z=0\}.
\end{cases}
\end{equation}
If $z_0 \leq 0$, then there exist a classical subsolution $\phi = \phi(x,z)$ to
\begin{equation} \label{eq:subsoln1-neg}
\begin{cases}
a^{ij}(x) \partial_{ij}\phi + \abs{z}^{2-\frac{1}{s}}\partial_{zz} \phi > a(n\Lambda+1) & \hbox{in}~[S_{2r}(x_0,z_0) \setminus {S}_{\gamma r}(x_0,z_0)]^-\cap \{z\not=0\}\\
-\partial_{z-}\phi(x,0) > 0 & \hbox{on}~[S_{2r}(x_0,z_0) \setminus {S}_{\gamma r}(x_0,z_0)]^-\cap \{z=0\}.
\end{cases}
\end{equation}
In each case, 
$\phi >0$ in $[S_{r}(x_0,z_0) \setminus S_{\gamma r}(x_0,z_0)]^{\pm}$,  $\phi \leq 0$ on $[\partial S_{2r}(x_0,z_0)]^{\pm}$, and there is a constant $C = C(n,\lambda,\Lambda, \gamma) >0$ such that $\phi \leq Car$ on $[\partial S_{\gamma r}(x_0,z_0)]^{\pm}$.
\end{lem}

\begin{proof}
The proof of \eqref{eq:subsoln1-neg} will follow from \eqref{eq:subsoln1} at the end by symmetry. 
The construction of the subsolution in \eqref{eq:subsoln1} will depend on whether $z_0>0$ or $z_0=0$ and on whether $0 < s \leq 1/2$ or $1/2 < s < 1$. 

\medskip
\noindent
\underline{\bf Case 1}: $z_0> 0$ and  $0 < s \leq 1/2$.

\smallskip

We begin by considering the function $(\delta_\Phi((x_0,z_0),(x,z)))^{-\alpha}$ for a large constant $\alpha = \alpha(\gamma,n,\lambda,\Lambda,s) >0$ which will be fixed later on.
Let $\mQ(z)$ be the function defined in Lemma \ref{lem:quotient}. 
For a point $(x,z) \in [S_{2r}(x_0,z_0) \setminus S_{\gamma r}(x_0,z_0)]^{+} \setminus \{z=0\}$, we use ellipticity and Lemma \ref{lem:quotient} to estimate
\begin{align*}
a^{ij}&(x) \partial_{ij} (\delta_\Phi((x_0,z_0),(x,z)))^{-\alpha} + \abs{z}^{2-\frac{1}{s}} \partial_{zz} (\delta_\Phi((x_0,z_0),(x,z)))^{-\alpha}\\
&= \alpha (\delta_\Phi((x_0,z_0),(x,z)))^{-\alpha-2} \\
&\quad	\bigg[ (\alpha+1) \big( a^{ij}(x) (x-x_0)_i(x-x_0)_j+ \abs{z}^{2-\frac{1}{s}}(h'(z) - h'(z_0))^2\big) \\
&\qquad	-  \(\trace(A(x))+1\) \delta_\Phi((x_0,z_0),(x,z))\bigg]\\
&\geq \alpha (\delta_\Phi((x_0,z_0),(x,z)))^{-\alpha-2} \\
&\quad	\bigg[ (\alpha+1) \big( \lambda \abs{x-x_0}^2+ \abs{z}^{2-\frac{1}{s}}(h'(z) - h'(z_0))^2\big) 	-  \(n\Lambda+1\) \delta_\Phi((x_0,z_0),(x,z))\bigg]\\
&= \alpha (\delta_\Phi((x_0,z_0),(x,z)))^{-\alpha-2} \\
&\quad	\bigg[ (\alpha+1) \bigg( 2\lambda \delta_{\varphi}(x_0,x)+\frac{(h'(z) - h'(z_0))^2}{h''(z) \delta_h(z_0,z)} \delta_h(z_0,z)\bigg) 	-  \(n\Lambda+1\)\(\delta_\varphi(x_0,x) + \delta_h(z_0,z)\)\bigg]\\
&= \alpha (\delta_\Phi((x_0,z_0),(x,z)))^{-\alpha-2} \\
&\quad	\bigg[ \(2\lambda (\alpha+1) - (n\Lambda+1)\) \delta_{\varphi}(x_0,x)
	+\(\mQ(z)(\alpha+1) -  \(n\Lambda+1\)\) \delta_h(z_0,z)\bigg]\\
&\geq \alpha (\delta_\Phi((x_0,z_0),(x,z)))^{-\alpha-2} \\
&\quad	\bigg[ \(2\lambda (\alpha+1) - (n\Lambda+1)\) \delta_{\varphi}(x_0,x)
	+\((\alpha+1) -  \(n\Lambda+1\)\) \delta_h(z_0,z)\bigg].	
\end{align*}
Choose $\alpha = \alpha(\gamma, n, \lambda, \Lambda)$ large so that
\[
 2\lambda (\alpha+1) - (n\Lambda+1) > 4\gamma^{-1}(n\Lambda+1) 
 \quad \hbox{and} \quad
 (\alpha+1) -  \(n\Lambda+1\) > 4\gamma^{-1}(n\Lambda+1).
\]
Since $\gamma r\leq \delta_{\Phi}((x_0,z_0),(x,z)) = \delta_{\varphi}(x_0,x) + \delta_h(z_0,z)$,
it must be that $\delta_{\varphi} (x_0,x)\geq \gamma r/2$ or $\delta_{h}(z_0,z) \geq \gamma r/2$.
If $\delta_{\varphi}(x_0,x) \geq \gamma r/2$, then 
\begin{align*}
a^{ij}&(x) \partial_{ij} (\delta_\Phi((x_0,z_0),(x,z)))^{-\alpha} + \abs{z}^{2-\frac{1}{s}} \partial_{zz} (\delta_\Phi((x_0,z_0),(x,z)))^{-\alpha}\\
&> \alpha(\delta_\Phi((x_0,z_0),(x,z)))^{-\alpha-2} \bigg[ 4\gamma^{-1}(n\Lambda+1) \delta_{\varphi}(x_0,x)
	+0\bigg]\\
&\geq \alpha (n\Lambda+1)(2r)^{-\alpha-1}.
\end{align*}
If $\delta_{h}(z_0,z)\geq \gamma r/2$, then
\begin{align*}
a^{ij}&(x) \partial_{ij} (\delta_\Phi((x_0,z_0),(x,z)))^{-\alpha} + \abs{z}^{2-\frac{1}{s}} \partial_{zz} (\delta_\Phi((x_0,z_0),(x,z)))^{-\alpha}\\
&> \alpha(\delta_\Phi((x_0,z_0),(x,z)))^{-\alpha-2}\bigg[0	+4\gamma^{-1}\(n\Lambda+1\) \delta_h(z_0,z)\bigg]\\
&\geq \alpha (n\Lambda+1)(2r)^{-\alpha-1}.
\end{align*}
Combing the previous two estimates, we have that, for all $(x,z)\in [S_{2r}(x_0,z_0) \setminus S_{\gamma r}(x_0,z_0)]^{+} \setminus \{z=0\}$,
\[
a^{ij}(x) \partial_{ij} (\delta_\Phi((x_0,z_0),(x,z)))^{-\alpha} + \abs{z}^{2-\frac{1}{s}} \partial_{zz}(\delta_\Phi((x_0,z_0),(x,z)))^{-\alpha}\\ > \alpha (n\Lambda+1)(2r)^{-\alpha-1}.
\]

Define $\phi$ in $[S_{2r}(x_0,z_0) \setminus S_{\gamma r}(x_0,z_0)]^{+} $ by
\[
\phi(x,z) 
=\alpha^{-1}a(2r)^{\alpha+1}[(\delta_\Phi((x_0,z_0),(x,z)))^{-\alpha} - r^{-\alpha}].
\]
Then $a^{ij}(x) \partial_{ij} \phi(x,z)+ \abs{z}^{2-\frac{1}{s}} \phi(x,z)> a(n\Lambda+1)$.
If $ [S_{2r}(x_0,z_0) \setminus S_{\gamma r}(x_0,z_0)]^+ \cap \{z=0\} \not= \varnothing$, we need to check the Neumann condition. 
In this case, let $(x,0) \in [S_{2r}(x_0,z_0) \setminus S_{\gamma r}(x_0,z_0)]^+ \cap \{z=0\}$ and observe that
\begin{align*}
\partial_{z+}\phi(x,0)
	&= -a(2r)^{\alpha+1}(\delta_{\Phi}((x_0,z_0),(x,z))^{-\alpha-1}(h'(z) -h'(z_0)) \big|_{z=0}\\
	&= a(2r)^{\alpha+1}(\delta_{\Phi}((x_0,z_0),(x,0))^{-\alpha-1} h'(z_0)\\
	&\geq ah'(z_0) >0
\end{align*}
since $z_0>0$. 
Therefore, $\phi$ defined in $[S_{2r}(x_0,z_0) \setminus S_{\gamma r}(x_0,z_0)]^+$ is a subsolution to \eqref{eq:subsoln1}. 
It is easy to check that $\phi \leq 0$ in $[S_{2r}(x_0,z_0) \setminus S_r(x_0,z_0)]^+$ and $\phi>0$ in $[S_r(x_0,z_0)]^+$.
Lastly, for $(x,z) \in [\partial S_{\gamma r}(x_0,z_0)]^{\pm}$, we have that
$\phi(x,z)
	= \alpha^{-1} a2^{\alpha+1}(\gamma^{-1}-1) r  = Car$, where $C = C(\gamma, n, \lambda, \Lambda)>0$.

\medskip
\noindent
\underline{\bf Case 2}: $z_0\geq 0$ and  $1/2 < s <1$.

\smallskip

Here we need to bypass the points where $\abs{z}^{2-\frac{1}{s}}$ is small with respect to the size of the section $S_{2r}(z_0) \subset \R$.
Let $0<\varepsilon<1$ be a small constant, to be chosen.  
Let $0<\varepsilon_0<1$ be as in Lemma \ref{lem:A-infty}
and define the set $H_{\varepsilon}$ by
\begin{align*}
H_{\varepsilon}
&= \left\{z \in S_{2r}(z_0) : \abs{z}^{2-\frac{1}{s}} \leq \varepsilon_0 \frac{\abs{S_{2r}(z_0)}}{\mu_h(S_{2r}(z_0))}\right\}\\
&= \left\{z \in S_{2r}(z_0) : 1 \leq \varepsilon_0 \frac{\abs{S_{2r}(z_0)}}{\mu_h(S_{2r}(z_0))}h''(z)\right\}.
\end{align*}
We first show that the measure of  $H_{\varepsilon}$ is  small with respect to the measure of the section $S_{2r}(z_0)$.
Indeed, using Lemma \ref{lem:h-integral}, we estimate
\begin{align*}
\abs{H_\varepsilon} = \int_{H_{\varepsilon}} \, dz
	&\leq \int_{H_{\varepsilon}} \varepsilon_0 \frac{\abs{S_{2r}(z_0)}}{\mu_h(S_{2r}(z_0))} h''(z) \, dz\\
	&\leq \varepsilon_0 \frac{\abs{S_{2r}(z_0)}}{\mu_h(S_{2r}(z_0))} \int_{S_{2r}(z_0)} h''(z) \, dz
	= \varepsilon_0 \abs{S_{2r}(z_0)}.
\end{align*}
By Lemma \ref{lem:A-infty},  $\mu_h(H_{\varepsilon}) \leq \varepsilon \mu_h(S_{2r}(z_0))$. 

We will construct a function $h_{\varepsilon}$ in $[S_{2r}(z_0)]^+$ that bypasses the points in $H_{\varepsilon}$.
Let $\tilde{H}_{\varepsilon}$ be an open interval such that
\[
H_{\varepsilon} \subset \tilde{H}_{\varepsilon} \subset S_{2r}(z_0), \quad \mu_{h}(\tilde{H}_{\varepsilon} \setminus H_{\varepsilon}) \leq \varepsilon  \mu_{h}(S_{2r}(z_0)),
\] 
and let $\psi_{\varepsilon} = \psi_{\varepsilon}(z)$ be a smooth function satisfying
\[
\psi_\varepsilon = 1~\hbox{in}~H_{\varepsilon}, 
\quad \psi_\varepsilon = \varepsilon~\hbox{in}~S_{2r}(z_0) \setminus \tilde{H}_{\varepsilon},
\quad \varepsilon~\leq \psi_\varepsilon \leq 1~\hbox{in}~S_{2r}(z_0).
\]

We use the notation
\begin{align*}
[S_{2r}(z_0)]^+ = (z_L,z_R), \quad \hbox{where}~0 \leq z_L \leq z_0<z_R.
\end{align*}
Note that $z_L = 0$ if $0 \in S_{2r}(z_0)$.

In $[S_{2r}(z_0)]^+$, let $h_{\varepsilon} = h_{\varepsilon}(z)$ be the strictly convex solution to 
\[
\begin{cases}
h_{\varepsilon}'' = 2(n\Lambda+1) \psi_\varepsilon h''   & \hbox{in}~[S_{2r}(z_0)]^+ \\
h_{\varepsilon} (z_R) = 0 & \\
h'_{\varepsilon}(z_L) = \varepsilon \mu_h(S_{2r}(z_0)).
\end{cases}
\]
We remark that $h_{\varepsilon} \in C^{\infty}((z_L,z_R))$ and, since $h \in C^1(\R)$, we have $h_{\varepsilon} \in C^1([\overline{S}_{2r}(z_0)]^+)$.
Since $h_{\varepsilon}$ is strictly convex in $[S_{2r}(z_0)]^+$ and $h_{\varepsilon} \in C^1([\overline{S}_{2r}(z_0)]^+)$, it follows that $h_{\varepsilon}'>0$ in $[S_{2r}(z_0)]^+$. 
Moreover, since $h_{\varepsilon}$ is strictly increasing, $h_{\varepsilon}$ achieves its maximum at $z=z_R$, so that $h_{\varepsilon}\leq 0$ in $[\overline{S}_{2r}(z_0)]^+$.

To bound $h_{\varepsilon}$ and $h_{\varepsilon}'$, we first estimate
\begin{align*}
 \int_{S_{2r}(z_0)} \psi_\varepsilon \, d \mu_{h}
 	&=  \int_{H_{\varepsilon}} \psi_\varepsilon \, d \mu_{h}
		+ \int_{\tilde{H}_{\varepsilon} \setminus H_{\varepsilon}} \psi_\varepsilon  \, d \mu_{h}
		 +  \int_{S_{2r}(z_0)\setminus \tilde{H}_{\varepsilon}} \psi_\varepsilon \, d \mu_{h}\\
 	&\leq  \int_{H_{\varepsilon}}  d \mu_{h}
		+ \int_{\tilde{H}_{\varepsilon} \setminus H_{\varepsilon}}  \,d \mu_{h}
		 +  \int_{S_{2r}(z_0)\setminus \tilde{H}_{\varepsilon}} {\varepsilon} \,d \mu_{h}\\
	&=  \mu_{h}(H_{\varepsilon})
		+ \mu_{h}(\tilde{H}_{\varepsilon} \setminus H_{\varepsilon})
		 + \varepsilon \mu_{h}(S_{2r}(z_0)\setminus \tilde{H}_{\varepsilon})\\
	&\leq \varepsilon \mu_{h}(S_{2r}(z_0))
		+ \varepsilon \mu_{h}(S_{2r}(z_0))
		 + \varepsilon \mu_{h}(S_{2r}(z_0))= 3\varepsilon \mu_{h}(S_{2r}(z_0)).
\end{align*}
Let $\delta>0$. For $z \in [S_{2r}(z_0)]^+$, by the previous estimate,
\begin{align*}
\abs{h_{\varepsilon}'(z)} = h_{\varepsilon}'(z)
	&= \int_{z_L + \delta}^{z} h_{\varepsilon}''(w) \, dw + h_{\varepsilon}'(z_L + \delta)\\
	&= \int_{z_L + \delta}^{z} 2(n\Lambda+1) \psi_\varepsilon h''(w) \, dw + h_{\varepsilon}'(z_L + \delta)\\
	&\leq 2(n\Lambda+1)\int_{S_{2r}(z_0)} \psi_\varepsilon \, d \mu_h + h_{\varepsilon}'(z_L + \delta)\\
	&\leq 6(n\Lambda+1)\varepsilon \mu_{h}(S_{2r}(z_0)) + h_{\varepsilon}'(z_L + \delta)\\
	&= C\varepsilon \mu_{h}(S_{2r}(z_0))  + h_{\varepsilon}'(z_L + \delta)
\end{align*}
for a constant $C = C(n,\Lambda)>0$.
Taking the limit as $\delta \to 0$, we have
\begin{align*}
h_{\varepsilon}'(z)
	\leq C\varepsilon \mu_{h}(S_{2r}(z_0))  + h_{\varepsilon}'(z_L)
	= C\varepsilon \mu_{h}(S_{2r}(z_0))  +\varepsilon \mu_{h}(S_{2r}(z_0)) 
	= C_1 \varepsilon \mu_{h}(S_{2r}(z_0)) 
\end{align*}
for a constant $C_1 = C_1(n,\Lambda)$.

Again, let $\delta>0$.  For $z \in [S_{2r}(z_0)]^+$, 
by Corollary \ref{lem:doubling} part (3),
\begin{align*}
\abs{h_{\varepsilon}(z)}
	= - h_{\varepsilon}(z)
	&= \int_z^{z_R-\delta}h_\varepsilon'(w) \, dw - h_{\varepsilon}(z_R-\delta)\\
	&\leq  C_1 \varepsilon \mu_{h}(S_{2r}(z_0))  \int_z^{z_R-\delta} \, dw - h_{\varepsilon}(z_R-\delta)\\
	&\leq C_1 \varepsilon \mu_{h}(S_{2r}(z_0)) \abs{S_{2r}(z_0)} - h_{\varepsilon}(z_R-\delta)\\
	&\leq C_2 \varepsilon r - h_{\varepsilon}(z_R-\delta)
\end{align*}
for a constant $C_2 = C_2(n,\Lambda,s)>0$. 
Taking the limit as $\delta \to 0$, we have
\[
\abs{h_{\varepsilon}(z)}
	\leq C_2 \varepsilon r - h_{\varepsilon}(z_R)
	= C_2\varepsilon r.
\]

Suppose that $\gamma r/2\leq \delta_h(z_0,z) < 2r$. 
By the convexity of $\delta_h(z_0,z)$ in the variable $z$, we obtain
\[
0 = \delta_h(z_0,z_0) \geq \delta_h(z_0,z) + \partial_z\delta_h(z_0,z) \cdot (z_0-z).
\]
By Corollary \ref{lem:doubling} part (3), this implies
\[
\abs{\partial_z\delta_h(z_0,z) } 
	\geq \frac{\delta_h(z_0,z)}{\abs{z-z_0}}
	\geq \frac{\gamma r/2}{\abs{S_{2r}(z_0)}}\\
	\geq C_{3} \mu_h(S_{2r}(z_0))
\]
for a constant $C_3= C_3(\gamma,s)$.
Choose $\varepsilon = \varepsilon(\gamma,n,\Lambda,s)>0$ small so that $C_1\varepsilon < C_3$. Then, 
\[
\abs{\partial_z \delta_h(z_0,z) - h'_{\varepsilon}(z)}
	\geq \abs{\partial_z \delta_h(z_0,z)} - \abs{h'_{\varepsilon}(z)}
	\geq (C_3 - C_1\varepsilon)\mu_h(S_{2r}(z_0))>0
\]
and 
\begin{equation}\label{eq:deriv squared-new}
(\partial_z \delta_h(z_0,z) - h'_{\varepsilon}(z))^2
	\geq (C_3 - C_1\varepsilon)^2[\mu_h(S_{2r}(z_0))]^2
	= C_4 [\mu_h(S_{2r}(z_0))]^2
\end{equation}
for a constant $C_4 = C_4(\gamma,n,\Lambda,s)>0$.

For a large constant $\alpha = \alpha(\gamma,n,\lambda,\Lambda,s)>0$, we define the function $\tilde{\phi}$ on $[S_{2r}(x_0,z_0) \setminus S_{\gamma r}(x_0,z_0)]^+$ by 
\[
\tilde{\phi}(x,z) = (\delta_{\Phi}((x_0,z_0),(x,z))- h_{\varepsilon}(z))^{-\alpha}.
\]
Let $(x,z) \in [S_{2r}(x_0,z_0) \setminus S_{\gamma r}(x_0,z_0)]^+ \setminus \{z=0\}$. Since $h_{\varepsilon} \leq 0$, we first note that
\begin{equation}\label{eq:barrier bound}
\gamma r 
	\leq \delta_{\Phi}((x_0,z_0),(x,z))
	\leq \delta_{\Phi}((x_0,z_0),(x,z)) -h_{\varepsilon}(z) 
	< 2r + C_{2}\varepsilon r
	= (2 + C_{2} \varepsilon)r. 
\end{equation}
The equation for $\tilde{\phi}$ in $[S_{2r}(x_0,z_0) \setminus S_{\gamma r}(x_0,z_0)]^+ \setminus \{z=0\}$ is
\begin{align*}
a^{ij}(x)& \partial_{ij}\tilde{\phi} + \abs{z}^{2-\frac{1}{s}} \partial_{zz} \tilde{\phi}\\
&= \alpha(\delta_{\Phi}((x_0,z_0),(x,z))- h_{\varepsilon}(z))^{-\alpha-2}\\
&\quad\bigg((\alpha+1)\bigg[a^{ij}(x)(x-x_0)_i(x-x_0)_j + \abs{z}^{2-\frac{1}{s}}(\partial_z(\delta_{\Phi}((x_0,z_0),(x,z)))- h_{\varepsilon}'(z))^2\bigg]\\
&\quad - (\delta_{\Phi}((x_0,z_0),(x,z))- h_{\varepsilon}(z))\bigg[\trace(A(x))+ 1- 2(n\Lambda+1)\psi_\varepsilon\bigg]\bigg).
\end{align*}
Using ellipticity and 
\begin{equation}\label{eq:partialz}
\partial_z \delta_{\Phi}((x_0,z_0),(x,z)) = \partial_z(\delta_{\varphi}(x_0,x) + \delta_{h}(z_0,z)) = \partial_z\delta_{h}(z_0,z),
\end{equation}
we estimate
\begin{align*}
a^{ij}(x)& \partial_{ij}\tilde{\phi} + \abs{z}^{2-\frac{1}{s}} \partial_{zz} \tilde{\phi}\\
&\geq  \alpha(\delta_{\Phi}((x_0,z_0),(x,z))- h_{\varepsilon}(z))^{-\alpha-2}\\
&\quad\bigg((\alpha+1)\bigg[2\lambda \delta_{\varphi}(x_0,x) + \abs{z}^{2-\frac{1}{s}}(\partial_z\delta_{h}(z_0,z)- h_{\varepsilon}'(z))^2\bigg]\\
&\quad - (\delta_{\Phi}((x_0,z_0),(x,z))- h_{\varepsilon}(z))(1-2\psi_\varepsilon)(n\Lambda+1)\bigg).
\end{align*}

Suppose that $z \in H_{\varepsilon}$. Since $\psi_\varepsilon(z) = 1$, we can use \eqref{eq:barrier bound} to estimate
\begin{equation}\label{eq:Case2 -He}
\begin{aligned}
a^{ij}&(x) \partial_{ij}\tilde{\phi} + \abs{z}^{2-\frac{1}{s}} \partial_{zz} \tilde{\phi}\\
&\geq \alpha(\delta_{\Phi}((x_0,z_0),(x,z))- h_{\varepsilon}(z))^{-\alpha-2}
\bigg(0+ (\delta_{\Phi}((x_0,z_0),(x,z))- h_{\varepsilon}(z))(n\Lambda+1)\bigg)\\
&\geq \alpha(n\Lambda+1)(2 + C_{2} \varepsilon)^{-\alpha-1} r^{-\alpha-1}.
\end{aligned}
\end{equation}

Next, suppose that $z \notin H_{\varepsilon}$. Since $\psi_\varepsilon(z) >0$ and $\abs{z}^{2-\frac{1}{s}} >\varepsilon_0 \abs{S_{2r}(z_0)}/\mu_h(S_{2r}(z_0))$,
we estimate 
\begin{equation}\label{eq:notin H_e-new}
\begin{aligned}
a^{ij}(x)& \partial_{ij}\tilde{\phi} + \abs{z}^{2-\frac{1}{s}} \partial_{zz} \tilde{\phi}\\
&\geq  \alpha(\delta_{\Phi}((x_0,z_0),(x,z))- h_{\varepsilon}(z))^{-\alpha-2}\\
&\quad\bigg((\alpha+1)\bigg[2\lambda \delta_{\varphi}(x_0,x) +\varepsilon_0 \frac{\abs{S_{2r}(z_0)}}{\mu_h(S_{2r}(z_0))}(\partial_z\delta_{h}(z_0,z)- h'_{\varepsilon}(z))^2\bigg]\\
&\quad - (\delta_{\Phi}((x_0,z_0),(x,z))- h_{\varepsilon}(z))(n\Lambda+1)\bigg).
\end{aligned}
\end{equation}
Since $\delta_{\Phi}((x_0,z_0),(x,z))\geq\gamma r$, we have that $\delta_{\varphi} (x_0,x)\geq \gamma r/2$ or $\delta_{h}(z_0,z) \geq \gamma r/2$.
Suppose first that $\delta_{\varphi}(x_0,x) \geq \gamma r/2$. Then 
\[
2\lambda\delta_{\varphi}(x_0,x)  +\varepsilon_0 \frac{\abs{S_{2r}(z_0)}}{\mu_h(S_{2r}(z_0))}(\partial_z\delta_{h}(z_0,z) - h'_{\varepsilon}(z))^2
\geq 2\lambda\delta_{\varphi}(x_0,x) 
\geq\lambda\gamma r.
\]
Choose $\alpha = \alpha(\gamma, n,\lambda,\Lambda,s)$ large enough to guarantee that
\[
(\alpha+1)\lambda \gamma  -(n\Lambda+1)(2 + C_{2} \varepsilon) > (n\Lambda+1)(2+ C_{2} \varepsilon).
\]
Then, from \eqref{eq:notin H_e-new} and \eqref{eq:barrier bound},
\begin{equation}\label{eq:case2-x}
\begin{aligned}
a^{ij}(x)& \partial_{ij}\tilde{\phi} + \abs{z}^{2-\frac{1}{s}} \partial_{zz} \tilde{\phi}\\
&\geq  \alpha(\delta_{\Phi}((x_0,z_0),(x,z))- h_{\varepsilon}(z))^{-\alpha-2}
\bigg((\alpha+1)\lambda\gamma r - (n\Lambda+1)(2+C_2\varepsilon)r\bigg)\\
&>  \alpha(\delta_{\Phi}((x_0,z_0),(x,z))- h_{\varepsilon}(z))^{-\alpha-2}
(n\Lambda+1)(2+ C_{2} \varepsilon)r\\
&\geq \alpha(n\Lambda+1) (2+ C_{2} \varepsilon)^{-\alpha-1}r^{-\alpha-1}.
\end{aligned}
\end{equation}
Next, suppose that $\delta_h(z_0,z) \geq \gamma r/2$. Since $S_{2r}(x_0,z_0) \subset S_{2r}(x_0) \times S_{2r}(z_0)$, we know that $\gamma r/2 \leq \delta_h(z_0,z)<2r$. 
By \eqref{eq:deriv squared-new} and Corollary \ref{lem:doubling} part (3), we obtain
\begin{align*}
2\lambda\delta_{\varphi}(x_0,x) & +\varepsilon_0 \frac{\abs{S_{2r}(z_0)}}{\mu_h(S_{2r}(z_0))}(\partial_z\delta_{h}(z_0,z) - h'_{\varepsilon}(z))^2\\
&\geq\varepsilon_0 \frac{\abs{S_{2r}(z_0)}}{\mu_h(S_{2r}(z_0))}(\partial_z\delta_{h}(z_0,z)- h'_{\varepsilon}(z))^2\\
&\geq \varepsilon_0 \frac{\abs{S_{2r}(z_0)}}{\mu_h(S_{2r}(z_0))} C_4 [\mu_h(S_{2r}(z_0))]^2\\
&\geq C_5 \varepsilon_0r
\end{align*}
for some constant $C_5 = C_5(\gamma, n ,\Lambda,s)>0$.
Let $\alpha = \alpha(\gamma,n,\lambda,\Lambda,s)>0$ be large so that
\[
(\alpha+1) C_{5} \varepsilon_0  -(n\Lambda+1)(2 + C_{2} \varepsilon) > (n\Lambda+1)(2 + C_{2} \varepsilon).
\]
Then, from \eqref{eq:notin H_e-new}, we use \eqref{eq:barrier bound} to obtain
\begin{equation}\label{eq:case2-z-new}
\begin{aligned}
a^{ij}&(x) \partial_{ij}\tilde{\phi} + \abs{z}^{2-\frac{1}{s}} \partial_{zz} \tilde{\phi}\\
&\geq  \alpha(\delta_{\Phi}((x_0,z_0),(x,z))- h_{\varepsilon}(z))^{-\alpha-2}
\bigg((\alpha+1) C_{5} \varepsilon_0r- (n\Lambda+1)(2+C_2\varepsilon)r\bigg)\\
&>  \alpha(\delta_{\Phi}((x_0,z_0),(x,z))- h_{\varepsilon}(z))^{-\alpha-2}
(n\Lambda+1)(2 + C_{2} \varepsilon)r\\
&\geq \alpha(n\Lambda+1) (2+ C_{2} \varepsilon)^{-\alpha-1}r^{-\alpha-1}.
\end{aligned}
\end{equation}

From \eqref{eq:Case2 -He}, \eqref{eq:case2-x}, and \eqref{eq:case2-z-new}, there is an $\alpha = \alpha(\gamma,n,\lambda,\Lambda,s)>0$ such that
for all $(x,z) \in [S_{2r}(x_0,z_0) \setminus S_{\gamma r}(x_0,z_0)]^{+} \setminus \{z=0\}$, we have
\[
a^{ij}(x) \partial_{ij}\tilde{\phi} + \abs{z}^{2-\frac{1}{s}} \partial_{zz} \tilde{\phi}
> \alpha(n\Lambda+1)(2+C_{2} \varepsilon)^{-\alpha-1}r^{-\alpha-1}.
\]

We define the barrier $\phi$ on $[S_{2r}(x_0,z_0) \setminus S_{\gamma r}(x_0,z_0)]^+$ by
\begin{align*}
\phi(x,z)
	&= a\alpha^{-1}(2+C_{2}\varepsilon)^{\alpha+1} r^{\alpha+1} \( \tilde{\phi}(x,z) -(1 + C_2\varepsilon)^{-\alpha}r^{-\alpha}\).
\end{align*}
For  $(x,z) \in [S_{2r}(x_0,z_0) \setminus S_{\gamma r}(x_0,z_0)]^+ \setminus \{z=0\}$, it then follows that
$a^{ij}(x) \partial_{ij}{\phi} + \abs{z}^{2-\frac{1}{s}} \partial_{zz} {\phi}> a(n\Lambda+1)$.
If $z_L = 0$, we need to check the Neumann condition. In this case, 
let $(x,0) \in [S_{2r}(x_0,z_0) \setminus S_{\gamma r}(x_0,z_0)]^+ \cap \{z=0\}$. Using \eqref{eq:barrier bound}, we see that
\begin{align*}
\partial_{z+} \phi(x,0)
	&= a(2+C_{2}\varepsilon)^{\alpha+1} r^{\alpha+1}  (\delta_{\Phi}((x_0,z_0),(x,0))- h_{\varepsilon}(0))^{-\alpha-1}(h'(z_0) + \varepsilon \mu_h(S_{2r}(z_0)))\\
	&> a(2+C_{2}\varepsilon)^{\alpha+1} r^{\alpha+1}  (2+C_{2}\varepsilon)^{-\alpha-1} r^{-\alpha-1} (h'(z_0) + \varepsilon \mu_h(S_{2r}(z_0)))\\
	&= a(h'(z_0) + \varepsilon\mu_h(S_{2r}(z_0))) >0,
\end{align*}
since $z_0\geq 0$.
Therefore, $\phi$ is a subsolution to \eqref{eq:subsoln1}.
In $[S_{r}(x_0,z_0) \setminus S_{\gamma r}(x_0,z_0)]^+$, we have
\[
\gamma r \leq \delta_\Phi((x_0,z_0),(x,z)) - h_\varepsilon(z) < (1 + C_2\varepsilon)r,
\]
so that $\phi>0$ in $[S_{r}(x_0,z_0) \setminus S_{\gamma r}(x_0,z_0)]^+$.
Choose $\varepsilon>0$ small so that $2 > 1 + C_2\varepsilon$.
Then, $\phi \leq 0$ on  $[\partial S_{2r}(x_0,z_0)]^+$. 
Indeed, for $(x,z) \in [\partial S_{2r}(x_0,z_0)]^+$, we have that
\[
- h_{\varepsilon}(z) \geq 0> (1 + C_2\varepsilon - 2)r
\]
which implies 
\[
 \delta_{\Phi}((x_0,z_0),(x,z)) - h_{\varepsilon}(z) = 2r - h_{\varepsilon}(z)  > (1 + C_2\varepsilon)r.
\]
Thus, $\phi(x,z) \leq 0$.
Lastly, let $(x,z) \in [\partial S_{\gamma r}(x_0,z_0)]^+$ and observe that
\begin{align*}
\phi(x,z)
	&= a\alpha^{-1}(2+C_{2}\varepsilon)^{\alpha+1} r^{\alpha+1} \( (\gamma r - h_{\varepsilon}(z))^{-\alpha} -(1 + C_2\varepsilon)^{-\alpha}r^{-\alpha}\)\\
	&\leq a \alpha^{-1}(2+C_{2}\varepsilon_2)^{\alpha+1} r^{\alpha+1} \( (\gamma r +0)^{-\alpha} -0\)= Car
\end{align*}
for $C = C(\gamma,n,\lambda,\Lambda,s)>0$.

\medskip
\noindent
\underline{\bf Case 3}: $z_0= 0$ and  $0 < s \leq 1/2$. 

\smallskip

For the barrier constructed in Case 1, the inequality for the Neumann condition was not strict for $z_0 = 0$. 
We will add a function $g_{\varepsilon}$ to the quasi-distance function $\delta_\Phi$ to adjust the barrier as we did in Case 2. 

Let $(x,z) \in [S_{2r}(x_0,0) \setminus S_{\gamma r}(x_0,0)]^+$.
Since $(x,z) \in S_{2r}(x_0,0) \subset S_{2r}(x_0) \times S_{2r}(0)$, 
we know that $z \in S_{2r}(0) = B_{q_s2^sr^2}(0)$ by Remark \ref{rem:r^s new}. 
That is, 
\begin{equation}\label{eq:c2-bar}
\abs{z} < q_s2^s r^s = \bar{C}_2r^s. 
\end{equation}
Also, since $2- \frac{1}{s} \leq 0$,
\begin{equation}\label{eq:1/h''}
\abs{z}^{2-\frac{1}{s}} \geq \bar{C}_2^{\frac{2s-1}{s}}r^{2s-1}.
\end{equation}

Given $\varepsilon>0$, define $g_{\varepsilon}$ in $[S_{2r}(0)]^{+}$ by 
\[
g_{\varepsilon}(z) = \varepsilon r^{1-s}z - \varepsilon \bar{C}_2r.
\]
For all $z \in [S_{2r}(0)]^{+}$, we have that $g_{\varepsilon} \leq 0$ by \eqref{eq:c2-bar}.
We also have that
\begin{align*}
\abs{g_\varepsilon (z)} 
	= \varepsilon \bar{C}_2r -\varepsilon r^{1-s}z
	\leq \bar{C}_2 \varepsilon r,\quad\hbox{and}\quad g_{\varepsilon}'(z) = \varepsilon r^{1-s} >0.
\end{align*} 

Let $z$ be such that
$\gamma r/2 \leq \delta_h(0,z) < 2r$.
As in Case 2 above, since $z \in S_{2r}(0) = B_{\bar{C}_2r^s}(0)$, we can use the convexity of $\delta_h(0,z)$ in the variable $z$ to obtain
\[
\abs{\partial_z\delta_h(0,z) } 
	\geq \frac{\delta_h(0,z)}{\abs{z}}
	\geq \frac{\gamma r/2}{\bar{C}_2r^s}
	= \bar{C}_{3} r^{1-s}
\]
for a constant $\bar{C}_3= \bar{C}_3(\gamma,s)$.
Choose $\varepsilon = \varepsilon(\gamma,s)>0$ small so that $\varepsilon < \bar{C}_3$. Then, 
\[
\abs{\partial_z \delta_h(0,z) - g'_{\varepsilon}(z)}
	\geq \abs{\partial_z \delta_h(0,z)} - \abs{g'_{\varepsilon}(z)}
	\geq (\bar{C}_3 - \varepsilon)r^{1-s}>0
\]
and
\begin{equation} \label{eq:deriv squared - 0}
(\partial_z \delta_h(0,z) - g'_{\varepsilon}(z))^2
	\geq (\bar{C}_3 - \varepsilon)^2r^{2-2s}
	= \bar{C}_4 r^{2-2s}
\end{equation}
for a constant $\bar{C}_4 = \bar{C}_4(\gamma,s)>0$.

Let $(x,z) \in [S_{2r}(x_0,0) \setminus S_{\gamma r}(x_0,0)]^+$.
Since $-g_{\varepsilon} \geq 0$, we have that
\begin{equation} \label{eq:z0=0 estimate}
\gamma r 
	\leq \delta_{\Phi}((x_0,0),(x,z))
	\leq \delta_{\Phi}((x_0,0),(x,z))- g_{\varepsilon}(z) 
	< 2r + \varepsilon \bar{C}_2 r
	= (2 +\varepsilon \bar{C}_2)r. 
\end{equation}

We define a function $\bar{\phi}$ on $[S_{2r}(x_0,0) \setminus S_{\gamma r}(x_0,0)]^+$ by
\begin{align*}
\bar{\phi}(x,z)
	&=(\delta_{\Phi}((x_0,0),(x,z)) - g_{\varepsilon}(z))^{-\alpha}.
\end{align*}
Let $(x,z) \in [S_{2r}(x_0,0) \setminus S_{\gamma r}(x_0,0)]^+ \setminus \{z=0\}$. Using ellipticity, \eqref{eq:partialz}, \eqref{eq:1/h''}, and \eqref{eq:z0=0 estimate}, we estimate the equation for $\bar{\phi}$ as follows:
\begin{equation}\label{eq:case3-equation}
\begin{aligned}
a^{ij}&(x) \partial_{ij}\bar{\phi} + \abs{z}^{2-\frac{1}{s}} \partial_{zz} \bar{\phi}\\
&\geq  \alpha(\delta_{\Phi}((x_0,0),(x,z))- g_{\varepsilon}(z))^{-\alpha-2}\\
&\quad\bigg((\alpha+1)\bigg[2\lambda \delta_{\varphi}(x_0,x) + \bar{C}_2^{\frac{2s-1}{s}}r^{2s-1}(\partial_z\delta_{h}(0,z)- g_{\varepsilon}'(z))^2\bigg]\\
&\quad - (\delta_{\Phi}((x_0,0),(x,z))- g_{\varepsilon}(z))(n\Lambda+1)\bigg)\\
&\geq  \alpha(\delta_{\Phi}((x_0,0),(x,z))- g_{\varepsilon}(z))^{-\alpha-2}\\
&\quad\bigg((\alpha+1)\bigg[2\lambda \delta_{\varphi}(x_0,x) + \bar{C}_2^{\frac{2s-1}{s}}r^{2s-1}(\partial_z\delta_{h}(0,z) - g_{\varepsilon}'(z))^2\bigg]\\
&\quad - (n\Lambda+1)(2+\varepsilon \bar{C}_2)r\bigg).
\end{aligned}
\end{equation}
Since $\delta_{\Phi}((x_0,0),(x,z)) \geq \gamma r$, we know that $\delta_{\varphi}(x_0,x) \geq \gamma r/2$ or $\delta_{h}(0,z) \geq \gamma r/2$.
Suppose first that $\delta_{\varphi}(x_0,x) \geq \gamma r/2$. Then 
\[
2\lambda\delta_{\varphi}(x_0,x)  +\bar{C}_2^{\frac{2s-1}{s}}r^{2s-1}(\partial_z\delta_{h}(0,z)- g_{\varepsilon}'(z))^2
\geq 2\lambda\delta_{\varphi}(x_0,x) 
\geq\lambda\gamma r.
\]
Choose $\alpha = \alpha(\gamma, n,\lambda,\Lambda,s)$ large enough to guarantee that
\[
(\alpha+1)\lambda \gamma  -(n\Lambda+1)(2 + \bar{C}_2 \varepsilon) > (n\Lambda+1)(2+ \bar{C}_2 \varepsilon).
\]
Then, from \eqref{eq:case3-equation} and using \eqref{eq:z0=0 estimate}, we have that
\begin{equation}\label{eq:case3-x}
\begin{aligned}
a^{ij}(x)& \partial_{ij}\bar{\phi} + \abs{z}^{2-\frac{1}{s}} \partial_{zz} \bar{\phi}\\
&\geq  \alpha(\delta_{\Phi}((x_0,0),(x,z))- g_{\varepsilon}(z))^{-\alpha-2}
\big((\alpha+1)\lambda\gamma r - (n\Lambda+1)(2+ \bar{C}_2\varepsilon)r\big)\\
&>  \alpha(\delta_{\Phi}((x_0,0),(x,z))- g_{\varepsilon}(z))^{-\alpha-2}
(n\Lambda+1)(2+ \bar{C}_2\varepsilon)r\\
&\geq \alpha(n\Lambda+1) (2+ \bar{C}_2\varepsilon)^{-\alpha-1}r^{-\alpha-1}.
\end{aligned}
\end{equation}
Next, suppose that $\delta_h(0,z) \geq \gamma r/2$. We further know that
$\gamma r/2\leq \delta_{h}(0,z) < 2r$, so, by \eqref{eq:deriv squared - 0},
\begin{align*}
2\lambda\delta_{\varphi}(x_0,x)  &+\bar{C}_2^{\frac{2s-1}{s}}r^{2s-1}(\partial_z\delta_{h}(0,z) - g_{\varepsilon}(z)))^2\\
&\geq 0+\bar{C}_2^{\frac{2s-1}{s}}r^{2s-1}\bar{C}_{4} r^{2-2s} 
= \bar{C}_{4} \bar{C}_2^{\frac{2s-1}{s}} r.
\end{align*}
Let $\alpha = \alpha(\gamma,n,\lambda,\Lambda,s)>0$ be large so that
\[
(\alpha+1) \bar{C}_{4} \bar{C}_2^{\frac{2s-1}{s}}  -(n\Lambda+1)(2 + \bar{C}_{2} \varepsilon) > (n\Lambda+1)(2 + \bar{C}_{2} \varepsilon).
\]
Then, from \eqref{eq:case3-equation} and using \eqref{eq:z0=0 estimate},
\begin{equation}\label{eq:case3-z}
\begin{aligned}
a^{ij}&(x) \partial_{ij}\bar{\phi} + \abs{z}^{2-\frac{1}{s}} \partial_{zz} \bar{\phi}\\
&\geq  \alpha(\delta_{\Phi}((x_0,0),(x,z))- g_{\varepsilon}(z))^{-\alpha-2}
\bigg((\alpha+1) \bar{C}_{4}\bar{C}_2^{\frac{2s-1}{s}}r- (n\Lambda+1)(2r+\bar{C}_2\varepsilon)r\bigg)\\
&>  \alpha(\delta_{\Phi}((x_0,0),(x,z))- g_{\varepsilon}(z))^{-\alpha-2}
(n\Lambda+1)(2 + \bar{C}_{1} \varepsilon)r\\
&\geq \alpha(n\Lambda+1) (2+ \bar{C}_{2} \varepsilon)^{-\alpha-1}r^{-\alpha-1}.
\end{aligned}
\end{equation}

From \eqref{eq:case3-x} and \eqref{eq:case3-z},  there is an $\alpha = \alpha(\gamma,n,\lambda,\Lambda,s)>0$ such that
for all $(x,z) \in [S_{2r}(x_0,0) \setminus S_{\gamma r}(x_0,0)]^{+} \setminus \{z=0\}$, we have
\[
a^{ij}(x) \partial_{ij}\bar{\phi} + \abs{z}^{2-\frac{1}{s}} \partial_{zz} \bar{\phi}
> \alpha(n\Lambda+1)(2+\bar{C}_{2} \varepsilon)^{-\alpha-1}r^{-\alpha-1}.
\]

We define the barrier $\phi$ on $[S_{2r}(x_0,0) \setminus S_{\gamma r}(x_0,0)]^+$ by
\begin{align*}
\phi(x,z)
	&= a \alpha^{-1}(2+\bar{C}_{2}\varepsilon)^{\alpha+1} r^{\alpha+1} \( \bar{\phi}(x,z) -(1 + \bar{C}_1\varepsilon)^{-\alpha}r^{-\alpha}\).
\end{align*}
For  $(x,z) \in [S_{2r}(x_0,0) \setminus S_{\gamma r}(x_0,0)]^+ \setminus \{z=0\}$, it follows that
$a^{ij}(x) \partial_{ij}{\phi} + \abs{z}^{2-\frac{1}{s}} \partial_{zz} {\phi}> a(n\Lambda+1)$.
If~$(x,0) \in [S_{2r}(x_0,0) \setminus S_{\gamma r}(x_0,0)]^+ \cap \{z=0\}$, by \eqref{eq:z0=0 estimate},
\begin{align*}
\partial_{z+} \phi(x,0)
	&= a(2+\bar{C}_{2}\varepsilon)^{\alpha+1} r^{\alpha+1}  (\delta_{\Phi}((x_0,0),(x,0))- g_{\varepsilon}(0))^{-\alpha-1}\varepsilon r^{1-s}\\
	&\geq a(2+\bar{C}_{2}\varepsilon)^{\alpha+1} r^{\alpha+1} (2+\bar{C}_{2}\varepsilon)^{-\alpha-1} r^{-\alpha-1} \varepsilon r^{1-s}\\
	&= a\varepsilon r^{1-s} >0.
\end{align*}
Therefore, $\phi$ defined in $[S_{2r}(x_0,0) \setminus S_{\gamma r}(x_0,0)]^+$ is a subsolution to \eqref{eq:subsoln1}.
One can also check that $\phi >0$ in $[S_r(x_0,0) \setminus S_{\gamma r}(x_0,0)]^+$ and 
that $\phi \leq 0$ on $[\partial S_{2r}(x_0,0)]^+$ when $\varepsilon = \varepsilon(\gamma,s)$ is small enough to guarantee that $2>1 + \bar{C}_2 \varepsilon$. Moreover, there is a constant $C = C(\gamma,n,\lambda,\Lambda,s)>0$ such that
$\phi(x,z) \leq Car$ on $[\partial S_{\gamma r}(x_0,0)]^+$.

\medskip
\noindent
\underline{\bf Case 4}: $z_0\leq 0$ and  $0 < s <1$.

\smallskip

By \eqref{eq:delta_h anti-symmetry}, if $(x,z) \in [S_{2r}(x_0,z_0) \setminus S_{\gamma r}(x_0,z_0)]^-$, then $(x,-z) \in [S_{2r}(x_0,-z_0) \setminus S_{\gamma r}(x_0,-z_0)]^+$.
Define $\psi$ in $[S_{2r}(x_0,z_0) \setminus S_{\gamma r}(x_0,z_0)]^-$ to be the even reflection across $\{z=0\}$ of the solution $\phi$ to \eqref{eq:subsoln1} in $[S_{2r}(x_0,-z_0) \setminus S_{\gamma r}(x_0,-z_0)]^+$:
\[
\psi(x,z) = \phi(x,-z), \quad \hbox{for}~(x,z) \in [S_{2r}(x_0,z_0) \setminus S_{\gamma r}(x_0,z_0)]^-.
\]

Since $D^2\psi(x,z) = D^2\phi(x,-z)$, we know, for  $(x,z) \in [S_{2r}(x_0,z_0) \setminus S_{\gamma r}(x_0,z_0)]^- \setminus \{z=0\}$, that
\begin{align*}
a^{ij}(x) \partial_{ij}\psi(x,z) + \abs{z}^{2-\frac{1}{s}} \partial_{zz} \psi(x,z)
	&= a^{ij}(x) \partial_{ij}\phi(x,-z) + \abs{z}^{2-\frac{1}{s}} \partial_{zz} \phi(x,-z) \\
	&> a(n\Lambda+1).
\end{align*}
For $(x,0) \in  [S_{2r}(x_0,z_0) \setminus S_{\gamma r}(x_0,z_0)]^- \cap \{z=0\}$, we have
$-\partial_{z-}\psi(x,0)= \partial_{z+}\phi(x,0) >0$.
Therefore, $\psi$ is a subsolution to \eqref{eq:subsoln1-neg}. It is straightforward to check that $\psi >0$ in $[ S_{r}(x_0,z_0) \setminus S_{\gamma r}(x_0,z_0)]^-$ and that $\psi \leq 0$  on $[\partial S_{2 r}(x_0,z_0)]^-$.
Lastly, if $(x,z) \in [\partial S_{\gamma r}(x_0,z_0)]^-$, then $(x,-z) \in [\partial S_{\gamma r}(x_0,-z_0)]^+$. This gives the desired estimate
$\psi(x,z) = \phi(x,-z) \leq Car$ for $(x,z) \in [\partial S_{\gamma r}(x_0,z_0)]^-$.
\end{proof}

%%%%%%%%%%%%%%%%%%%%%%%%%%%%%%%%
\section{Localization lemma}\label{sec:lem2}
%%%%%%%%%%%%%%%%%%%%%%%%%%%%%%%%

In this section, we prove the main localization estimate, Lemma \ref{Lem:Cc}.
We show that if a supersolution $U$ can be touched from below with a paraboloid $P$ of opening $a>0$ in a cube $Q_r$, then 
the set in which $U$ can be touched from below by paraboloids of increased opening $Ca>0$, where $C=C(n,\lambda,\Lambda,s)>0$,
in a smaller cube $Q_{\eta r}$ makes up a universal proportion of $Q_r$.
To prove this result, we first use the barrier $\phi$ constructed
in Lemma \ref{lem:subsoln} to control how $U$ detaches from a touching paraboloid $P$, see Lemma \ref{claim:Car}.

Before stating the main lemma of this section, we need to introduce some notation.
First, we define a constant $\hat{K}_2$ to be large enough so that for any $(x_0,z_0), (\tilde{x},\tilde{z}) \in \R^{n+1}$ and $R>0$,
if $Q_r(x_0,z_0) \subset Q_R(\tilde{x},\tilde{z})$, then  $Q_{2(n+1)r}(x_0,z_0) \subset Q_{\hat{K}_2R}(\tilde{x},\tilde{z})$. 
By Lemma \ref{lem:ordering}, we know that if $Q_r(x_0,z_0) \subset Q_R(\tilde{x},\tilde{z})$ then $r \leq R$. 
If $(x,z) \in Q_{2(n+1)r}(x_0,z_0)$ then, by the quasi-triangle inequality (see Notation \ref{note:2}),
\begin{align*}
\delta_\varphi(\tilde{x},x) 
	&\leq K\(\delta_{\varphi}(\tilde{x},x_0) +\delta_{\varphi}(x_0,x) \) < K\(R + 2(n+1)r\) < K(1+2(n+1))R\\
\delta_h(\tilde{z},z) 
	&\leq K\(\delta_{h}(\tilde{z},z_0) +\delta_{h}(z_0,z) \)< K(1+2(n+1))R.
\end{align*}
We then take $\hat{K}_2 = \hat{K}_2(n,s)$ as
\begin{equation}\label{eq:K_2}
\hat{K}_2 = (2n+3)K.
\end{equation}

Let $\hat{K}_3 = \hat{K}_3(n,s)$ be given by
\begin{equation} \label{eq:K_3}
\hat{K}_3 = \theta^2\hat{K}_2.
\end{equation}

If $Q_{\hat{K}_2R}(\tilde{x},\tilde{z}) \cap \{z=0\} \not= \varnothing$, then $0 \in S_{\hat{K}_2R}(\tilde{z})$ and, by the engulfing property,
\[
Q_{\hat{K}_2R}(\tilde{x},\tilde{z})
	= Q_{\hat{K}_2R}(\tilde{x}) \times S_{\hat{K}_2R}(\tilde{z})
	\subset Q_{\theta\hat{K}_2R}(\tilde{x}) \times S_{\theta\hat{K}_2R}(0)
	= Q_{\theta\hat{K}_2R}(\tilde{x},0)
\]	
and 
\[
Q_{\theta\hat{K}_2R}(\tilde{x},0)
	 =Q_{\theta\hat{K}_2R}(\tilde{x}) \times S_{\theta\hat{K}_2R}(0)
	 \subset Q_{\theta^2\hat{K}_2R}(\tilde{x}) \times S_{\theta^2\hat{K}_2R}(\tilde{z})
	= Q_{\hat{K}_3R}(\tilde{x},\tilde{z}). 
\]
We define a vertex set $B_v \subset \overline{Q}_{\hat{K}_3R}(\tilde{x},\tilde{z})$ by
\[
B_{v} = 
\begin{cases}
\overline{Q}_{\hat{K}_2R}(\tilde{x},\tilde{z}) & \hbox{if}~\tilde{z} = 0~\hbox{or if}~ \overline{Q}_{\hat{K}_2R}(\tilde{x},\tilde{z}) \cap \{z=0\} = \varnothing\\
\overline{Q}_{\theta\hat{K}_2R}(\tilde{x},0) & \hbox{if}~\tilde{z} \not= 0~\hbox{and }~ \overline{Q}_{\hat{K}_2R}(\tilde{x},\tilde{z}) \cap \{z=0\} \not= \varnothing,
\end{cases}
\]
so that $B_v$ is symmetric with respect to $\{z=0\}$ if $\overline{Q}_{\hat{K}_2R}(\tilde{x},\tilde{z}) \cap \{z=0\} \not= \varnothing$.

Define the contact set $A_{a,R}$ for a continuous function $U$ on $Q_{\hat{K}_2R}(\tilde{x},\tilde{z})$ by
\begin{equation}\label{eq:AaR}
\begin{aligned}
A_{a,R}:=
\bigg\{& (x,z) \in Q_{\hat{K}_2R}(\tilde{x},\tilde{z})   : U(x,z) \leq aR~ \hbox{and there is}~(x_v,z_v) \in B_v~\hbox{such that }\\
&\quad\quad\hbox{$U$ can be touched from below at $(x,z)$ in $Q_{\hat{K}_3R}(\tilde{x},\tilde{z})$}\\
&\quad \quad \hbox{by a paraboloid of opening $a>0$ with vertex $(x_v,z_v)$}\bigg\}.
\end{aligned}
\end{equation}

\begin{lem}
The contact set $A_{a,R}$ is closed in $Q_{\hat{K}_2R}(\tilde{x},\tilde{z})$.
\end{lem}

\begin{proof}
Let $(x_k,z_k) \in A_{a,R}$ and $(x_0,z_0) \in Q_{\hat{K}_2R}(\tilde{x},\tilde{z})$ be such that $(x_k,z_k) \to (x_0,z_0)$. 
Since $U(x_k,z_k) \leq aR$ and $U$ is continuous, $U(x_0,z_0) \leq aR$.
By the same argument as in the proof of Theorem \ref{lem:ABP-super} with $B = B_v$, 
we can touch $U$ from below in $Q_{\hat{K}_3R}(\tilde{x},\tilde{z})$ at $(x_0,z_0)$ by a paraboloid $P$ of opening $a>0$ with vertex $(x_v^0,z_v^0) \in B_v$. Therefore, $(x_0,z_0) \in A_{a,R}$ which shows that $A_{a,R}$ is closed in $Q_{\hat{K}_2R}(\tilde{x},\tilde{z})$.
\end{proof}

\begin{lem}\label{claim:Car}
Fix $0 < \gamma <1$.
Assume that $\Omega$ is a bounded domain and that
$a^{ij}(x):\Omega \to \R$ are bounded, measurable functions that satisfy \eqref{eq:ellipticity}.
For a cube $Q_R = Q_{R}(\tilde{x},\tilde{z}) \subset \R^{n+1}$,
consider a cube $Q_{\hat{K}_3R} = Q_{\hat{K}_3R}(\tilde{x},\tilde{z})$ where $\hat{K}_3$ is as in \eqref{eq:K_3}.
Let  $f \in L^{\infty}(Q_{\hat{K}_3R} \cap \{z=0\})$ be nonnegative.
Suppose 
$U \in C^2(Q_{\hat{K}_3R} \setminus \{z=0\}) \cap C(Q_{\hat{K}_3R})$ such that $U$ is symmetric across $\{z=0\}$ and  $U_{z+} \in C(Q_{\hat{K}_3R} \cap \{z\geq0\})$
is a supersolution to
\[
\begin{cases}
a^{ij}(x) \partial_{ij}U + \abs{z}^{2-\frac{1}{s}}\partial_{zz}U \leq  0 & \hbox{in}~Q_{\hat{K}_3R}  \cap \{z\not=0\}\\
-\partial_{z+}U \geq f & \hbox{on}~Q_{\hat{K}_3R}  \cap \{z=0\}.
\end{cases}
\]
Assume that $Q_r(x_0,z_0) \subset Q_R$ for some point $(x_0,z_0)$ such that $z_0 \geq 0$. 
Suppose that $U$ is touched from below at $(x_1,z_1) \in [S_r(x_0,z_0)]^+ \cap A_{a,R}$ in $Q_{\hat{K}_3R}$ by a paraboloid $P$ of opening $a>0$ with vertex $(x_v,z_v)$ such that $z_v \geq 0$. 
Then, there exists a constant $C = C(\gamma,n,\lambda,\Lambda)>0$ and a point $(x_2,z_2) \in [\overline{S}_{\gamma r}(x_0,z_0)]^+$ such that
\[
U(x_2,z_2) - P(x_2,z_2) \leq Car.
\]
\end{lem}

\begin{proof}
If $(x_1,z_1) \in [\overline{S}_{\gamma r}(x_0,z_0)]^+$, 
\[
U(x_1,z_1) - P(x_1,z_1) = 0 \leq Car
\]
for all $C>0$, so we can take $(x_2,z_2) = (x_1,z_1)$. 
Therefore, we assume for the remainder of the proof that $(x_1,z_1) \in [S_r(x_0,z_0) \setminus \overline{S}_{\gamma r}(x_0,z_0)]^+$.

Let $W= U - P$. 
For $(x,z) \in Q_{\hat{K}_3 R} \setminus \{z=0\}$, we have that
\begin{align*}
a^{ij}(x) \partial_{ij}P(x,z) + \abs{z}^{2-\frac{1}{s}} \partial_{zz}P(x,z)
	&=  -a(\trace(A(x)) + 1) 
	\geq -a(n\Lambda +1)
\end{align*}
which implies
\begin{align*}
a^{ij}(x) \partial_{ij}W(x,z) + \abs{z}^{2-\frac{1}{s}} \partial_{zz}W(x,z)
	&\leq a(n\Lambda +1).
\end{align*}
Since $z_v \geq 0$, we also have that
\[
-\partial_{z+}W(x,0) \geq  f(x) + a h'(z_v) \geq 0.
\]

Let $\phi$ be the subsolution to \eqref{eq:subsoln1} in $[S_{2r}(x_0,z_0) \setminus S_{\gamma r}(x_0,z_0)]^+$. 
By the choice of $\hat{K}_2$ in \eqref{eq:K_2}, we have that $Q_{r}(x_0,z_0) \subset Q_R$ implies
\[
S_{2r}(x_0,z_0) \subset Q_{2r}(x_0,z_0) \subset Q_{\hat{K}_2R} \subset Q_{\hat{K}_3R}
\]
Therefore, $W-\phi$ satisfies 
\begin{equation}\label{eq:equationW}
\begin{cases}
a^{ij}(x) \partial_{ij}(W-\phi) + \abs{z}^{2-\frac{1}{s}}\partial_{zz} (W-\phi)< 0 & \hbox{in}~[S_{2r}(x_0,z_0) \setminus {S}_{\gamma r}(x_0,z_0)]^+\cap \{z\ne 0\}\\
-\partial_{z+}(W-\phi)(x,0) > 0 & \hbox{on}~[S_{2r}(x_0,z_0) \setminus {S}_{\gamma r}(x_0,z_0)]^+\cap \{z=0\}.
\end{cases}
\end{equation}

Let $(x_2,z_2) \in [\overline{S}_{2r}(x_0,z_0) \setminus S_{\gamma r}(x_0,z_0)]^+$ be such that
\[
W(x_2,z_2) - \phi(x_2,z_2) = \min_{[\overline{S}_{2r}(x_0,z_0) \setminus S_{\gamma r}(x_0,z_0)]^+} (W-\phi).
\]
By the maximum principle (see \cite[Theorem 3.1]{GilbargTrudinger}), 
the minimum of $W - \phi$ occurs on the boundary $\partial[S_{2r}(x_0,z_0) \setminus S_{\gamma r}(x_0,z_0)]^+$. 
That is,
\[
	(x_2,z_2) \in [\partial S_{2r}(x_0,z_0)]^+ \cup [\partial S_{\gamma r}(x_0,z_0)]^+ \cup [(S_{2r}(x_0,z_0) \setminus S_{\gamma r}(x_0,z_0)) \cap \{z=0\}].
\]
We claim that $(x_2,z_2) \in [\partial S_{\gamma r}(x_0,z_0)]^+$.

First, we will show that $(x_2,z_2) \notin [\partial S_{2r}(x_0,z_0)]^+$.
Since $(x_1,z_1) \in [S_r(x_0,z_0)]^+$, we know that $\phi(x_1,z_1)>0$ which implies
$W(x_1,z_1) -\phi(x_1,z_1) = 0-\phi(x_1,z_1)<0$. 
Moreover, since $\phi \leq 0$ on $[\partial S_{2r}(x_0,z_0)]^+$, we have that
$W(x,z) - \phi(x,z) \geq 0$ on $[\partial S_{2r}(x_0,z_0)]^+$.
Therefore, the minimum is strictly negative and cannot occur on $[\partial S_{2r}(x_0,z_0)]^+$.

If  $[S_{2r}(x_0,z_0)]^+ \cap \{z=0\} = \varnothing$, then our claim holds.
Suppose that $[S_{2r}(x_0,z_0)]^+ \cap \{z=0\} \not= \varnothing$. Assume,
by way of contradiction, that the minimum occurs on $[S_{2r}(x_0,z_0) \setminus S_{\gamma r}(x_0,z_0)]^+ \cap \{z=0\}$, i.e. $z_2 = 0$. Then
$ -\partial_{z+}(W-\phi)(x_2,0) \leq 0$, which contradicts \eqref{eq:equationW}.
Therefore, it must be that the minimum occurs at $(x_2,z_2) \in [\partial S_{\gamma r}(x_0,z_0)]^+ \subset [\overline{S}_{\gamma r}(x_0,z_0)]^+$.

It follows from Lemma \ref{lem:subsoln} that $\phi(x_2,z_2) \leq Car$ for $C = C(n,\lambda,\Lambda, \gamma) >0$. 
Since $W (x_2,z_2) - \phi(x_2,z_2) <0$, this implies that
\[
U(x_2,z_2) - P(x_2,z_2)  = W(x_2,z_2)<  \phi(x_2,z_2) \leq Car.
\]
\end{proof}

\begin{rem}
An analogue of Lemma \ref{claim:Car} with $z_0,z_1,z_v \leq 0$ can be similarly proved using the subsolution $\phi$ to \eqref{eq:subsoln1-neg} in $[S_{2r}(x_0,z_0) \setminus S_{\gamma r}(x_0,z_0)]^-$. 
\end{rem}

To state the main result of this section, we define positive constants $K_0>1$ and $\eta <1$ by
\begin{equation}\label{eq:K0 and eta}
K_0 = 2K^2+2K \quad \hbox{and} \quad \eta = \frac{1}{K^2( 2KK_0 +1)}.
\end{equation}

\begin{lem}\label{Lem:Cc}
Fix $a>0$. 
Assume that $\Omega$ is a bounded domain and that
$a^{ij}(x):\Omega \to \R$ are bounded, measurable functions that satisfy \eqref{eq:ellipticity}.
For a cube $Q_R = Q_{R}(\tilde{x},\tilde{z}) \subset \R^{n+1}$,
 consider $Q_{\hat{K}_3R} = Q_{\hat{K}_3R}(\tilde{x},\tilde{z})$
 where $\hat{K}_3$ is as in \eqref{eq:K_3}.
 Let $f \in L^{\infty}(Q_{\hat{K}_3R} \cap \{z=0\})$ be nonnegative.
Suppose 
$U \in C^2(Q_{\hat{K}_3R} \setminus \{z=0\}) \cap C(Q_{\hat{K}_3R})$ such that $U$ is symmetric across $\{z=0\}$ and  $U_{z+} \in C(Q_{\hat{K}_3R} \cap \{z\geq0\})$ 
is a supersolution to
\[
\begin{cases}
a^{ij}(x) \partial_{ij}U + \abs{z}^{2-\frac{1}{s}}\partial_{zz}U \leq  0 & \hbox{in}~Q_{\hat{K}_3R}  \cap \{z\not=0\}\\
-\partial_{z+}U \geq f & \hbox{on}~Q_{\hat{K}_3R}  \cap \{z=0\}.
\end{cases}
\]
Let $Q_r(x_0,z_0)$ be such that
\[
\overline{Q}_r(x_0,z_0) \subset Q_R \quad \hbox{and} \quad \overline{Q}_r(x_0,z_0) \cap A_{a,R} \not= \varnothing.
\]
There exists positive constants 
$C = C(n,\lambda,\Lambda,s)>1$ and $c = c(n,\lambda,\Lambda,s)<1$ such that
\[
\mu_{\Phi}(A_{Ca,R} \cap Q_{\eta r}(x_0,z_0)) \geq c \mu_{\Phi}(Q_r(x_0,z_0)).
\]
where $\eta = \eta (n,s)<1$ is as in \eqref{eq:K0 and eta}.
\end{lem}

\begin{rem}\label{rem:lem2}
Once the existence of $C = C(n,\lambda,\Lambda,s)>1$ has been established in Lemma \ref{Lem:Cc}, one can always take $C$ larger.
Indeed, if $C'>C$ then, by Lemma  \ref{lem:narrower opening}, we have that $A_{Ca,R} \subset A_{C'a,R}$.
\end{rem}

\begin{proof}[Proof of Lemma \ref{Lem:Cc}]
Without loss of generality, we can assume that $ Q_r(x_0,z_0) \cap A_{a,R} \not=\varnothing$. Otherwise, we replace $r$ by $r + \varepsilon$ and
then take the limit as $\varepsilon \to 0^+$ at the end.
Let $(x_1,z_1) \in  Q_r(x_0,z_0) \cap A_{a,R}$.

Since $(x_1,z_1) \in A_{a,R}$, there is a paraboloid $P$ of opening $a>0$ with vertex $(x_v,z_v) \in B_v$ 
that touches $U$ from below in $Q_{\hat{K}_3R}$ at $(x_1,z_1)$. 
We write $P$ as
\[
P(x,z) = -a \delta_{\Phi}((x_v,z_v),(x,z)) + a \delta_{\Phi}((x_v,z_v),(x_1,z_1)) + U(x_1,z_1).
\]

As for $z_0$, it must be that either $z_0 \geq 0$ or $z_0 <0$. 
We may assume that $z_1$ has the same sign as $z_0$, meaning that $z_0, z_1\geq 0$ or that  $z_0, z_1 \leq 0$.
Indeed, suppose that $z_0 \geq 0$ and $z_1<0$. 
If $\overline{Q}_{\hat{K}_2R}\cap \{z=0\} = \varnothing$, this is a contradiction. 
If $\overline{Q}_{\hat{K}_2R}\cap \{z=0\} \not= \varnothing$ then,
by Lemma \ref{lem:pos vertex}, $\tilde{P}(x,z) = P(x,-z)$ touches $U$ from below in $Q_{\hat{K}_3R}$ at $(x_1,-z_1)$ with vertex $(x_v,-z_v) \in B_v$.
Since 
 \begin{align*}
  \delta_h(z_0,-z_1)
 	&= h(z_1) - h(z_0) + h'(z_0)z_1 + h'(z_0)z_0 \\
 	&< h(z_1) - h(z_0) - h'(z_0)z_1 + h'(z_0)z_0 \quad \hbox{since}~z_0\geq0~\hbox{and}~z_1 < 0 < -z_1\\
	&= \delta_h(z_0,z_1)
	<r,
\end{align*}
it follows that $(x_1,-z_1) \in Q_r(x_0,z_0) \cap A_{a,R}$.
We proceed with the proof of the lemma using $\tilde{P}$ and $-z_1>0$ in place of $P$ and $z_1<0$. 
The argument for $z_0 \leq 0$ and $z_1>0$ follows similarly.

Hence, without loss of generality, let us assume that $z_0,z_1 \geq 0$. Then, $z_v \geq 0$.
Indeed, if $z_1>0$, then by Lemma \ref{lem:pos vertex}, we know that $z_v \geq 0$. 
If $z_1 = 0$, then, since $f \geq 0$, by Lemma \ref{lem:vertices on z=0}, it must be that $f(x_1) = 0$ and, consequently, $z_v = 0$.

Let $\gamma = \eta/(2 \theta^2)$. 
Note that $(x_1,z_1) \in  Q_r(x_0,z_0)\subset S_{(n+1)r}(x_0,z_0)$. 
We apply Lemma \ref{claim:Car} with ${r}_0 = (n+1)r$ and ${\gamma}_0 = \gamma/(n+1)$ to find a point 
\[
(x_2,z_2) 
\in [\overline{S}_{\gamma_0 r_0}(x_0,z_0)]^+
 = [\overline{S}_{\gamma r}(x_0,z_0)]^+
\subset \overline{S}_{\gamma r}(x_0,z_0)
\]
and a constant $C = C(n,\lambda,\Lambda,s)>0$ such that
\[
U(x_2,z_2) - P(x_2,z_2) \leq Car.
\]

Let $\alpha = \eta/(2\theta^3)<1$ and let $C' = C'(n,\lambda,\Lambda,s)>1$ be a large constant, to be determined.
Slide from below the family of paraboloids
\begin{equation}\label{eq:family}
\bar{P}(x,z) = P(x,z) - C'a \delta_{\Phi}((\bar{x}_v,\bar{z}_v),(x,z)) + d, \quad\hbox{for}~(\bar{x}_v,\bar{z}_v)\subset S_{\alpha r}(x_2,z_2)
\end{equation}
until they touch the graph of $U$ in $Q_{\hat{K}_3R}$ for the first time. It is clear that
$$\bar{P}(x,z)=  -a \delta_{\varphi}(x_v,x)  - C'a \delta_{\varphi}(\bar{x}_v,x) 
-a \delta_{h}(z_v,z) - C'a \delta_{h}(\bar{z}_v,z)+ d'$$
for some constant $d'$.
Let $\xi \in \R$ be such that
\[
h'(\xi) = \frac{ h'(z_v)+C'h'(\bar{z}_v)}{C'+1}.
\] 
It follows that, for some constant $b'$,
$$-a \delta_{h}(z_v,z) - C'a \delta_{h}(\bar{z}_v,z)=  -(C'+1)a\delta_h(\xi,z)+ b'.$$
Since
\[
 \frac{ \nabla\varphi(x_v)+C'\nabla\varphi(\bar{x}_v)}{C'+1} = \frac{ x_v+C'\bar{x}_v}{C'+1}  = \nabla\varphi\( \frac{ x_v+C'\bar{x}_v}{C'+1}\) ,
\]
we similarly write, for some constant $b''$,
\begin{align*}
-&a \delta_{\varphi}(x_v,x) - C'a \delta_{\varphi}(\bar{x}_v,x) 
	= -(C'+1)a\delta_\varphi\(\frac{ x_v+C'\bar{x}_v}{C'+1} ,x\)+ b''.
\end{align*}
Therefore
\[
\bar{P}(x,z)
	= -(C'+1)a \delta_\Phi\(\(\frac{ x_v+C'\bar{x}_v}{C'+1} ,\xi\),(x,z) \) + d'',
\]
for some constant $d''$.
Hence, the opening of $\bar{P}$ is $(C'+1)a>0$ and its vertex is of the form 
\[
\(\frac{x_v+C'\bar{x}_v}{C'+1}, \xi\) \quad \hbox{where}~h'(\xi) = \frac{ h'(z_v)+C'h'(\bar{z}_v)}{C'+1}.
\] 
Let $B$ be the set of these vertices and let $A$ denote the set of corresponding touching points.

Since $\bar{P}(x_2,z_2) \leq U(x_2,z_2)$, we have that
$P(x_2,z_2) - C'a \delta_{\Phi}((\bar{x}_v,\bar{z}_v),(x_2,z_2)) + d \leq U(x_2,z_2)$.
By the engulfing property, $S_{\alpha r}(x_2,z_2) \subset S_{\alpha \theta r} (\bar{x}_v,\bar{z}_v)$, 
so that $\delta_{\Phi}((\bar{x}_v,\bar{z}_v),(x_2,z_2)) < \alpha \theta r$.
Therefore, 
\begin{align*}
d 
&\leq U(x_2,z_2) - P(x_2,z_2) + C'a\delta((\bar{x}_v,\bar{z}_v),(x_2,z_2)) 
\leq Car +  C' \alpha \theta a r.
\end{align*}

Since 
$(x_2,z_2) \in S_{\alpha \theta r}(\bar{x}_v,\bar{z}_v) \subset S_{2 \alpha \theta r}(\bar{x}_v,\bar{z}_v)$, we again use the engulfing property to see that 
$S_{2\alpha \theta r}(\bar{x}_v,\bar{z}_v) \subset S_{2\alpha \theta^2 r}(x_2,z_2)$.
Suppose that $(x,z) \in Q_{\hat{K}_3R}$ is such that $\delta_{\Phi}((x_2,z_2),(x,z))\geq 2\alpha \theta^2 r$.
Then $\delta_{\Phi}((\bar{x}_v,\bar{z}_v),(x,z))\geq 2\alpha \theta r$ and
\begin{align*}
\bar{P}(x,z) 
&\leq P(x,z) - C'a ( 2\alpha \theta r)+ \(Car +  C' \alpha \theta a r\) \\
&= P(x,z) + \( C-  C'\theta \alpha \)ar
< P(x,z)
\leq U(x,z)
\end{align*}
when $C' = C'(n,\lambda,\Lambda,s)>1$ is such that $C'> C/(\theta \alpha)$. 
Hence, the contact points for $\bar{P}$ are inside $S_{2\alpha \theta^2 r}(x_2,z_2)$. That is, $A \subset S_{2\alpha \theta^2 r}(x_2,z_2)$.

Recall that $(x_2,z_2) \in \overline{S}_{\gamma r}(x_0,z_0)$. Since $\gamma  =  \alpha \theta$, we use the engulfing property to obtain
\begin{align*}
\overline{S}_{\gamma r}(x_0,z_0) 
	= \overline{S}_{ \alpha \theta r}(x_0,z_0) 
	&\subset  \overline{S}_{\alpha \theta^2  r}(x_2,z_2) \\
	&\subset {S}_{ 2\alpha \theta^2  r}(x_2,z_2)\\
	&\subset  {S}_{ 2\alpha \theta^3  r}(x_0,z_0)
	={S}_{\eta r}(x_0,z_0) 
	\subset  Q_{\eta r}(x_0,z_0).
\end{align*}
Consequently, $A \subset S_{2\alpha \theta^2 r}(x_2,z_2) \subset Q_{\eta r}(x_0,z_0)$.

We now estimate
\begin{align*}
\bar{P}(x,z)
	&\leq P(x,z)  + d\\
	&\leq a \delta_{\Phi}((x_v,z_v),(x_1,z_1)) + U(x_1,z_1) + d\\
	&\leq a \delta_{\Phi}((x_v,z_v),(x_1,z_1)) + aR + (Car + C' \alpha \theta a r)\\
	&\leq aK\(\delta_{\Phi}((\tilde{x},\tilde{z}),(x_v,z_v)) + \delta_{\Phi}((\tilde{x},\tilde{z}),(x_1,z_1))\)+ aR + (CaR + C'\alpha \theta a R)\\
	&\leq aK(\hat{K}_3R + R)+ aR + (CaR + C'\alpha \theta a R)\\
	&= \((\hat{K}_3+1)K+ 1 + C + C'\alpha \theta \)a R.
\end{align*}
If $C' = C'(n,\lambda,\Lambda,s)>1$ is sufficiently large, then
\[
\bar{P}(x,z) \leq (C'+1)aR
\]
which shows that $A \subset A_{(C'+1)a,R}$.

Since $f \geq 0$, we trivially have that
\begin{align*}
\mu_{\Phi}\bigg(B \cap \bigg\{ (x,z) &: \abs{h'(z)} \leq \frac{\norm{f^-}_{L^{\infty}(Q_{\hat{K}_3R})}}{(C'+1)a} \bigg\}\bigg) 
= \mu_{\Phi}\(B \cap \left\{ (x,z) : z=0 \right\}\) 
= 0.
\end{align*}
Therefore, by Theorem \ref{lem:ABP-super},
\begin{equation}\label{eq:lem2 equation}
\mu_{\Phi}(A_{(C'+1)a,R} \cap Q_{\eta r}(x_0,z_0))
	\geq \mu_{\Phi}(A \cap Q_{\eta r}(x_0,z_0))
	= \mu_{\Phi}(A)
	\geq c \mu_{\Phi}(B).
\end{equation}

We claim that 
\begin{equation}\label{eq:PhiB}
c\mu_{\Phi}(B)\geq c'\mu_{\Phi}(Q_r(x_0,z_0))
\end{equation}
for a positive constant $c' = c'(n,\lambda,\Lambda,s)<1$. 

For the proof of \eqref{eq:PhiB}, we first show that
\begin{equation}\label{eq:Cprime}
\mu_{\Phi}(B)  \geq  \(\frac{C'}{C'+1}\)^{n+1} \mu_{\Phi}(S_{\frac{\alpha r}{2}}(x_2,z_2)).
\end{equation}
Observe that the $B$ can be expressed as
\[
B = \left\{ (x,z) : x = \frac{x_v+C'\bar{x}_v}{C'+1},
~h'(z) = \frac{ h'(z_v)+C'h'(\bar{z}_v)}{C'+1},
~(\bar{x}_v,\bar{z}_v) \in S_{\alpha r}(x_2,z_2) \right\}.
\]
Define the sets $B_1$ and $B_2$ by
\begin{align*}
B_1&= \left\{ x = \frac{x_v+C'\bar{x}_v}{C'+1} : \bar{x}_v \in S_{\alpha r/2}(x_2)\right\} \\
B_2 &= \left\{ z = (h')^{-1}\(\frac{ h'(z_v)+C'h'(\bar{z}_v)}{C'+1}\): \bar{z}_v \in S_{\alpha r/2}(z_2)\right\}.
\end{align*}
Since
$S_{\alpha r/2}(x_2,z_2) \subset S_{\alpha r/2}(x_2)\times S_{\alpha r/2}(z_2)
\subset S_{\alpha r}(x_2,z_2)$,
we know that $B_1\times B_2 \subset B$ and 
\begin{equation}\label{eq:measureofB}
\mu_\Phi(B)\geq \mu_{\Phi}(B_1 \times B_2) = \mu_{\varphi}(B_1)\mu_h(B_2).
\end{equation}
By a change of variables, 
\begin{align*}
 \mu_{\varphi}(B_1) 
 	= \int_{B_1}  \, dx 
 	&= \(\frac{C'}{C'+1}\)^n \int_{S_{\alpha r/2}(x_2) } \, d\bar{x}_v
	=  \(\frac{C'}{C'+1}\)^n\mu_{\varphi}\(S_{\alpha r/2}(x_2)\).
\end{align*}
Notice that the set $Z_0$ given by 
\[
Z_0= \bigg\{\bar{z}_v\in\R: h'(\bar{z}_v) = -\frac{1}{C'} h'(z_v) \bigg\}
\]
is a singleton. 
Then, by using a change of variables,
\begin{align*}
& \mu_{h}(B_2)= \int_{B_2 \setminus \{z=0\} } h''(z) \, dz\\
&= \int_{S_{\alpha r/2}(z_2)\setminus  Z_0}  h''\((h')^{-1}\(\frac{ h'(z_v)+C'h'(\bar{z}_v)}{C'+1}\)\) \partial_z(h')^{-1}\bigg|_{\frac{ h'(z_v)+C'h'(\bar{z}_v)}{C'+1}} \(\frac{C'}{C'+1}\) h''(\bar{z}_v)\,d\bar{z}_v\\
&= \frac{C'}{C'+1}\int_{S_{\alpha r/2}(z_2)\setminus  Z_0}
	  h''(\bar{z}_v) \,d\bar{z}_v= \frac{C'}{C'+1}\mu_h\(S_{\alpha r/2}(z_2)\).
\end{align*}
Combining these estimates into \eqref{eq:measureofB}, we obtain
\begin{align*}
\mu_{\Phi}(B)
	&\geq \(\frac{C'}{C'+1}\)^{n+1}
		\mu_{\varphi}\(S_{\alpha r/2}(x_2)\)
		\mu_h\(S_{\alpha r/2}(z_2)\)\\
	&\geq \(\frac{C'}{C'+1}\)^{n+1}\mu_{\Phi}\(S_{\alpha r/2}(x_2,z_2)\)
\end{align*}
and \eqref{eq:Cprime} holds.

For \eqref{eq:PhiB}, observe that, by the doubling estimate \eqref{eq:reversedoubling} for $\mu_{\Phi}$,
\[
\mu_{\Phi}(S_{\gamma \theta r}(x_2,z_2))
	\leq  K_d\(\frac{2\theta \gamma}{\alpha}\)^{n+1} \mu_{\Phi}(S_{\frac{\alpha r}{2}}(x_2,z_2))
\]
and
\[
\mu_{\Phi}(S_{(n+1)r}(x_0,z_0))
	\leq K_d \(\frac{n+1}{\gamma}\)^{n+1} \mu_{\Phi}(S_{\gamma r}(x_0,z_0)).
\]
Since $(x_2,z_2) \in \overline{S}_{\gamma r}(x_0,z_0)$, the engulfing property gives $\overline{S}_{\gamma r}(x_0,z_0) \subset \overline{S}_{\gamma \theta r}(x_2,z_2)$. 
Hence, by using \eqref{eq:Cprime} and the previous two estimates,
\begin{align*}
c\mu_{\Phi}(B)
	&\geq  c\(\frac{C'}{C'+1}\)^{n+1} \mu_{\Phi}(S_{\frac{\alpha r}{2}}(x_2,z_2))\\
	&\geq c\(\frac{C'}{C'+1}\)^{n+1}  \frac{1}{K_d}\(\frac{\alpha}{2\theta \gamma}\)^{n+1} \mu_{\Phi}(S_{\gamma \theta r}(x_2,z_2)) \\
	&\geq c\(\frac{C'}{C'+1}\)^{n+1}  \frac{1}{K_d}\(\frac{\alpha}{2\theta \gamma}\)^{n+1} \mu_{\Phi}(S_{\gamma  r}(x_0,z_0)) \\
	&\geq c\(\frac{C'}{C'+1}\)^{n+1}   \frac{1}{K_d^2}\(\frac{\alpha}{2\theta(n+1)}\)^{n+1}  \mu_{\Phi}(S_{(n+1)r}(x_0,z_0))\\
	&\geq c' \mu_{\Phi}(Q_{r}(x_0,z_0)).
\end{align*}
This completes the proof of \eqref{eq:PhiB}.

From \eqref{eq:lem2 equation} and \eqref{eq:PhiB}, the lemma follows.
\end{proof}

%%%%%%%%%%%%%%%%%%%%%%%%%%%%%%%%
\section{Covering lemma}\label{sec:lem3}
%%%%%%%%%%%%%%%%%%%%%%%%%%%%%%%%

Here, we establish the following covering lemma. 

\begin{lem}\label{lem:Dk}
Let $K_0 = K_0(n,s)>1$, $\eta = \eta (n,s)<1$ be as in \eqref{eq:K0 and eta}, 
and fix $0 < c <1$. 
Consider a cube $Q_{R/K_0} = Q_{R/K_0}(\tilde{x},\tilde{z})$.
Suppose there is a countable family of closed sets $D_k \subset \R^{n+1}$
that satisfy the following properties:
\begin{enumerate}
	\item[1)] $D_0\subset D_1\subset\dots \subset D_k \subset \dots \subset \overline{Q}_{R/K_0}$,~$D_0 \not= \varnothing$;
	\item[2)] for any $(x,z) \in \R^{n+1}$, $\rho>0$ such that
	\begin{align*}
	&Q_\rho(x,z) \subset Q_{R}(\tilde{x},\tilde{z}) , \quad Q_{\eta \rho}(x,z) \subset Q_{R/K_0}(\tilde{x},\tilde{z}),
	\quad \overline{Q}_{\rho}(x,z)  \cap D_k \not= \varnothing,
	\end{align*}
	we have
	\[
	\mu_{\Phi}(Q_{\eta \rho}(x,z) \cap D_{k+1}) \geq c \mu_{\Phi}(Q_\rho(x,z)).
	\]
\end{enumerate}
Then
\[
\mu_{\Phi}(Q_{R/K_0} \setminus D_k) \leq (1-c)^k\mu_{\Phi}(Q_{R/K_0}).
\]
\end{lem}

\begin{rem}
Observe that Lemma \ref{lem:Dk} is similar the  Calder\'on--Zygmund lemma in \cite{Cabre}.
In fact, the sets $Q_{R/K_0} \setminus D_{k+1}$ and $Q_{R/K_0} \setminus D_{k}$, the parameter $1-c$,  the Monge--Amp\`ere cubes $Q_r$ and $Q_{\eta \rho}$, 
and the Monge--Amp\`ere measure $\mu_{\Phi}$ can be seen as analogues of the sets $A$ and $B$, the parameter $\delta$, the dyadic cubes $\tilde{Q}$ and $Q$, and the Lebesgue measure of Lemma 4.2 in \cite{Cabre}, respectively.
See also \cite[Lemma 2.3]{Savin}.
\end{rem}

To prove Lemma \ref{lem:Dk}, we need the following simple consequence 
of \cite[Theorem 1.2]{CoifmanWeiss} for Monge--Amp\`ere cubes.

\begin{lem}\label{Cor:cubecovering}
Let $E \subset \R^{n+1}$ be a bounded subset. For each $(x,z) \in E$, consider a cube $Q_{r_{(x,z)}}(x,z)$
with radius $r_{(x,z)}>0$. Then there is a countable subfamily of such cubes $\{Q_{r_i}(x_i,z_i)\}_{i=1}^\infty$ such that
\[
E \subset \bigcup_{i=1}^{\infty} Q_{r_i}(x_i,z_i), \quad \hbox{with}~Q_{{r_i}/{K_0}}(x_i,z_i)~\hbox{pairwise disjoint}.
\]
\end{lem}

\begin{proof}[Proof of Lemma \ref{lem:Dk}.]
For any $(x_0,z_0) \in E := Q_{R/K_0}(\tilde{x},\tilde{z}) \setminus D_k$ and let $r$ be given by
\begin{equation}\label{eq:r}
r = r_{(x_0,z_0)} = \inf\{r_0 : Q_{r_0}(x_0,z_0)\cap D_k \not= \varnothing\}.
\end{equation}
The family $\{Q_r(x_0,z_0)\}$ covers $E$.
By Lemma \ref{Cor:cubecovering}, there is a countable collection of cubes $\{Q_{r_i}(x_i,z_i)\}_{i=1}^\infty$ such that 
$E = Q_{R/K_0} \setminus D_k \subset \bigcup_i Q_{r_i}(x_i,z_i)$, with $Q_{r_i/K_0}(x_i,z_i)$ pairwise disjoint.
Then,
\[
\mu_{\Phi}(Q_{R/K_0} \setminus D_k) 
\leq \mu_{\Phi} \( \bigcup_{i} Q_{r_i}(x_i,z_i) \cap Q_{R/K_0}\)
\leq \sum_i  \mu_{\Phi}(Q_{r_i}(x_i,z_i) \cap Q_{R/K_0}).
\]
We claim that, for any $(x_0,z_0) \in E$
and $r$ given by \eqref{eq:r},
\begin{equation}\label{eq:lem3-claim}
\mu_{\Phi}(Q_{r}(x_0,z_0) \cap Q_{R/K_0}) \leq \frac{1}{c} \mu_{\Phi}(Q_{r/K_0}(x_0,z_0) \cap D_{k+1}).
\end{equation}
Suppose for now that \eqref{eq:lem3-claim} holds. Then
\begin{align*}
\mu_{\Phi}(Q_{R/K_0} \setminus D_k) 
	&\leq \sum_i  \mu_{\Phi}(Q_{r_i}(x_i,z_i) \cap Q_{R/K_0})\\
	&\leq \sum_i \frac{1}{c} \mu_{\Phi}(Q_{r_i/K_0}(x_i,z_i) \cap D_{k+1})\\
	&= \frac{1}{c} \mu_{\Phi}\(\bigcup_i Q_{r_i/K_0}(x_i,z_i) \cap (D_{k+1}\setminus D_k)\)\\
	&\leq \frac{1}{c}\mu_{\Phi}(D_{k+1}\setminus D_k).
\end{align*}
In the second to last estimate, we used our choice of $r$ in \eqref{eq:r}.
Since
\begin{align*}
\mu_{\Phi}(Q_{R/K_0} \setminus D_{k+1}) 
	&= \mu_{\Phi}(Q_{R/K_0} \setminus D_{k}) - \mu_{\Phi}(D_{k+1} \setminus D_k)\\
	&\leq \mu_{\Phi}(Q_{R/K_0} \setminus D_{k}) - c \mu_{\Phi}(Q_{R/K_0} \setminus D_{k})\\
	&= (1-c)\mu_{\Phi}(Q_{R/K_0} \setminus D_{k}),
\end{align*}
by iteration, we finally obtain $\mu_{\Phi}(Q_{R/K_0} \setminus D_{k})\leq (1-c)^k \mu_{\Phi}(Q_{ R/K_0} )$,
and the lemma is proved.

It is left to prove \eqref{eq:lem3-claim}. 
We will present the proof for $n=1$ for which 
\[
Q_{R/K_0}(\tilde{x},\tilde{z}) = S_{R/K_0}(\tilde{x}) \times S_{R/K_0}(\tilde{z}) \subset \R^2.
\]
The more general case follows similarly and is left to the reader. 

First, we estimate $r$.
Given any point $(x,z) \in Q_{R/K_0}$ and $(x_0,z_0) \in Q_{R/K_0} \setminus D_k$ , we have
\[
\delta_\varphi (x_0,x) 
\leq K \(\delta_\varphi(\tilde{x},x_0) + \delta_{\varphi}(\tilde{x},x)\) 
< \frac{2KR}{K_0}.
\]
and, similarly, $\delta_h (z_0,z)\leq 2K R/K_0$. 
Therefore, $r < 2K R/K_0$ whenever $r$ is given by \eqref{eq:r}.

Let $(x_0,z_0) \in Q_{R/K_0} \setminus D_k$ and $r$ as in \eqref{eq:r} be fixed.

Next, let $(x,z) \in Q_r(x_0,z_0)$. 
By the quasi-triangle inequality, the choice of $K_0$ in \eqref{eq:K0 and eta}, and the estimate on $r$,
\begin{align*}
\delta_{\varphi}(\tilde{x},x)
	&\leq K \( \delta_{\varphi}(\tilde{x}, x_0)+\delta_{\varphi}(x_0,x)\)< K  \( \frac{R}{K_0}+r\) \leq R.
\end{align*}
Similarly, one can show that $\delta_{h}(\tilde{z},z) <  R$.
Therefore, we have that 
\begin{equation}\label{eq:Qr-QR}
S_r(x_0) \times S_r(z_0) = Q_r(x_0,z_0) \subset Q_R(\tilde{x},\tilde{z}) = S_R(\tilde{x}) \times S_R(\tilde{z}).
\end{equation}
We will break into cases based on how far $(\tilde{x},\tilde{z})$ is from $(x_0,z_0)$.

\medskip

\noindent
{\bf Case 1}. Suppose that $\tilde{x} \in S_{r/K_0}(x_0)$, $\tilde{z} \in S_{r/K_0}(z_0)$. 

We will show that $Q_r(x_0,z_0)$ satisfies the hypothesis 2) in the statement with $\rho=r$:
\[
Q_r(x_0,z_0) \subset Q_{R}(\tilde{x},\tilde{z}), 
\quad Q_{\eta r}(x_0,z_0) \subset Q_{R/K_0}(\tilde{x},\tilde{z}),
\quad \overline{Q}_{r}(x_0,z_0)  \cap D_k \not= \varnothing.
\]
We have already established \eqref{eq:Qr-QR}. 
By the definition of $r$, we know that $\overline{Q}_r(x_0,z_0) \cap D_k \not= \varnothing$.
Thus, it is left to show that $Q_{\eta r}(x_0,z_0) \subset Q_{R/K_0}(\tilde{x},\tilde{z})$.
Let $(x,z) \in Q_{\eta r}(x_0,z_0)$. 
By the quasi-triangle inequality and by choice of $K_0$ and $\eta$ in \eqref{eq:K0 and eta},
since $x \in S_{\eta r}(x_0)$,
\begin{align*}
\delta_\varphi (\tilde{x},x)
	&\leq K \(\delta_\varphi (x_0,\tilde{x}) + \delta_{\varphi}(x_0,x)\) < K \(\frac{r}{K_0} + \eta r\) \leq \frac{R}{K_0}.
\end{align*}
We can similarly show that 
\begin{equation}\label{eq:ztilde estimate}
\delta_h(\tilde{z},z) <  \frac{R}{K_0} \quad \hbox{since}~z \in S_{\eta r}(z_0).
\end{equation}
Hence, $Q_{\eta r}(x_0,z_0) \subset Q_{R/K_0}(\tilde{x},\tilde{z})$.

Therefore, since $\eta \leq 1/K_0$, by property 2),
we obtain the desired estimate:
\begin{align*}
 \mu_{\Phi}(Q_{r/K_0}(x_0,z_0) \cap D_{k+1}) 
 	&\geq \mu_{\Phi}(Q_{\eta r}(x_0,z_0) \cap D_{k+1})\\
	&\geq c \mu_{\Phi}(Q_r(x_0,z_0))\\
	& \geq c \mu_{\Phi}(Q_r(x_0,z_0) \cap Q_{R/K_0}(\tilde{x},\tilde{z})).
\end{align*}

\medskip

\noindent
{\bf Case 2}. Suppose that $\tilde{x} \notin S_{r/K_0}(x_0)$, $\tilde{z} \in S_{r/K_0}(z_0)$. 

It must be that $x_0< \tilde{x}$ or $\tilde{x} < x_0$. 
Without loss of generality, we assume that $x_0 < \tilde{x}$.

From \eqref{eq:Qr-QR} and \eqref{eq:ztilde estimate}, we deduce that
\[
S_r(z_0) \subset S_R(\tilde{z}), \quad S_{\eta r}(z_0) \subset S_{R/K_0}(\tilde{z}).
\]
We will find $x_1$ between $x_0$ and $\tilde{x}$ such that 
\begin{equation}\label{eq:x1-x0-tildex}
S_{r/(2K^2K_0)}(x_1) \subset S_{r/K_0}(x_0) \cap S_{R/K_0}(\tilde{x}).
\end{equation}
Let $x_1>x_0$ be such that $\delta_\varphi(x_0,x_1) = r/(2KK_0)$. 
We first show that $S_{r/(2KK_0)}(x_1)\subset S_{r/K_0}(x_0)$.
Indeed, for $x \in  S_{r/(2KK_0)}(x_1)$, we have that
\[
\delta_\varphi(x_0,x) 
	\leq K\(\delta_\varphi(x_0,x_1)  + \delta_\varphi(x_1,x) \)
	< \frac{r}{K_0}.
\]

Since
\[
\frac{r}{2KK_0} = \delta_{\varphi}(x_0,x_1) \leq K \delta_{\varphi}(x_1,x_0) \leq K^2 \delta_{\varphi}(x_0,x_1) = K^2 \frac{r}{2KK_0},
\]
we know that
\[
\frac{r}{2K^2K_0} \leq \delta_{\varphi}(x_1,x_0) \leq \frac{r}{2K_0}.
\]
Thus, $x_0 \notin S_{r/(2K^2K_0)}(x_1)$. 
Since the sections $S_{r/(2K^2K_0)}(x_1)$ and $S_{r/K_0}(x_0)$ are one-dimensional intervals, we can write them as
\begin{align*}
S_{r/(2K^2K_0)}(x_1) = (x_L, x_R) \quad &\hbox{where}~x_L < x_1 < x_R\\
S_{r/K_0}(x_0) = (x^0_L, x_R^0) \quad &\hbox{where}~x_{L}^0 < x_0 < x_{R}^0.
\end{align*}
Since $\tilde{x} \notin S_{r/K_0}(x_0)$ and $x_0 < \tilde{x}$, we know that
\[
x_{L}^0 < x_0 < x_{R}^0 < \tilde{x}.
\]
Since $x_0 < x_1$ and $S_{r/(2K^2K_0)}(x_1) \subset S_{r/K_0}(x_0)$, we have that
\[
x_0 < x_L < x_1 < x_R < x_{R}^0 < \tilde{x}.
\]
Thus, for any $x \in S_{r/(2K^2K_0)}(x_1)$, we know that $x_0 < x < \tilde{x}$.
By Lemma \ref{lem:ordering}, 
\[
\delta_{\varphi}(\tilde{x},x) 
	<  \delta_{\varphi}(\tilde{x},x_0) 
	< \frac{R}{K_0}
\]
Hence, $S_{r/(2K^2K_0)}(x_1)\subset S_{R/K_0}(\tilde{x})$ and we proved \eqref{eq:x1-x0-tildex}. 

Define
\[
\rho = \(K + \frac{1}{2K_0}\)r.
\]
Clearly $S_r(z_0) \subset S_\rho(z_0)$.
Let $x \in S_r(x_0)$. Then, 
\begin{align*}
\delta_\varphi(x_1,x)
	&\leq K\(\delta_\varphi(x_0,x_1) + \delta_\varphi(x_0,x)\)\leq K \( \frac{r}{2KK_0} + r\) = \rho.
\end{align*}
Hence, $S_r(x_0) \subset S_\rho(x_1)$. Therefore,
\begin{equation}\label{eq:x0z0-x1z0}
Q_r(x_0,z_0) = S_r(x_0) \times S_r(z_0) \subset S_\rho(x_1) \times S_\rho(z_0) = Q_\rho(x_1,z_0).
\end{equation}

Since $\overline{Q}_r(x_0,z_0) \cap D_k \not= \varnothing$, 
we know by \eqref{eq:x0z0-x1z0} that $\overline{Q}_\rho(x_1,z_0) \cap D_k \not= \varnothing$.
Next, in order to apply property 2) in the statement, we will show that $Q_\rho(x_1,z_0)$ satisfies the following:
\begin{equation}\label{eq:rho-property2}
Q_\rho(x_1,z_0) \subset Q_R(\tilde{x},\tilde{z}), 
\quad Q_{\eta \rho}(x_1,z_0) \subset Q_{R/K_0}(\tilde{x},\tilde{z}), 
\quad Q_{\eta \rho}(x_1,z_0) \subset Q_{r/K_0}(x_0,z_0).
\end{equation}
First, let us check that $Q_\rho(x_1,z_0) \subset Q_R(\tilde{x},\tilde{z})$. 
Take $(x,z) \in Q_\rho(x_1,z_0)$ and observe that
\begin{align*}
\delta_\varphi(\tilde{x},x)
	&\leq K\(\delta_\varphi(\tilde{x},x_1) + \delta_\varphi(x_1,x)\)< K\(\frac{R}{K_0} + \rho\)\leq R.
\end{align*}
We can similarly show that $\delta_h(\tilde{z},z) < R$. Hence, $Q_\rho(x_1,z_0) \subset Q_R(\tilde{x},\tilde{z})$.
Next, by the choice of $\eta$ in \eqref{eq:K0 and eta},
we know that 
\begin{equation}\label{eq:eta rho-r}
\eta \rho = \frac{r}{2K^2K_0} \leq \frac{r}{K_0}.
\end{equation}
Then, by \eqref{eq:x1-x0-tildex}, 
\begin{align*}
 Q_{\eta \rho}(x_1,z_0)  
	&= S_{r/(2K^2K_0)}(x_1) \times S_{r/(2K^2K_0)}(z_0) \\
	&\subset S_{r/K_0}(x_0) \times S_{r/K_0}(z_0)
	= Q_{r/K_0}(x_0,z_0).
\end{align*}
Lastly, since $\tilde{z} \in S_{r/K_0}(z_0)$, for $z \in S_{\eta \rho}(z_0)$, by \eqref{eq:K0 and eta},
\begin{align*}
\delta_h(\tilde{z},z)
	&\leq K \( \delta_h(z_0,\tilde{z}) + \delta_h(z_0,z)\) < K\(\frac{r}{K_0} + \eta \rho\)  \leq \frac{R}{K_0}.
\end{align*}
Therefore,  $S_{\eta \rho}(z_0) \subset S_{ R/K_0}(\tilde{z})$.
With this, \eqref{eq:eta rho-r}, and \eqref{eq:x1-x0-tildex},
we obtain
\begin{align*}
 Q_{\eta \rho}(x_1,z_0) 
	&= S_{r/(2K^2K_0)}(x_1) \times S_{\eta \rho}(z_0) \\
	&\subset S_{ R/K_0}(\tilde{x}) \times S_{ R/K_0}(\tilde{z})
	= Q_{ R/K_0}(\tilde{x},\tilde{z}).
\end{align*}

We have shown that $Q_\rho(x_1,z_0)$ satisfies the hypotheses of property 2). 
Therefore, by using \eqref{eq:rho-property2}, the conclusion of 2), and
\eqref{eq:x0z0-x1z0}, we obtain the desired estimate:
\begin{align*}
\mu_{\Phi}(Q_{r/K_0}(x_0,z_0) \cap D_{k+1}) 
	&\geq \mu_{\Phi}(Q_{\eta \rho}(x_1,z_0) \cap D_{k+1}) \\
	&\geq c \mu_{\Phi}(Q_\rho(x_1,z_0))\\
	&\geq c \mu_{\Phi}(Q_r(x_0,z_0)).
\end{align*}

\medskip

\noindent
{\bf Case 3}. Suppose that $\tilde{x} \in S_{r/K_0}(x_0)$, $\tilde{z} \not\in S_{r/K_0}(z_0)$. 

This follows exactly as in Case 2 by switching the roles of $\tilde{x}$ and $\tilde{z}$ and using $\delta_h$ in place of $\delta_\varphi$.

\medskip

\noindent
{\bf Case 4}. Suppose that $\tilde{x} \notin S_{r/K_0}(x_0)$, $\tilde{z} \notin S_{r/K_0}(z_0)$. 

This follows by combining the arguments in Case 2 and Case 3.
\end{proof}

%%%%%%%%%%%%%%%%%%%%%%%%%%%%%%%%
\section{Proof of Theorem \ref{thm:reduction3} and Theorem \ref{thm:harnack Ls}}\label{sec:main proofs}
%%%%%%%%%%%%%%%%%%%%%%%%%%%%%%%%

%%%%%%%%%%%%%%%%%%%%%%%%%%%%%%%%
\subsection{Proof of Theorem \ref{thm:reduction3}}
%%%%%%%%%%%%%%%%%%%%%%%%%%%%%%%%

We begin by sliding a paraboloid $P$ of opening $a>0$ with vertex $(\tilde{x},\tilde{z})$ from below until it touches the graph of $U$ for the first time in $Q_{\hat{K}_3R}$, say at $(x_0,z_0) \in Q_{\hat{K}_3R}$. 
Then
\[
P(x,z) = -a\delta_\Phi((\tilde{x},\tilde{z}),(x,z)) + a\delta_\Phi((\tilde{x},\tilde{z}),(x_0,z_0)) + U(x_0,z_0).
\]
If $\delta_\Phi((\tilde{x},\tilde{z}),(x_0,z_0)) > R/K_0$, then 
\begin{align*}
\frac{aR}{2K_0}
	\geq U(\tilde{x},\tilde{z})
	\geq P(\tilde{x},\tilde{z})
	= a\delta_\Phi((\tilde{x},\tilde{z}),(x_0,z_0)) + U(x_0,z_0)
	> \frac{aR}{K_0}.
\end{align*}
Hence, $(x_0,z_0) \in \overline{S}_{R/K_0} = \overline{S}_{R/K_0}(\tilde{x},\tilde{z}) \subset \overline{Q}_{R/K_0}$ and
\[
U(x_0,z_0) = P(x_0,z_0) \leq P(\tilde{x},\tilde{z}) < aR.
\] 
Thus, if $A_{a,R}$ is defined as in \eqref{eq:AaR},
\[
A_{a,R} \cap \overline{Q}_{R/K_0} \not= \varnothing.
\]

In order to apply Lemma \ref{lem:Dk}, we define the closed sets $D_k \subset \overline{Q}_{R/K_0}$ by
\[
D_k:= A_{aC^k,R} \cap \overline{Q}_{R/K_0}, \quad k \geq 0
\]
where $C = C(n,\lambda,\Lambda,s)>1$ is the constant from Lemma \ref{Lem:Cc}. If necessary, we can enlarge $C$ to guarantee that
\begin{equation}\label{eq:Clarge}
C - 2K \geq 2 \quad \hbox{and} \quad C - 2K - \frac{2K}{\theta} >0,
\end{equation}
see Remark \ref{rem:lem2}.
As a consequence of Lemma \ref{lem:narrower opening}, we have 
\[
\varnothing \not= D_0 \subset D_1 \subset D_2 \subset \dots \subset D_k \subset \dots \subset \overline{Q}_{R/K_0}.
\]
Thus, hypothesis 1) of Lemma \ref{lem:Dk} is satisfied. To check that property 2) in Lemma \ref{lem:Dk} holds,
let $(x,z) \in \R^{n+1}$, $\rho>0$ be such that
\[
Q_{\rho}(x,z) \subset Q_{R}(\tilde{x},\tilde{z}), \quad Q_{\eta \rho}(x,z) \subset Q_{R/K_0}(\tilde{x},\tilde{z}), \quad \overline{Q}_\rho(x,z) \cap D_k \not= \varnothing.
\]
By Lemma \ref{Lem:Cc}, there is a positive constant $c = c(n,\lambda,\Lambda,s)<1$ such that
\[
\mu_{\Phi}(D_{k+1} \cap Q_{\eta \rho}(x,z)) 
	=\mu_{\Phi}( A_{aC^{k+1},R} \cap Q_{\eta \rho}(x,z)) 
	\geq c \mu_{\Phi}(Q_{\eta \rho}(x,z)).
\]
Hence, property 2) is satisfied. 
It follows from Lemma \ref{lem:Dk} that
\begin{equation}\label{eq:Dk estimate}
\mu_{\Phi}(Q_{R/K_0} \setminus D_k) \leq (1-c)^k \mu_{\Phi}(Q_{R/K_0}).
\end{equation}
Also, from the definition of $A_{aC^k,R}$,
\begin{equation}\label{eq:Dkestimate-harnack}
U(x,z) \leq aRC^k \quad \hbox{for}~(x,z) \in D_k. 
\end{equation}

For $k \geq 0$, let $\rho_k = \rho_k(n,\lambda,\Lambda,s) <1$ be a sequence of positive constants, to be determined, such that $\rho_k \searrow 0$ as $k \to \infty$.
For convenience in the notation, let 
\[
\beta = \frac{1}{3K_0}.
\]

Let $k_0 = k_0(n,\lambda,\Lambda,s)>0$ be a large constant, to be determined. 

\bigskip

\noindent{\bf Claim}. \textit{Suppose that, for some $k \geq k_0$, there exists a point $(x_k,z_k) \in Q_{\beta R/(n+1)} \subset S_{\beta R} = S_{\beta R}(\tilde{x},\tilde{z})$ such that}
\[
U(x_{k},z_{k}) \geq aRC^{k+1}.
\]
\textit{Then there is a point $(x_{k+1},z_{k+1}) \in \partial S_{\rho_kR}(x_k,z_k)$ such that}
\[
U(x_{k+1},z_{k+1}) \geq aRC^{k+2}.
\]

\begin{proof}[Proof of claim.]
Suppose, by way of contradiction, that $U < aRC^{k+2}$ on $\partial S_{\rho_kR}(x_k,z_k)$. 
In the section
$$S_k = \overline{S}_{\rho_kR}(x_k,z_k)$$
we lower paraboloids of the form
\begin{equation}\label{eq:P-harnack}
P(x,z) = \frac{2aKC^{k+2}}{\rho_k} \delta_\Phi((x_v,z_v),(x,z)) + c_v, \quad (x_v,z_v) \in S_{\frac{\rho_kR}{\theta C^2}}(x_k,z_k)
\end{equation}
from above until they touch the graph of $U$ for the first time in $S_k$. Let $A$ denote the set of contact points. 
Fix a point $(x_0,z_0) \in A$ and a corresponding paraboloid $P$ as in \eqref{eq:P-harnack} that touches $U$ from above in $S_k$ at $(x_0,z_0)$. 

If necessary, slide $P$ further until it intersects $U$ at $(x_{k},z_{k})$ and let us denote this paraboloid by $\tilde{P}$. 
By Lemma \ref{lem:PintersectsU}, we can write
\begin{align*}
\tilde{P}(x,z)
		&= \frac{2aKC^{k+2}}{\rho_k} \delta_\Phi((x_v,z_v),(x,z)) 
			- \frac{2aKC^{k+2}}{\rho_k} \delta_\Phi((x_v,z_v),(x_k,z_k)) + U(x_k,z_k).
\end{align*}
Since $(x_v,z_v) \in S_{\frac{\rho_kR}{\theta C^2}}(x_k,z_k)$, by the engulfing property,
$S_{\frac{\rho_kR}{\theta C^2}}(x_k,z_k) 
\subset S_{\frac{\rho_kR}{C^2}}(x_v,z_v)$.
In particular, $\delta_\Phi((x_v,z_v),(x_k,z_k)) \leq \frac{\rho_k R}{C^2}$.
Therefore, for $(x,z) \in S_k$,
\begin{equation}\label{eq:tildeP-harnack}
\begin{aligned}
\tilde{P}(x,z)
&\geq  \frac{2aKC^{k+2}}{\rho_k} \delta_\Phi((x_v,z_v),(x,z)) 
			- \frac{2aKC^{k+2}}{\rho_k} \frac{\rho_k R}{C^2} + aRC^{k+1}\\
&\geq \frac{2aKC^{k+2}}{\rho_k} \delta_\Phi((x_v,z_v),(x,z)) + 2aRC^k,
\end{aligned}
\end{equation}
where we used \eqref{eq:Clarge}.
Therefore, 
\[
U(x_0,z_0) = P(x_0,z_0) \geq \tilde{P}(x_0,z_0) \geq 2aRC^k
\]
which shows that
\[
A \subset \{(x_0,z_0) \in \overline{S}_{\rho_kR}(x_k,z_k) : U(x_0,z_0) \geq 2aRC^k\}.
\]

We will next prove that $(x_0,z_0) \in {S}_{\rho_kR}(x_k,z_k)$; that is, the contact points in $A$ are interior points of the section ${S}_{\rho_kR}(x_k,z_k)$. 
Assume, by way of contradiction, that 
$\delta_{\Phi}((x_k,z_k),(x_0,z_0)) = \rho_kR$.
By the quasi-triangle inequality,
\begin{align*}
\rho_kR 
	&\leq K \(  \delta_{\Phi}((x_k,z_k),(x_v,z_v)) + \delta_{\Phi}((x_v,z_v),(x_0,z_0))  \)\\
	&< K \( \frac{\rho_k R}{\theta C^2}+ \delta_{\Phi}((x_v,z_v),(x_0,z_0))  \),
\end{align*}
so that
\[
 \delta_{\Phi}((x_v,z_v),(x_0,z_0)) > \rho_k R\( \frac{1}{K} - \frac{1}{\theta C^2}\).
\] 
Since $(x_0,z_0) \in S_k$, from \eqref{eq:tildeP-harnack} and \eqref{eq:Clarge}, we get
\begin{align*}
U(x_0,z_0) = P(x_0,z_0) 
&\geq \tilde{P}(x_0,z_0) \\
&\geq \frac{2aKC^{k+2}}{\rho_k} \delta_\Phi((x_v,z_v),(x_0,z_0)) + aRC^k\(C-2K\)\\
&> \frac{2aKC^{k+2}}{\rho_k} \rho_k R\( \frac{1}{K} - \frac{1}{\theta C^2}\)+aRC^k\(C-2K\)> 2aRC^{k+2},
\end{align*}
which contradicts our assumption that $U < aRC^{k+2}$ on $\partial S_{\rho_kR}(x_k,z_k)$. 
Therefore, it must be that $(x_0,z_0) \in S_{\rho_kR}(x_k,z_k)$. Consequently, 
\begin{equation}\label{eq:A2}
A \subset \{(x_0,z_0) \in {S}_{\rho_kR}(x_k,z_k) : U(x_0,z_0) \geq 2aRC^k\}.
\end{equation}

Next, we want to apply Theorem \ref{lem:ABP-sub} with Remark \ref{rem:lem1-sections} in $S_k$ with $\varepsilon_0 = 1/2$.
For this, we need to choose $k_0 = k_0(n,\lambda,\Lambda,s)$ sufficiently large to guarantee that
\begin{equation}\label{eq:12Phi}
\mu_{\Phi}\(S_{\frac{\rho_kR}{\theta C^2}}(x_k,z_k)
	 \cap \left\{ (x,z) : \abs{h'(z)} \leq \frac{\norm{f^+}_{L^{\infty}(S_k \cap \{z=0\})}}{(2aKC^{k+2}/\rho_k)}\right\}\)
	\leq \frac{1}{2} \mu_{\Phi}(S_{\frac{\rho_kR}{\theta C^2}}(x_k,z_k))
\end{equation}
for all $k \geq k_0$.
Indeed, observe that
\begin{align*}
\mu_{\Phi}&\(S_{\frac{\rho_kR}{\theta C^2}}(x_k,z_k)
	 \cap \left\{ (x,z) : \abs{h'(z)} \leq \frac{\norm{f^+}_{L^{\infty}(S_k \cap \{z=0\})}}{(2aKC^{k+2}/\rho_k)}\right\}\)\\
&\leq \mu_{\varphi}\(S_{\frac{\rho_kR}{\theta C^2}}(x_k)\)
	 \mu_h\( \left\{ z \in \R: \abs{h'(z)} \leq \frac{\norm{f}_{L^{\infty}(S_k \cap \{z=0\})}}{(2aKC^{k+2}/\rho_k)}\right\}\).
\end{align*}
Notice that
\begin{align*}
 \mu_h\( \left\{ z \in \R: \abs{h'(z)} \leq \frac{\norm{f}_{L^{\infty}(S_k \cap \{z=0\})}}{(2aKC^{k+2}/\rho_k)}\right\}\)
	&= 2 \frac{\norm{f}_{L^{\infty}(S_k \cap \{z=0\})}}{(2aKC^{k+2}/\rho_k)}\\
	&= \frac{\rho_k}{KC^{k+2}} \frac{\norm{f}_{L^{\infty}(S_k \cap \{z=0\})}}{a}\\
	&\leq \frac{\rho_k}{KC^{k+2}} \mu_h(S_R(\tilde{z})).
\end{align*}
Since $z_k \in S_{\beta R/(n+1)}(\tilde{z})$, by the engulfing property, we have $S_{\beta R/(n+1)}(\tilde{z}) \subset S_{\theta \beta R/(n+1)}(z_k)$. With this and the doubling property \eqref{eq:reversedoubling} for $\mu_h$,
\begin{align*}
\mu_h(S_R(\tilde{z}))
	&\leq K_d \(\frac{R}{\beta R/(n+1)}\)^{1} \mu_h(S_{\frac{\beta R}{n+1}}(\tilde{z}))\\
	&= K_d \(\frac{n+1}{\beta}\) \mu_h(S_{\frac{\beta R}{n+1}}(\tilde{z}))\\
	&\leq K_d \(\frac{n+1}{\beta}\) \mu_h(S_{\frac{\theta \beta R}{n+1}}(z_k))\\
	&\leq K_d \(\frac{n+1}{\beta}\)K_d\(\frac{\theta \beta R/(n+1)}{\rho_kR/(\theta C^2)} \)^{1}\mu_h(S_{\frac{\rho_kR}{\theta C^2}}(z_k))\\
	&= \frac{ K_d^2 \theta^2C^2  }{\rho_k} \mu_h(S_{\frac{\rho_kR}{\theta C^2}}(z_k)).
\end{align*}
Hence
\begin{align*}
 \mu_h\( \left\{ z \in \R: \abs{h'(z)} \leq \frac{\norm{f}_{L^{\infty}(S_k \cap \{z=0\})}}{(2aKC^{k+2}/\rho_k)}\right\}\)
 	&\leq \frac{\rho_k}{KC^{k+2}}  \frac{ K_d^2 \theta^2C^2  }{\rho_k}\mu_h(S_{\frac{\rho_kR}{\theta C^2}}(z_k))\\
	&=\frac{K_d^2\theta^{2}}{KC^{k}}  \mu_h(S_{\frac{\rho_kR}{\theta C^2}}(z_k)).
\end{align*}
This and the doubling property \eqref{eq:reversedoubling} for $\mu_{\Phi}$ give
\begin{align*}
\mu_{\Phi}&\(S_{\frac{\rho_kR}{\theta C^2}}(x_k,z_k)
	 \cap \left\{ (x,z) : \abs{h'(z)} \leq \frac{\norm{f^+}_{L^{\infty}(S_k \cap \{z=0\})}}{(2aKC^{k+2}/\rho_k)}\right\}\)\\
	&\leq\frac{K_d^2\theta^{2}}{KC^{k}}\mu_{\varphi}\(S_{\frac{\rho_kR}{\theta C^2}}(x_k)\)\mu_h\(S_{\frac{\rho_kR}{\theta C^2}}(z_k)\)\\
	&\leq\frac{K_d^2\theta^{2}}{KC^{k}}\mu_{\Phi}\(S_{\frac{2\rho_kR}{\theta C^2}}\(x_k,z_k\)\) \\
	&\leq \frac{K_d^2\theta^{2}}{KC^{k}}K_d2^{n+1}\mu_{\Phi}\(S_{\frac{\rho_kR}{\theta C^2}}\(x_k,z_k\)\).
\end{align*}
Therefore, \eqref{eq:12Phi} holds if we choose $k_0 = k_0(n,\lambda,\Lambda,s)$ large enough so that
\[
\frac{K_d^{3}\theta^{2}2^{n+1}}{KC^{k}} \leq \frac{1}{2} \quad \hbox{for all}~k \geq k_0.
\]
Hence, by Theorem \ref{lem:ABP-sub} with Remark \ref{rem:lem1-sections} for $\varepsilon_0 = 1/2$, it follows that
\begin{equation}\label{eq:mu(A)}
\mu_{\Phi}(A) \geq \frac{c}{2} \mu_{\Phi}(S_{\frac{\rho_kR}{C^2\theta}}(x_k,z_k)).
\end{equation}

Next, we will choose $\rho_k$ in order to estimate $\mu_{\Phi}(S_{\frac{\rho_kR}{C^2\theta}}(x_k,z_k))$ in \eqref{eq:mu(A)}
from below by $\mu_{\Phi} (Q_{R/K_0}(\tilde{x},\tilde{z}))$ and get
\begin{equation}\label{eq:A}
\mu_{\Phi}(A) 
	\geq 2(1-c)^k\mu_{\Phi}(Q_{R/K_0}(\tilde{x},\tilde{z}) ).
\end{equation}
In fact, since $\beta < 1/K_0$, we have that $(x_k,z_k) \in Q_{\beta R/(n+1)}(\tilde{x},\tilde{z}) 
\subset S_{\beta R}(\tilde{x},\tilde{z})
\subset S_{R/K_0}(\tilde{x},\tilde{z})$, so that, by the engulfing property,
\[
S_{ R/K_0}(\tilde{x},\tilde{z}) \subset S_{\theta R/K_0 } (x_k,z_k). 
\]
As a consequence of the doubling property \eqref{eq:reversedoubling} for $\mu_{\Phi}$,
\begin{align*}
\mu_{\Phi}(S_{\theta  R/K_0 } (x_k,z_k) )
&\leq K_d\(\frac{C^2\theta^2}{\rho_kK_0}\)^{n+1} \mu_{\Phi}(S_{\frac{\rho_kR}{C^2\theta}}(x_k,z_k))
\end{align*}
and
\begin{align*}
\mu_{\Phi}(S_{R(n+1)/K_0}(\tilde{x},\tilde{z})) 
&\leq K_d \( n+1\)^{n+1} \mu_{\Phi}(S_{R/K_0}(\tilde{x},\tilde{z})).
\end{align*}
Combining these estimates, we obtain
\begin{align*}
\mu_{\Phi}(S_{\frac{\rho_kR}{C^2\theta}}(x_k,z_k))
	&\geq K_d^{-1} \(\frac{\rho_kK_0}{C^2\theta^2 }\)^{n+1}\mu_{\Phi}(S_{\theta R /K_0} (x_k,z_k))  \\
	&\geq  K_d^{-1} \(\frac{\rho_kK_0}{C^2\theta^2 }\)^{n+1}\mu_{\Phi}(S_{ R/K_0}(\tilde{x},\tilde{z}) )\\
	&\geq  K_d^{-1} \(\frac{\rho_kK_0}{C^2\theta^2 }\)^{n+1} K_d^{-1} \(n+1\)^{-(n+1)} \mu_{\Phi}(S_{ R(n+1)/K_0}(\tilde{x},\tilde{z}) )\\
	&\geq  K_d^{-1} \(\frac{\rho_kK_0}{C^2\theta^2 }\)^{n+1} K_d^{-1} \(n+1\)^{-(n+1)}\mu_{\Phi}(Q_{R/K_0}(\tilde{x},\tilde{z}) ).
\end{align*}
If we take
\[
\rho_k = c _0(1-c)^{k/(n+1)}, \quad c_0 = \frac{C^2\theta^2(n+1)}{K_0}\(\frac{4K_d^2}{c}\)^{1/(n+1)}
\]
we arrive at \eqref{eq:A}.

We next show, by enlarging $k_0$ if necessary, that
\[
S ((x_k,z_k), \rho_kR)  = S_{\rho_kR}(x_k,z_k) \subset \subset Q_{R/K_0}(\tilde{x},\tilde{z})
\] 
so that 
\begin{equation}\label{eq:Srhok}
A = A \cap S_{\rho_kR}(x_k,z_k) = A \cap Q_{R/K_0}(\tilde{x},\tilde{z}).
\end{equation}
Let $C_0>0$ and $p>1$ be the constants in Lemma \ref{lem:Guti}. 
Since $(x_k,z_k) \in S_{\beta R}(\tilde{x},\tilde{z})$, we know by Lemma \ref{lem:Guti} with $r_1 = \beta$, $r_2 = \beta + (\rho_k/C_0)^{1/p}$, and $t=R$, that
\begin{align*}
S((x_k,z_k),\rho_{k}R)&\subset S\((\tilde{x},\tilde{z}),  \(\beta + \(\frac{\rho_{k}}{C_0}\)^{1/p}\) R\).
\end{align*}
If necessary, make $k_0 = k_0(n,\lambda,\Lambda,s)$ larger to guarantee that
\begin{equation}\label{eq:k0-sum}
\sum_{j=k_0}^\infty \(\frac{\rho_j}{C_0}\)^{1/p} 
	< \frac{1}{2K_0} - \beta.
\end{equation}
In particular, 
\[
\beta + \(\frac{\rho_{k}}{C_0}\)^{1/p} \leq \frac{1}{2K_0} \quad \hbox{for all}~k \geq k_0.
\]
Therefore,
$S_{\rho_kR}(x_k,z_k) \subset S_{R/(2K_0)}(\tilde{x},\tilde{z})\subset \subset S_{R/K_0}(\tilde{x},\tilde{z}) \subset Q_{R/K_0}(\tilde{x},\tilde{z})$,
which shows \eqref{eq:Srhok}.

By the definition of $D_k$,
\[
\{ (x,z): U(x,z) > aRC^k\} \cap \overline{Q}_{R/K_0}  \subset \overline{Q}_{R/K_0} \setminus D_k.
\]
With this, \eqref{eq:Dk estimate}, \eqref{eq:A}, \eqref{eq:Srhok}, and \eqref{eq:A2}, we estimate
\begin{align*}
\mu_{\Phi}(\{U> aRC^k\} \cap Q_{R/K_0} )
	&\leq \mu_{\Phi}(Q_{ R/K_0} \setminus D_k)\\
	&\leq (1-c)^k \mu_{\Phi}(Q_{ R/K_0})\\
	&\leq \frac{1}{2} \mu_{\Phi}(A) \\
	&= \frac{1}{2} \mu_{\Phi}(A \cap Q_{ R/K_0}) \\
	&\leq \frac{1}{2} \mu_{\Phi}(\{U \geq 2aRC^k\} \cap Q_{ R/K_0})\\
	&\leq \frac{1}{2} \mu_{\Phi}(\{U > aRC^k\} \cap Q_{ R/K_0}),
\end{align*}
which is a contradiction. This completes the proof of the claim. 
\end{proof}

We now use the claim to prove \eqref{eq:harnack-finalreduction} with $\kappa_2 = \beta/(n+1)$ and $C_H = C^{k_0+1}$. 
Suppose, by way of contradiction, that there is a point $(x_{k_0},z_{k_0}) \in Q_{\beta R/(n+1)}$ such that
\[
\sup_{Q_{\beta R/(n+1)}}U \geq U(x_{k_0},z_{k_0}) > aRC^{k_0+1}.
\]
By the claim, there is a point $(x_{k_0+1},z_{k_0+1}) \in \partial S_{\rho_{k_0}R}(x_{k_0},z_{k_0})$ such that
\[
U(x_{k_0+1},z_{k_0+1}) > aRC^{k_0+2}.
\]
Repeating this process, we can find a sequence $(x_{k+1},z_{k+1}) \in \partial S_{\rho_{k}R}(x_{k},z_{k})$ such that
\[
U(x_{k+1},z_{k+1}) > aRC^{k+2} \quad \hbox{for}~k > k_0.
\]
For all $k \geq k_0$, by Lemma \ref{lem:Guti} with 
\[
r_1 = \beta + \sum_{j=k_0}^k \(\frac{\rho_j}{C_0}\)^{1/p}, \quad
r_2 =  \beta + \sum_{j=k_0}^{k+1} \(\frac{\rho_j}{C_0}\)^{1/p}, \quad
t=R,
\]
and by \eqref{eq:k0-sum}, we obtain
\begin{align*}
 S((x_{k+1},z_{k+1}),\rho_{k+1}R)
&\subset S\((\tilde{x},\tilde{z}),  \(\beta + \sum_{j=k_0}^{k+1} \(\frac{\rho_{j}}{C_0}\)^{1/p} \) R\)\\
&\subset S\((\tilde{x},\tilde{z}),   \frac{R}{2K_0}\)
\subset Q\((\tilde{x},\tilde{z}), \frac{R}{2K_0}\).
\end{align*}
Therefore, $(x_{k+1},z_{k+1}) \in Q_{ R/(2K_0)}$ for all $k \geq k_0$.
In particular, $U$ is unbounded in $\overline{Q}_{ R/(2K_0)}$. This is a contradiction and completes the proof.\qed

%%%%%%%%%%%%%%%%%%%%%%%%%%%%%%%%
\subsection{Proof of Theorem \ref{thm:harnack Ls}}
%%%%%%%%%%%%%%%%%%%%%%%%%%%%%%%%

Let $\kappa = \kappa(n,s)<1$ and $\hat{K} = \hat{K}(n,s)>1$ be such that
\[
\kappa = \sqrt{\kappa_0} 
\quad \hbox{and} \quad
\sqrt{2\hat{K}_0}= \hat{K}
\]
where $\kappa_0$ and $\hat{K}_0$ are the constants from Theorem \ref{thm:harnack extension}.
We recall from \eqref{eq:sections-balls} that $ B_{r}(x_0) = S_{r^2/2}(x_0)$ for any $r>0$.
By taking $r = \sqrt{\kappa_0}R$,
\begin{equation} \label{eq:ball to section}
\begin{split}
B_{\kappa R}(x_0) \times \{z=0\}
	&\subset S_{\kappa_0 R^2/2}(x_0) \times S_{\kappa_0 R^2/2}(0)\subset S_{\kappa_0 R^2}(x_0,0).
\end{split}
\end{equation}
By taking $r = \sqrt{2\hat{K}_0}R$,
\begin{align*}
S_{\hat{K}_0R^2}(x_0,0)
	&\subset B_{\sqrt{2\hat{K}_0}R}(x_0)    \times S_{\hat{K}_0R^2}(0)=   B_{\hat{K}R}(x_0)    \times S_{\hat{K}_0R^2}(0)
	 \subset \subset \Omega \times \R.
\end{align*}
We also note that
\begin{equation}\label{eq:almostthere}
 S_{\hat{K}_0R^2}(x_0,0) \cap\{z=0\}
= B_{\hat{K}R}(x_0) \times \{z=0\}.
\end{equation}

Let $U$ be as in Theorem \ref{thm:Calpha extension}.
By Proposition \ref{thm:arendt} and \eqref{eq:U-defn}, it follows that $U \geq 0$. 
Let $\tilde{U}$ be the even reflection of $U$ so that $\tilde{U}$ is symmetric across $\{z=0\}$. Notice that 
$\tilde{U} \in C^2(S_{\hat{K}_0R^2}(x_0,0) \setminus \{z=0\}) \cap C(S_{\hat{K}_0R^2}(x_0,0))$, $\tilde{U}_{z+} \in C(S_{\hat{K}_0R^2}(x_0,0) \cap \{z\geq0\})$ 
 and that $\tilde{U}$ is a nonnegative solution to 
\[
\begin{cases}
a^{ij}(x) \partial_{ij}\tilde{U} + z^{2-\frac{1}{s}}\partial_{zz}\tilde{U} = 0 & \hbox{in}~S_{\hat{K}_0R^2}(x_0,0) \cap \{z\ne 0\}\\
-\partial_{z+}\tilde{U}(x,0) = f(x) & \hbox{on}~S_{\hat{K}_0R^2}(x_0,0) \cap \{z=0\}.
\end{cases}
\]
Since $U(x,0)=u(x)$, by \eqref{eq:ball to section}, Theorem \ref{thm:harnack extension},  and \eqref{eq:almostthere},
we have that
\begin{align*}
\sup_{B_{\kappa R}(x_0)} u
	&\leq \sup_{S_{\kappa_0 R^2}(x_0,0)}\tilde{U}\leq C_H \( \inf_{S_{\kappa_0 R^2}(x_0,0)} \tilde{U} + \norm{f}_{L^{\infty}(S_{\hat{K}_0 R^2}(x_0,0)\cap \{z=0\})}R^{2s}\)\\
	&\leq C_H \( \inf_{B_{\kappa R}(x_0)}u + \norm{f}_{L^{\infty}(B_{\hat{K}R}(x_0))}R^{2s}\),
	\end{align*}
which proves \eqref{eq:Ls Harnack3}.
Since $u$ is bounded,  the H\"older estimate \eqref{eq:Ls Holder} immediately follows for $R \leq \abs{x-x_0} < \hat{K}R$. Assume that $\abs{x-x_0} <R$. 
Note that ${B}_{R}(x_0)\times \{z=0\}\subset {S}_{R^2}(x_0,0) \subset S_{\hat{K}_0R^2}(x_0,0)$. 
By this, \eqref{eq:extension Holder}, and \eqref{eq:almostthere},
we have, for any $x \in B_{R}(x_0)$, that
\begin{align*}
|u(x_0) &- u(x)|= |\tilde{U}(x_0,0) - \tilde{U}(x,0)|\\
&\leq \frac{\hat{C}_1}{(\hat{K}_0R^2)^{\alpha_1}}\delta_{\Phi}((x_0,0),(x,0))^{\alpha_1}\( \sup_{S_{\hat{K}_0R^2}(x_0,0)} |\tilde{U}|+  \norm{f}_{L^{\infty}(S_{\hat{K}_0 R^2}(x_0,0)\cap \{z=0\})}R^{2s} \)\\
&\leq \frac{\hat{C}_1'}{(\hat{K}R)^{2\alpha_1}}\abs{x_0-x}^{2\alpha_1}\(\sup_{B_{\hat{K}R}(x_0)    \times S_{\hat{K}_0R^2}(0)} |\tilde{U}|+  \norm{f}_{L^{\infty}(B_{\hat{K}R}(x_0))}R^{2s}\).
\end{align*}
For each fixed $z\geq 0$, by \eqref{eq:semigroup bound},
\begin{align*}
\norm{U(\cdot,z)}_{L^{\infty}(B_{\hat{K}R}(x_0))}
	&\leq \frac{(2s)z}{4^s\Gamma(s)}\int_0^{\infty} e^{-\frac{s^2}{t}z^{\frac{1}{s}}} \norm{e^{-tL} u}_{L^{\infty}(B_{\hat{K}R}(x_0))} \,\frac{dt}{t^{1+s}}\leq M\norm{u}_{L^{\infty}(\Omega)}.
\end{align*}
Letting $\hat{C} =M\hat{C}_1'$, and $\alpha_0 = 2\alpha_1<1$,
we conclude \eqref{eq:Ls Holder}.\qed

\bigskip

\noindent\textbf{Acknowledgements.} We are grateful to Diego Maldonado for helpful discussions 
about his works \cites{Maldonado3,MaldonadoPDEs3}.
We also thank the referees for detailed comments that helped us improve the presentation of our paper. 

%%%%%%%%%%%%%%%%%%

%%%%%%%%%%%%%%%%%%


\begin{thebibliography}{10}
%%%%%%%%%%%%%%%%%%

\bibitem{Arendt} W.~Arendt and R.~M.~Sch\"atzle,
{Semigroups generated by elliptic operators in non-divergence form on $C_0(\Omega)$},
\textit{Ann. Sc. Norm. Super. Pisa Cl. Sci.}
\textbf{13} (2014) 1--18. 
\filbreak

\bibitem{Valdinoci} X.~Cabr\'e, S.~Diperro, and E.~Valdinoci,
{The Bernstein technique for integro-differential equations},
arXiv:2010.00376 (2021), 57 pp.
\filbreak

\bibitem{Cabre} L.~Caffarelli and X.~Cabr\'e,
\textit{Fully Nonlinear Elliptic Equations},
Amer. Math. Soc. Colloquium Publications \textbf{43},
American Mathematical Society, 1995.
\filbreak

\bibitem{Caffarelli-Charro} L.~Caffarelli and F.~Charro, 
{On a fractional Monge-Amp\`re operator,}
\textit{Ann. PDE}
\textbf{1} (2015), Art.~4, 47 pp.
\filbreak

\bibitem{Caffarelli-Guti} L.~A.~Caffarelli and C.~E.~Guti\'errez,
{Properties of the solutions of the linearized Monge-Amp\`ere equation},
\textit{Amer.~J.~Math.}
\textbf{119} (1997), 423--465.
\filbreak

\bibitem{Caffarelli-Stinga} L.~A.~Caffarelli and P.~R.~Stinga,
{Fractional elliptic equations, Caccioppoli estimates and regularity},
\textit{Ann. Inst. H. Poincar\'e Anal. Non Lin\'eaire}
\textbf{33} (2016), 767--807.
\filbreak

\bibitem{CoifmanWeiss} R.~R.~Coifman and G.~Weiss,
\textit{Analyse Harmonique Non-commutative Sur Certains Espaces Homog\`enes},
Lecture Notes in Mathematics \textbf{242},
Springer--Verlag, Berlin Heidelberg, 1971.
\filbreak

\bibitem{Cont} R.~Cont and P.~Tankov,
\textit{Finantial Modelling with Jump Processes},
Chapman \& Hall/CRC, 
2003.
\filbreak

\bibitem{Duvaut} G.~Duvaut and J.~L.~Lions,
\textit{Inequalities in Mechanics and Physics},
Springer Verlag, Berlin, 1976.
\filbreak

\bibitem{Forzani} L.~Forzani and D.~Maldonado,
{A mean-value inequality for nonnegative solutions to the linearized Monge--Amp\`ere equation}
\textit{Potential Anal.}
\textbf{30} (2009), 251--270.
\filbreak

\bibitem{Friedman} A.~Friedman,
\textit{Partial Differential Equations of Parabolic Type},
Prentice-Hall, Inc.,
Englewood Cliffs, N.J., 1964.
\filbreak

\bibitem{Gale} J.~E.~Gal\'e, P.~J.~Miana, and P.~R.~Stinga,
{Extension problem and fractional operators: semigroups and wave equations},
\textit{J.~Evol.~Equ.}
\textbf{13} (2013), 343--386.
\filbreak

\bibitem{GilbargTrudinger} D.~Gilbarg and N.~S.~Trudinger,
\textit{Elliptic Partial Differential Equations of Second Order},
Springer--Verlag, Berlin Heidelberg, 2001.
\filbreak

\bibitem{Grafakos} L.~Grafakos, 
\textit{Modern Fourier Analysis}, Second edition,
 Graduate Texts in Mathematics \textbf{250}. 
 Springer, New York, 2009.
 \filbreak

\bibitem{Grubb} G.~Grubb,
{Fractional Laplacians on domains, a development of H\"omander's theory of $\mu$-transmission pseudodifferential operators},
\textit{Adv. Math.}
\textbf{268} (2015), 478--528.
\filbreak

\bibitem{Grubb2} G.~Grubb,
{Regularity of spectral fractional Dirichlet and Neumann problems},
\textit{Math. Nachr.}
\textbf{289} (2016), 831--844.
\filbreak

\bibitem{Gutierrez} C.~E.~Guti\'errez,
\textit{The Monge-Amp\`ere Equation},
Progress in Nonlinear Differential Equations and Their Applications {\bf 44}, 
Birkh\"auser, 
Basel (2001).
\filbreak

\bibitem{Jacob} N.~I.~Jacob and R.~L.~Schilling,
{Some Dirichlet spaces obtained by subordinate reflected diffusions},
\textit{Rev.~Mat.~Iberoamericana}
\textbf{15} (1999), 59--91.
\filbreak

\bibitem{Stinga-Jhaveri} Y.~Jhaveri and P.~R.~Stinga,
{The obstacle problem for a fractional Monge--Amp\`ere equation},
\textit{Comm. Partial Differential Equations}
\textbf{45} (2020), 457--482.
\filbreak

\bibitem{Le} N.~Q.~Le,
{On the Harnack inequality for degenerate and singular elliptic equations with unbounded lower order terms via sliding paraboloids},
\textit{Commun.~Contemp.~Math.}
\textbf{20} (2018), 38 pp.
\filbreak

\bibitem{Maldonado2} D.~Maldonado,
{Harnack's inequality for solutions to the linearized Monge--Amp\`ere operator with lower-order terms},
\textit{J.~Differential Equations}
\textbf{256} (2014), 1987--2022.
\filbreak

\bibitem{Maldonado3} D.~Maldonado,
{On certain degenerate and singular elliptic PDEs I: Nondivergence form operators with unbounded drifts and applications to subelliptic equations},
\textit{J.~Differential Equations}
\textbf{264} (2018), 624--678.
\filbreak

\bibitem{MaldonadoPDEs3} D.~Maldonado,
{On certain degenerate and singular elliptic PDEs III: Nondivergence form operators and $RH_{\infty}$-weights},
\textit{J.~Differential Equations}
\textbf{280} (2021), 805--840.
\filbreak


\bibitem{Maldonado-CVPDE} D.~Maldonado,
{The Monge--Amp\`ere quasi-metric structure admits a Sobolev inequality},
\textit{Math. Res. Lett.}
\textbf{20} (2013), 527--536.
\filbreak

\bibitem{Maldonado2017} D.~Maldonado,
{$W^{1,p}_{\varphi}$-estimates for Green's functions on the linearized Monge--Amp\`ere operator},
\textit{Manuscripta Math.}
\textbf{152} (2017), 539--554.
\filbreak

\bibitem{Maldonado} D.~Maldonado and P.~R.~Stinga,
{Harnack inequality for the fractional nonlocal linearized Monge-Amp\`ere equation},
\textit{Calc.~Var.~Partial Differential Equations}
\textbf{56} (2017), 56--103.
\filbreak

\bibitem{Pazy} A.~Pazy,
\textit{Semigroups of Linear Operators and Applications to Partial Differential Equations},
{Applied Mathematical Sciences} \textbf{44},
Springer--Verlag, New York, 1983.
\filbreak

\bibitem{Savin-Liouville} O.~Savin,
{A Liouville theorem for solutions to the linearized Monge--Amp\`ere equation},
\textit{Discrete Contin.~Dyn.~Syst.}
\textbf{28} (2010), 865--873.
\filbreak

\bibitem{Savin} O.~Savin,
{Small perturbation solutions for elliptic equations},
\textit{Comm.~Partial Differential Equations.}
\textbf{32} (2007) 557--578.
\filbreak

\bibitem{Seeley} R.~T.~Seeley,
{Norms and domains of the complex powers $A_B^z$},
\textit{Amer. J. Math.}
\textbf{93} (1971), 299--309.
\filbreak

\bibitem{Song} R. Song and Z. Vondra\v{c}ek,
{Potential theory of subordinate killed Brownian motion in a domain},
\textit{Probab. Theory Relat. Fields}
\textbf{125} (2003), 578--592.
\filbreak

\bibitem{Stinga} P.~R.~Stinga,
{User's guide to the fractional Laplacian and the method of semigroups},
in: \textit{Handbook of Fractional Calculus with Applications}
\textbf{2}, 235--265,
De Gruyter, Berlin, 2019.
\filbreak

\bibitem{Stinga-Torrea} P.~R.~Stinga and J.~L.~Torrea,
{Extension problem and Harnack's inequality for some fractional operators},
\textit{Comm. Partial Differential Equations}
\textbf{35} (2010), 2092--2122.
\filbreak

\bibitem{Vaughan} M.~Vaughan,
{Analysis of nonlocal equations: One-sided weighted fractional Sobolev spaces and Harnack inequality for fractional nondivergence form elliptic equations}, 
Thesis (Ph.D.)--Iowa State University (2020), 148 pp.
\textit{ProQuest LLC}.
\filbreak

\bibitem{Yosida} K. Yosida,
\textit{Functional Analysis},
Reprint of the sixth (1980) edition,
Classics in Mathematics,
Springer--Verlag, Berlin, 1995.
\filbreak


%%%%%%%%%%%%%%%%%%
\end{thebibliography}
\end{document}